\documentclass[11pt,letterpaper]{amsart}
\usepackage{geometry}
 \usepackage[linkcolor=green!35!blue,colorlinks=purple!68!black]{hyperref} 
\hypersetup{colorlinks=true, citecolor=purple!68!black, filecolor=Maroon,
 linkcolor=green!35!blue, urlcolor=purple!68!black}

\usepackage{graphicx} 
 \usepackage{amsthm,amsfonts,amssymb}
\usepackage{euscript}
\usepackage{mathrsfs}
\usepackage{hyperref} 
\newtheorem{theorem}{Theorem}[section]
\newtheorem{lemma}{Lemma}[section]
\newtheorem{definition}{Definition}[section]
\newtheorem{remark}{Remark}[section]
\newtheorem{corollary}{Corollary}[section]
\newtheorem{notate}{Notation}[section]
\newtheorem{assumption}{Assumption}[section]

\newtheorem{conjecture}{Conjecture}[section]
\newtheorem{example}{Example}[section]
\numberwithin{equation}{section}
\newcommand{\RS}{R\mathrm{Sol}}

\newcommand{\D}{\mathcal{D}}
\newcommand{\AD}{\mathcal{A}[\mathcal{D}]}
\newcommand{\M}{\EuScript{M}}
\newcommand{\EQ}{\mathcal{E}}
\newcommand{\A}{\EuScript{A}}
\newcommand{\F}{\mathcal{F}}
\newcommand{\RHom}{R\mathcal{H}om}
\newcommand{\Sol}{\mathsf{Sol}}
\usepackage{tikz-cd}
\tikzset{%
    symbol/.style={%
        draw=none,
        every to/.append style={%
            edge node={node [sloped, allow upside down, auto=false]{$#1$}}}
    }
}

\title[Index Theorems and Torsion Invariants]{Microlocal index theorems and analytic torsion invariants in the geometric theory of partial
differential equations}
\author{J. Kryczka}
\email{jkryczka@math.bas.bg}
\address{Institute of Mathematics and Informatics, Bulgarian Academy of Sciences}

\author{V. Rubtsov}
\email{volodya@univ-angers.fr}
\address{University of Angers, CNRS, LAREMA, SFR MATHSTIC and IITP RAS}

\author{A. Sheshmani}
\email{artan@mit.edu}
\address{Massachusetts Institute of Technology (MIT), IAiFi Institute, Cambridge, USA
\\
Beijing Institute of Mathematical Sciences and Applications, Beijing, China}

\author{S-T. Yau}
\email{styau@tsinghua.edu.cn}
\address{Yau Mathematical Sciences Center, Tsinghua University, Haidian District, Beijing, China}

\date{March 2026}

\begin{document}
\begin{abstract}
We develop a microlocal and derived-geometric framework for index theory and analytic torsion of nonlinear PDEs. By integrating Spencer hypercohomology, microlocal sheaf theory, and factorization algebras, we establish new connections between classical index theorems, BCOV invariants of Calabi–Yau manifolds, and the geometry of configuration spaces.  We prove sheaf-theoretic index formulas for families of formally integrable PDEs, a microlocal index theorem for D-algebras generalizing Atiyah–Singer, and a mixed-type index theorem via microlocal stratification. We construct Ray–Singer analytic torsion for involutive systems and interpret the BCOV invariant via the Spencer torsion of the de Rham system. A categorical trace interpretation leads to a virtual index theory for derived moduli spaces of solutions. Finally, we extend the theory to configuration spaces using factorization algebras, with applications to renormalization and QFT.  These results unify geometric perspectives on PDEs, torsion invariants, and moduli theory, with implications for mirror symmetry, quantum fields, and Calabi–Yau degenerations.
\end{abstract}
\subjclass[2021]{(Primary) 14A20, 14A30, 14F10, 35A20, 35A27, 58A20, (Secondary) 14G22, 14F08, 18N60, 32C35, 32C38}
\keywords{}
\maketitle

\tableofcontents

\section{Introduction}
The Atiyah–Singer index theorem \cite{AtiyahSinger1963} possesses a profound interpretation through supersymmetric and topological quantum mechanics \cite{Witten1982Constraints}. In this framework, the index is realized as a topological invariant encoded in the geometry of the loop space. A proposal by Witten \cite{WittenLoop} replaces the base circle with a two-dimensional torus. This indicates a higher-dimensional refinement, moreover suggesting the development of such a topological index theory for Dirac operators on loop spaces.

Invariants for families of Dirac operators have been computed by Bismut-Freed \cite{BismutFreed1986EllipticFamiliesII}, based on a more general procedure that constructs metrics and unitary connections on determinant bundles $\lambda$ associated with families of first-order elliptic differential operators \cite{BismutFreed1986EllipticFamiliesI}. This generalizes seminal work of Quillen \cite{Quillen1,Quillen2}, for the case of a family of $\overline{\partial}$-operator over a Riemann surface. For the infinite-dimensional setting, we refer to \cite{Bismut1986FamiliesDirac}.  

A central role is played by the Ray-Singer analytic torsion invariant \cite{RaySinger1971}, and its interaction with canonical invariants associated to Calabi-Yau manifolds.
Indeed, Bershadsky–Cecotti–Ooguri–Vafa \cite{BCOV1994} remark there is a real valued invariant associated to each Calabi-Yau $n$-fold $X$ called the BCOV-invariant, which is denoted by $\tau_{\mathrm{BCOV}}(X).$ It depends only on the complex structure of the mirror Calabi-Yau, and is built from Ray–Singer holomorphic analytic torsions (recalled in Sect.~\ref{ssec: AnalyticTorsion}).
Importantly, since it is expected (cf. \cite[p.~137]{Kontsevich1994HAMS}) that birational Calabi–Yau manifolds
have the same $B$-models, it is naturally conjectured that if $X,X'$ are birational CY-manifolds, then $\tau_{\mathrm{BCOV}}(X)=\tau_{\mathrm{BCOV}}(X')$ (cf. \cite[Conj.~B]{eriksson2018bcov}).

There are few known strategies to approach the
BCOV predictions in mirror symmetry, and one common way to do so consists in seeing the BCOV invariant as a function on a
moduli space of Calabi–Yau varieties and exploiting the  holomorphic anomaly equation. However, these moduli spaces are in general non-compact, and this differential equation only determines $\tau_{\mathrm{BCOV}}$
at most up to a\footnote{Generally non-constant.} pluriharmonic function. To fix this indeterminacy\footnote{Known as the \emph{holomorphic ambiguity}.} it is essential to know the limiting behaviour of $\tau_{\mathrm{BCOV}}$ as one approaches the boundary of the moduli space. In \cite[Theorem A,B]{eriksson2018bcov}, a general answer is given by identifying an explicit topological expression for such boundary conditions for one-parameter normal crossings degenerations by studying the geometry of the special fiber and the limiting mixed
Hodge structures associated with the hypercohomology and determinant line of the relative logarithmic de Rham hypercohomology associated to a projective degeneration $f:X\to S$ i.e.
a projective morphism of complex manifolds
with Calabi–Yau $n$-fold fibers outside the origin, and such that the special fiber $X_0:=\sum_i D_i$ is a
simple normal crossings divisor. Note that in this case, denoting by $\Omega_{f}^p(\mathrm{log})$,  the sheaf of relative differential $p$-forms with logarithmic poles along the special fiber $X_0$, then higher direct image sheaves $\mathbf{R}^qf_*\Omega_f^p(\mathrm{log})$ are locally free \cite[Thm.~2.11]{Steenbrink1977} (see also \cite{Steenbrink1975}, \cite{Schmid1973} and the extensive works \cite{Yoshikawa2004,Yoshikawa2007,Yoshikawa2015},\cite{FangLuYoshikawa2008}, and \cite{Zucker1979}).

Treating linear PDEs and their functorialities via $\mathcal{D}$-module theory \cite{Kashiwara1970}, where $\mathcal{D}_X$ denotes the sheaf of differential operators on a smooth variety $X$, as well as in the general context of  sheaves on manifolds \cite{KS90}, and within the more general non-linear setting of $D$-geometry \cite{BD04}\cite{Pau14}, elaborated further in \cite{KSY23,KSY2}, the first and second author prove \cite[Thm.~5.12]{KR} that relative Spencer hypercohomologies of symbolic systems, which provide natural cohomological resolutions for (linearizations of) overdetermined systems of non-linear PDEs \cite{Goldschmidt1967,Sp,Sp2,Q}, are Verdier dual (in the sense of \cite[Sect.~3.1]{KS90}) to relative de Rham hypercohomology, and may be computed via holomorphic Lie hypercohomology \cite{B,BR}.
Letting $T_{X/S}$ denote the holomorphic relative tangent bundle, the universal relative Spencer complex for a morphism $f:X\to S$ of relative dimension $d$ is
\begin{equation}
\label{eqn: RelSp}
\mathcal{S}p_{X/S}^{\bullet}:=\mathcal{D}_{X/S}\otimes_{f^{-1}\mathcal{O}_S}(\Omega_{X/S}^{\bullet})^{\vee},
\end{equation}
as a complex concentrated in degrees $[-d,0]$. 
The \emph{relative Spencer hypercohomology} is then the hypercohomology of the complex $\mathbf{R}f_*\mathcal{S}p_{X/S}^{\bullet}(\mathcal{D})$, denoted by
    $\mathscr{H}_{Sp}^{\bullet}(X/S):=\mathbb{H}^{\bullet}\big(S,\mathbf{R}f_*\mathcal{S}p_{X/S}^{\bullet}(\mathcal{D}_{X/S})\big),$ which for coefficients in $\mathcal{D}_{X/S}$-modules $\M$, is denoted by
\begin{equation}
\label{eqn: CoeffHyper}
\mathscr{H}_{Sp}^{\bullet}(X/S;\mathcal{M}):=\mathbb{H}^{\bullet}(S,\mathbf{R}f_*\mathcal{S}p_{X/S}^{\bullet}(\M)\big).
\end{equation}
Its extension to algebras, with applications to formally integrable PDEs is given in \cite[Sect.~5.3]{KR}. See \eqref{eqn: FiberwiseRelSpener} below.

Spencer cohomology for the de Rham, or Cauchy-Riemann operator recovers Hodge bundles, and Dolbeault cohomology, respectively, therefore classical results of the second author and V. Lychagin \cite{LR} (see also \cite{L} and \cite{LZ}), indicate possible generalizations of results of \cite{eriksson2018bcov}, to more general systems of non-linear PDEs, as they are treated naturally via the jet-bundle formalism \cite{KLV}. 
While the classical literature on such a geometric approach to PDEs (to our knowledge) does not treat logarithmic jet spaces, a $D$-module approach does exist, and thus a careful synthesis via $D$-geometry can be obtained, roughly speaking by studying (free) commutative algebra objects over the sub-sheaf $\mathcal{D}_{(X,D)}$ of $\mathcal{D}_X,$ given via the universal enveloping algebroid of $\Theta_X(-\mathrm{log}D),$ with standard order filtration and associated graded $\mathrm{Gr}^F(\mathcal{D}_{(X,D)})\simeq Sym_{\mathcal{O}_X}^*(\Theta_{X}(-\mathrm{log}D)).$

As indicated in \cite[Sect.~6]{KR}, one may then generalize the results to consider log-Spencer resolutions\footnote{Note this is augmented to $\M$ via the map $P\otimes m\mapsto P(m).$} of left $\D_{(X,D)}$-modules $\EuScript{M}$, given similarly to \eqref{eqn: RelSp}, by
\begin{eqnarray}
\label{LogSpencer}
\mathrm{Sp}^{\bullet}(\mathrm{log}D)(\M):=0&\rightarrow& \mathcal{D}_{(X,D)}\otimes_{\mathcal{O}_X}\bigwedge^n\Theta_{X}(-\mathrm{log}D)\otimes_{\mathcal{O}_X}\M\xrightarrow{\epsilon}\cdots \nonumber
\\
&\rightarrow& \mathcal{D}_{(X,D)}\otimes\Theta_X(-\mathrm{log}D)\otimes\M\rightarrow \mathcal{D}_{(X,D)}\otimes \M\to \M.
\end{eqnarray}
The differentials $\epsilon_{p}$ in \eqref{LogSpencer} are defined by mapping $P\otimes (\theta_1\wedge\cdots\wedge \theta_p)\otimes m$ for $\theta_i\in \Theta_X(-\mathrm{log}D)$ and $m\in \M$ to
\begin{align*}
&\sum_{i=1}^p(-1)^{i-1}P\theta_i\otimes (\theta_1\wedge\cdots\wedge\hat{\theta}_i\wedge\cdots\theta_p)\otimes m-\sum_{i-1}^p(-1)^{i-1}P\otimes (\theta_1\wedge\cdots\wedge\hat{\theta}_i\wedge\cdots\theta_p)\otimes \theta_im
\\
&+\sum_{1\leq i<j\leq p}(-1)^{i+j}P\otimes \big([\theta_i,\theta_j]\wedge\theta_1\wedge\cdots\wedge\hat{\theta}_i\wedge\cdots\wedge\hat{\theta}_j\wedge\cdots\theta_p)\otimes m.
\end{align*}
\begin{remark}
In the bounded derived category of (left) $\D_{(X,D)}$-modules, one has
$\mathbb{R}\mathcal{H}om_{\mathcal{D}_{(X,D)}}(\mathcal{O}_X,\M)\simeq \Omega_{X}^{\bullet}(\mathrm{log}D)(\M),$ and for further details about logarithmic Spencer complexes, we refer to \cite{ES}.
\end{remark}

Combining these approaches, there indicates an opportunity to find new applications to treat (singular) moduli spaces of non-linear PDEs appearing in gauge theory and complex geometry, and to obtain new methods for studying BCOV predictions in mirror symmetry, from the point of view of the geometric theory of PDEs.

\subsection{Overview and Main Results}
Let $f:X\to S$ be a smooth morphism of relative dimension $d$ of projective schemes over $\mathbb{C}$, a holomorphic vector bundle $E$ over $X/S.$ Consider the relative infinite jet bundle $J_{X/S}^{\infty} E := \varprojlim_k J_{X/S}^{k} E,$
where $J^k_{X/S} E$ denotes the $k$-th relative jet bundle along fibers of $f$. We assume $f$ is \emph{relatively non-characteristic}, in the sense of \cite{KS90}, for all NPDEs under consideration (see \cite[Rmk.~3.8,Prop.~A.1]{KR} for details).

Consider a family of formally integrable partial differential equations $Z_{X/S},$
\begin{equation}
    \label{FamilyPDE}
Z_{X/S}:=\big\{Z_{X_s}^{(\ell)}\hookrightarrow J_{X_s}^{\ell}(E_s)\big\}_{s\in S},\quad E_s:=E|_{X_s},
\end{equation}
so that involutivity and formal integrability hold fiberwise. Here, $Z^{(\ell)}$ denote the $\ell$-th prolongation of the PDE.
Denote the sheaf of relative differential operators by $\mathcal{D}_{X/S}$ with order filtration, $\mathcal{D}_{X/S}^{\le k},k\geq 0.$
Recall a \emph{$\mathcal{D}_{X/S}$-scheme} is a scheme $Z$ over $X$ together with a compatible action of the sheaf $\mathcal{D}_{X/S}$. Precise definitions may be found in \cite[Sect.~3.1,3.4]{Kry2026}.

Generalizing \eqref{eqn: RelSp}, by appropriately taking coefficients in \eqref{FamilyPDE}, gives the
\emph{fiber-wise (relative) Spencer complex},
\begin{equation}
    \label{eqn: FiberwiseRelSpener}
0\rightarrow Z_{X_s}^{(\ell)}\xrightarrow{D_{s}}Z_{X_s}^{(\ell-1)}\otimes \Omega_{X/S}^1\rightarrow\cdots\rightarrow Z_{X_s}^{(\ell-r)}\otimes\Omega_{X/S}^d\rightarrow 0.
\end{equation}
Denote by $\mathscr{H}_D^{\bullet}(Z_{X/S})$ the (stable) Spencer cohomologies of  \eqref{eqn: FiberwiseRelSpener}. They are naturally graded $H_{DR}^*(X;\mathbb{C})$-modules (e.g. \cite{LZ}). Stable relative Spencer cohomology is an important tool to study obstructions to solvability for overdetermined systems of non-linear PDEs \cite{Goldschmidt1967}, from a more general sheaf-theoretic perspective familiar to $D$-module theory (see \cite{Sa},\cite{Sch} for further references), from which micro-local index theorems can be obtained \cite{SS94a,SS94b}.
\begin{remark}
    As an example, given a determined linear system $P(f)=0,$ defined by a section $P\in D_X$ with corresponding coherent $D_X$-module $M_P$, one has that 
$\mathcal{H}^0(M_P)\simeq Ker(P), \mathcal{H}^1(M_P)\simeq Coker(P),$
while $\mathcal{H}^i(M_P)=0,i>1.$
\end{remark}
A more general sheaf-theoretic treatment for non-linear PDEs is given via the Spencer hypercohomology spectral sequence \cite[Thm.~5.12, Prop.~A.1]{KR}, as recalled in Subsect.~\ref{ssec: Hyp of geom str}.

This permits the efficient computation of global topological indices via a kind of Grothendieck-Riemann-Roch type formula for families of NPDEs \eqref{FamilyPDE}. In particular, when the fibration $f:X\to S$ is non-characteristic for the $\mathcal{D}_{X/S}$-scheme associated with $Z_{X/S},$ whose cohomologies are finite-dimensional $\mathbb{C}$-vector spaces, there is a well-defined Euler-characteristic. The latter is guaranteed, for example, if $Z_{X/S}$ corresponds to a fiber-wise elliptic system (see \cite{SS94a,SS94b} for the linear case).
The important case of pure hyperbolic systems is treated elsewhere, but for a novel definition of \emph{mixed-type systems} (Definition \ref{MixedDef}) we can prove a microlocal index theorem. In order to state it, we must first fix some notation.

\subsubsection{Mixed-type microlocal index.}
To this end, denote by $K^\bullet$ the filtration induced by the order of differential operators and by $\mathcal{S}p_{X/S}^\bullet(\mathcal{M})$ the relative Spencer complex. The \emph{topological index} of $\mathcal{M}$, a coherent $\D_{X/S}$-module corresponding to a formally integrable differential system (i.e., a $\D_{X/S}$-ideal $\mathcal{I}_{X/S} \subset \mathcal{O}(J_{X/S}^\infty E)$) is the class in $K_0(S)$ defined by
\begin{equation}
\label{eqn:DHilbIndexK0}
\mathrm{Ind}_D(f,\mathcal{M}) := \sum_{i \ge 0} (-1)^i \big[\mathbf{R}^i f_* \mathcal{S}p_{X/S}^\bullet(\mathcal{M})\big] \in K_0(S),
\end{equation}
where the right-hand side is computed via the relative Spencer hypercohomology spectral sequence, i.e.
\begin{equation}
\label{eqn:DHilbIndexSpectral}
\mathrm{Ind}_D(f,\mathcal{M}) = \sum_{p,q \ge 0} (-1)^{p+q} \big[E_1^{p,q}\big] = \sum_{q\ge 0} (-1)^q \big[\mathbf{R}^q f_* \mathcal{S}p_{X/S}^\bullet(\mathrm{gr}_K \mathcal{M})\big],
\end{equation}
with
\begin{equation}
E_1^{p,q} \simeq \bigoplus_{i+j=p} \mathcal{H}^{q-i}\Big(S, \wedge^i \Theta_S \otimes \mathbf{R}^{q-i} f_* \mathcal{S}p_{X/S}^{\bullet-i}(\mathrm{gr}_K^j \mathcal{M})\Big).
\end{equation}
We then have the following first basic result, which is of fundamental importance in this work.
\begin{lemma}
\label{GRRSpencerDelta}
    The Chern character of \eqref{eqn:DHilbIndexSpectral}, the Spencer index, is given by
\begin{equation}
\label{eqn:DHilbIndexGRR}
\mathrm{ch}\big(\mathrm{Ind}_D(f,\mathcal{M})\big) = f_* \Big( \mathrm{ch}(\mathrm{gr}_K \mathcal{M}) \cdot \mathrm{Td}(T_{X/S}) \Big) \in H^{2*}(S, \mathbb{Q}),
\end{equation}
where $\mathrm{gr}_K \mathcal{M}$ is the associated graded module with respect to the order filtration, $\mathrm{Td}(T_{X/S})$ is the Todd class of the relative tangent bundle, and $f_*$ denotes the pushforward in cohomology.
\end{lemma}
Lemma \ref{GRRSpencerDelta} provides the fundamental relation which allows us to prove a variety of microlocal index theorems in the abelian (that is, triangulated) setting of $\D$-schemes.
It is based on the fact that for a suitable class of cohomologically finite dg-modules $\M$ over the non-commutative ring $\A[\D]:=\A\otimes_{\mathcal{O}_X}\D_X,$ with $\A$ a $\D_X$-smooth commutative $\D$-algebra e.g. of algebraic infinite jets, one may consider the functor of de Rham cohomology, or dually, its functor of local solutions
$R\mathcal{S}ol_{\AD}(\M).$ 
When $\M$ is equipped with a suitable good filtration we formulate finiteness hypothesis (\ref{eqn: D-elliptic}), which leads to the following result.

\begin{theorem}
    \label{ThmA:Intro}
    Consider a $\D$-smooth $\D$-scheme $Spec_{\mathcal{D}}(\mathcal{A})$ and suppose that $\M^{\bullet}\in \AD-dgmod$ is relatively $\D$-elliptic. Then,
$$\chi_{\D}(\M^{\bullet}):=\sum_{i\geq 0}(-1)^i\mathrm{dim}_{\mathbb{C}}H^i\big(X,\mathbf{R}\mathcal{H}om_{\AD}(\M^{\bullet},\AD\otimes\omega_{X})\big),$$
is well-defined and finite-dimensional.
\end{theorem}


 We prove a version of Theorem \ref{ThmA:Intro} with supports, and we apply this to the setting of derived $\D$-spaces of solutions. In what follows $\mathrm{Ch}_{\Lambda}$, given by (\ref{eqn: MicroChernA}) is an analog of the Chern character for $\AD$-dg-modules, given in Definition \ref{defn: D-Chern character}.

\begin{theorem}
    \label{ThmB:Intro}
Let $\A$ be a derived $\D$-algebra which is $\D$-afp and eventually coconnective. Let $\mathbb{T}_{\A}^{\ell}$ denote its tangent $\D$-complex and suppose $\Lambda$ is a closed conic subset of $T^*(X,\A)$ containing $\mathsf{Char}_{\D}(\A).$ Then, the index with supports, $Ind_{\D}(\A):=\chi_{\D,\Lambda}(\mathbb{T}_{\mathcal{A}}^{\ell})$ is well-defined and for each classical solution $\varphi:X\rightarrow \mathcal{A},$ gives 
$$\mathrm{Ind}_{\D}(\A;\varphi)=\int_X \mathrm{Ch}_{\varphi^{-1}\Lambda}(\sigma(\varphi^*\mathbb{T}_{\A}))Td(X).$$
\end{theorem}

Let $C_V(S)$ denote the normal cone to $S\subset T^*Y,$ along a subset $V$ \cite[Def. 4.1.1, Prop. 4.1.2]{KS90}. Set $Char_{\D,V}^1$ to be the $1$-microcharacteristic variety defined and studied in \cite{KSY23}. The solution-wise $\D$-geometric index behaves well with respect to non-characteristic pull-back.

\begin{theorem}[Theorem \ref{thm: C}]
    \label{ThmC:Intro}
Let $f:X\rightarrow Y$ be a morphism of complex analytic manifolds and consider the associated microlocal correspondence \emph{(\ref{eqn: Microlocal correspondence})}. Let $\A_Y$ be a compact dg-$\D$-algebra on $Y$ and suppose that 
\begin{equation}
\label{eqn: Condition1}
\mathring{\pi}_X^{-1}(p)\cap C_{T_Y^*Y}\big(\mathsf{Char}_{T_Y^*Y}^1(\varphi^*\mathbb{T}_{\mathcal{B}})\big)=\emptyset,
\end{equation}
for all $p\in T_Y^*Y$ with $\mathring{\pi}_X:\mathring{T}_X^*Y\times_Y T_Y^*Y\rightarrow T_Y^*Y$ the natural projection and where $\mathring{T}_X^*Y$ denotes the removal of the zero-section. Then for every solution $\varphi$, we have 
$$f_{d*}f_{\pi}^*\sigma_{\Lambda_{Y}}^{\mathcal{D}}(\A_Y;\varphi)=\sigma_{f_{\mu}^{-1}\Lambda_Y}^{\mathcal{D}}(f^*\A_Y,\gamma(\varphi)),$$
holds in $K_{f_{\mu}^{-1}\Lambda_Y}^0(T^*X),$ and in particular, 
$$f^*\mathrm{Ind}_{\D}(\A_Y;\varphi)=\mathrm{Ind}_{\D}(f^*\A_Y;\gamma(\varphi)).$$
\end{theorem}

In Subsect. \ref{ssec: Relation to AS}, we show as an application to $\D$-elliptic $\AD$-modules $\M$, these results specialize to the case of a free $\D$-algebra and recover the classical Atiyah-Singer index \cite{AtiyahSinger1963}.

\begin{theorem}[Theorem \ref{ASinger}]
    \label{CorIntro}
Let $M$ be a smooth real analytic manifold with complexification $X,$ and consider a $\mathcal{D}$-algebra $\A=Sym^*(\M^{\bullet})$ and suppose its $\mathcal{D}_X$-space of solutions $\Sol_X(\mathcal{A})$ satisfies at each point,
\begin{equation}
\label{eqn: Elliptic condition 1}
\big(Char_{\mathcal{D}}(\mathcal{\A})\times \mathsf{Spec}_{X_{DR}}(\mathcal{\A})\big)\cap T_M^*X\subset \mathsf{Spec}_{X_{DR}}(\mathcal{\A})\underset{X}{\times} M\times T_X^*X.
\end{equation}
Then, letting $0_M:M\hookrightarrow T_M^*X$ denote the zero-section and $j:T_M^*X\hookrightarrow T^*X$ the canonical embedding, one has that
$$\chi_{\D}(\A)=\int_M 0_M^*\big[j^*\mathrm{ch}\sigma_{Char(\M)}(\M)\big]\cup td_M(TM^{\mathbb{C}}).$$
\end{theorem}
As an application, we compute the Spencer-index relative to a compact elliptic boundary $\partial X\hookrightarrow X.$
\begin{lemma}[Lemma  \ref{prop: BoundaryIndex}]
    \label{BoundaryIntro}
    We have 
$\mathrm{Ind}^{rel}(\mathcal{B}_{X/S})=\mathrm{Ind}(\mathcal{B}_{X/S})-\mathrm{Ind}^{\partial}(\mathcal{B}|_{\partial X}).$
\end{lemma}

The mixed hyperbolic/elliptic index is 
\begin{theorem}[Theorem \ref{theorem: Mixed index}]
    Suppose $Z\to X_{DR}$ is mixed-type with respect to $\Omega$. Suppose $\gamma\subset T_NL$ is an open convex cone such that $\overline{\gamma}$ contains the zero-section $N$ and $g_{N\pi}^{-1}(\lambda)=g_{Nd}^{-1}(\gamma^{\circ a}).$
    Then, the hypercohomologies 
    $$H^i\big(R\Gamma(M;R\mathrm{Sol}_{X}(T_{Z,\phi},\mathscr{B}_{M,\lambda}))\big),$$
    are finite-dimensional, thus the Euler-characteristic is well-defined and coincides with the index,
    $$\mathrm{Ind}^{hyp}_{(M,\lambda)}(Z,\phi)=\int_{T^*X} \mu\mathrm{Eu}(T_{Z,\phi})\cup \mu \mathrm{Eu}\big(\Gamma_{\gamma}(\mathscr{B}_{NL})\big).$$
\end{theorem}

\subsection{Elliptic Lie equations and torsion invariants}
Throughout this paper, we apply the general sheaf-theoretic and global derived geometric methods to compute index theorems in several examples, following \cite{LR}. 
For an \emph{elliptic geometric structure} e.g. an elliptic formally integrable PDE imposed in the Atiyah algebroid of a holomorphic vector bundle $E,$ we have 
\begin{lemma}[Lemma \ref{prop: TopIndex}]
\label{prop: TopIndexIntro}
Given $\EQ^1\subset \mathrm{Jets}^1(DerE),$ then $\EQ_{f}^1\subset \mathrm{Jets}^1(DerE|_s),$ and recalling 
$\mathrm{Ind}(\EQ):=\sum_{i\geq 0}(-1)^i\mathrm{dim}_{\mathbb{C}}\mathcal{H}^i(\EQ),$ we have 
$$\mathrm{ch}[\chi(\EQ^1)]=(-1)^{\mathrm{dim}(X_{\infty})}\chi(S)\sum_{r=0}^{2k+\ell_0}\sum_{i=0}^{2k}
(-1)^{i+k}\int_{X_{\infty}}\mathrm{ch}(\mathcal{H}_{\mathbb{C}}^{r-i,i}(\mathfrak{g}_f)\wedge \mathrm{Td}_{\mathbb{C}}(T_{X_{\infty}}X)\frac{\mathrm{Td}_{\mathbb{C}}S}{\widetilde{e}(S)},$$
where $\ell_0=\ell_0(rankE,dimX,order\EQ),$ is the degree of involutivity, obtained via the $\delta$-Poincare lemma in stable spencer cohomology.
\end{lemma}

We obtain a formula for the Ray-Singer analytic torsion invariant \cite{RaySinger1971} associated to any elliptic and formally integrable non-linear PDE in the analytic category.

\begin{theorem}[Theorem \ref{thm: RSingerSpencer}]
\label{RSTheoremIntro}
    Given a formally integrable involutive system of PDEs, $\{Z_{\ell}\}_{\ell\geq 0}.$ Assume $(g,M)$ is Riemannian and $E$ is given a Hermitian metric inducing metrics $\{h_k\}$ on each $Z_k$, compatible with $g.$ Then there is a family of local Laplacian operators,
    $$\Delta_k^{r-i,i}:=D_{k}^{(i-1)}\circ D_{k}^{(i-1),*}+D_k^{(i),*}\circ D_k^{(i)}:\mathcal{S}_{k}^{r-i,i}\to \mathcal{S}_{k}^{r-i,i},$$
    for each $k,$ where $D_{-1}:=\emptyset$ and $D_{2n+1}:=\emptyset.$ 
    We have
    $$\zeta_{k}^{r-i,i}(s):=\sum_{\lambda \in Spec(\Delta_k^{r-i,i})\backslash\{0\}}\lambda^{-s},\mathrm{Re}(s)>n,$$
    and the regularized $\zeta$-determinant.
    $$\mathrm{log}\big(\det\Delta_k^{r-i,i}\big):=-\frac{d}{ds}\zeta_k^{r-i,i}(s)\big|_{s=0},$$
    is well-defined since $\zeta_k(s)$ is holomorphic at $s=0.$ The Ray-Singer analytic torsion is 
    $$T_{Z}^{RS}(M)=\prod_{k=0}^{2n}\prod_{r=0}^{2n+\ell_0}\prod_{i=0}^{2n}\big(\det \Delta_k^{r-i,i})^{(-1)^{i+2}i/2},$$
    with $\ell_0$ the integer from Lemma \ref{prop: TopIndexIntro}.
\end{theorem}
The cohomological geometry of non-linear PDEs \cite{Vin01} provides a formalism unifying important curvature formulas for BCOV-like invariants across a broad class of elliptic geometric PDEs.

While a full treatment lies outside the scope of the current paper, we mention that via Spencer hypercohomology index \eqref{eqn:DHilbIndexK0} and using Lemma \ref{GRRSpencerDelta}, one may follow the construction of 
\cite{GilletSoule1991}, as follows. Consider $\mathscr{H}_{Sp}^\bullet(X/S)
=
\mathbb{H}^\bullet
\big(
S,\mathbf{R}f_*\mathcal{S}p^\bullet(Z)
\big),$ and if $\mathbf{R}f_*\mathcal{S}p^\bullet(Z)$ is a perfect complex on $S$, we obtain the \emph{determinant line}
$$
\lambda_{Sp}(Z):=
\det\mathbf{R}f_*\mathcal{S}p^\bullet(Z)
=
\bigotimes_k
\left(
\det\mathscr{H}^k_{Sp}(X/S)
\right)^{(-1)^k},$$
which is endowed with an associated metric,
\begin{equation}
\label{eqn: QuillenMetric}||-||_{\mathrm{Quillen}}:=||-||_{L^2}\cdot \exp\bigg[\frac{1}{2}\sum_{k=0}^d(-1)^{k+1}k\log\det' \Delta_k^{L}\bigg].
\end{equation}
Such objects will be described in a future work, but encode several key objects scattered throughout the literature. We briefly summarize by stating the PDE, the Laplacian, the torsion invariant and Quillen norm, arising from Theorem \ref{RSTheoremIntro}.

\subsubsection{de Rham system}
Let $X$ be a compact oriented Riemannian manifold and consider the de Rham PDE: For $d_{DR}\omega=0,$ the Spencer complex is the de Rham complex so $H_{Sp}^k(d_{DR}=0)\simeq H_{DR}^k(X)$ and the corresponding Laplacians are the standard ones, $
\Delta_k = dd^*+d^*d:\Omega^k(X)\to\Omega^k(X).$ Zeta functions
$
\zeta_k(s)=\sum_{\lambda>0}\lambda^{-s},
\lambda\in\mathrm{Spec}(\Delta_k),$ give the ordinary RS analytic torsion 
$
T_X
=
\sum_{k\ge0}(-1)^k k\,\zeta_k'(0).
$ whose determinant line, and Quillen-metric are the classical ones (see e.g, \cite{Quillen1,Quillen2}):
$$
\lambda_{Sp}
=
\bigotimes_k(\det H^k_{dR}(X))^{(-1)^k},\hspace{2mm}\text{ with }
\|\cdot\|_Q
=
\|\cdot\|_{L^2}\exp(-T_X).$$

\subsubsection{Dolbeault/Cauchy--Riemann equation}
Consider a holomorphic vector bundle on a compact K\"ahler manifold $E\to X.$ Consider $\overline{\partial}s=0.$
Let $(X,J,\omega)$ be a compact K\"ahler manifold and $E\to X$ a holomorphic
vector bundle.
The Spencer complex coincides with the Dolbeault complex and thus
$
H^q_{Sp}(X,E)=H^{0,q}(X,E).$ The corresponding Laplacians are given by $
\Delta_{\bar\partial,q}
=
\bar\partial\bar\partial^*
+
\bar\partial^*\bar\partial,$ with $
\zeta_q(s)
=
\sum_{\lambda>0}\lambda^{-s},
\lambda\in\mathrm{Spec}(\Delta_{\bar\partial,q}).$
The Ray-Singer torsion \cite{RaySinger1971}, determinant and metric are the classical Quillen metric on Dolbeault cohomology i.e 
$
T_E^0
=
\sum_{q\ge0}(-1)^q q\,\zeta_q'(0),$ such that 

$
\lambda_{Sp}
=
\bigotimes_q(\det H^{0,q}(X,E))^{(-1)^q},$ and $
\|\cdot\|_Q
=
\|\cdot\|_{L^2}\exp(-T_E^0).$

\begin{remark}
Let $(X,\mathcal F)$ be a smooth foliation with compact leaves. Let $d_{\mathcal{F}}$ be the leaf-wise (transversal) de Rham differential and consider $d_{\mathcal F}\omega=0.$ Spencer cohomology coincides with the leaf-wise de Rham cohomology and one has a leaf-wise local Laplacian; however, $\Delta_{\mathcal F}
=
d_{\mathcal F}d_{\mathcal F}^*
+
d_{\mathcal F}^*d_{\mathcal F},$ is defined with additional extra structure e.g. the leaves are K\"ahler, otherwise the formal adjoint $d_{\mathcal{F}}^*$ does not exist. Note moreover that if leaves are non-compact one must work with the von Neumann
determinant and $L^2$--torsion.
\end{remark}

\subsubsection{Invariants}
We can compute BCOV-type invariants associated to PDEs over Calabi-Yau $n$-folds $X$, since the
the Spencer cohomologies of the PDE $d_{DR}f=0$ are isomorphic to Hodge bundles, it is tempting ask: given
a Calabi-Yau $n$-fold $(X,\omega)$, does one have that 
$\tau_{\mathrm{Sp}}(X)=\tau_{BCOV}(X)?$
If so, this would be intriguing as it shows that BCOV theory is encoded by Spencer theory for the de Rham PDE and in turn, suggests a new class of invariants can be defined and studied for other formally integrable and involutive non-linear PDEs.

\subsection{Global index via categorical traces}
The above results are part of a larger class of \emph{secondary non-linear trace formulas} \cite{BenZviNadler2013SecondaryTraces}\cite{BenZviNadler2013NonlinearTraces}, arising by combining the theory of traces in homotopical algebra with sheaf-theory and derived algebraic geometry, which we discuss briefly in the final Subsection \ref{sec: Virtual}.

Assume the linearization complex, given by the $\D$-geometric tangent complex $\mathbb{T}_{\EQ}$,is represented by a homotopy cokernel of a morphism $\mathsf{P}$ in the $\infty$-category of ind-coherent complexes on $X_{DR}$. For a solution $\varphi$, its linearization is given by pull-back $\varphi^*\mathbb{T}_{\EQ}$, which is a $\D$-module on $X.$ 
We are interested in computing the following quantity.
\begin{definition}
\label{DFredIndex}
    \emph{The $\mathcal{D}$-Fredholm index} of $\EQ$ (along $\varphi$) is the Euler
characteristic of its linearization complex
$Index_{\varphi}(\EQ):=\chi\big(a_*(\varphi^*\mathbb{T}_{\EQ}^{\ell})\big),$
where $a:X\rightarrow \mathrm{Spec}(k)$ is the canonical map.
\end{definition}
Similarly, letting $\RS_X(\EQ)$ denote its pre-stack of solutions, the $\mathcal{D}$\emph{-geometric Euler-Poincaré} index (along $\varphi)$ is the corresponding index for the pre-stack $\RS_X(\EQ).$

The following example serves to describe the above introduced objects for the PDE whose solutions are of the form $\varphi=(P,A)$, with $P$ a principle $G$-bundle with flat connection $A.$ 

\begin{example}
Given an algebraic group $G$ with $C$ a smooth projective curve, consider the stack $[*/G]$ and the $C_{dR}$-prestack $Z:=q_*([*/G]\times C)$, given by Weil-restriction. Then $\RS(Z)$ is just the stack of $G$-bundles on $C$.
Considering also $\RS_C(pt/G\times C_{dR}),$ a $T$-parametrized solution, with $T=Spec(R)$ an affine dg-scheme, 
is a map
$\varphi_T:Spec(R)\rightarrow \RS_C(pt/G\times C_{dR}),$ specifying a principal $G$-bundle with flat connection $(P,A)_T$ on $T\times C.$
Thus,
\begin{align}
    \mathbb{T}_{(P,A)}\RS_C(pt/G\times C_{dR})&= R\Gamma(T\times C_{dR},\mathfrak{g}_P)[1],d_{T})\nonumber
    \\
    &=R\Gamma_{dR}(T\times C;\mathfrak{g}_P)[1]\nonumber
    \\
    &= \big(\Omega^{\bullet}(T\times C;\mathfrak{g}_P)[1],d_{A}+d_T\big).
\end{align}
The differential $d_T$ is that coming from the dg-$R$-module determined by $P.$ In this case, the ($T$-parameterized) index 
$Index_{(P,A)}(Z)_T$ is given by
$$\sum_i(-1)^i\mathrm{dim}_{\mathbb{C}}H_{dR}^{i}(T\times C;\mathfrak{g}_P).$$
\end{example}
With the above assumptions, we may interpret $\mathsf{P}$ as a linear differential operator of a finite order, while the induced morphism $\overline{a}_*(\mathsf{P})$ on (compactly) supported global sections should be interpreted as a linear differential operator which is Fredholm (in the usual sense). The induced morphism $Gr(\mathsf{P})$, defined as a morphism of crystals over the Hodge stack of $X$, is interpreted as the associated symbol \cite{KSY23}. 

These notions were motivated by their analogs in the underived setting of linear equations e.g. $\mathcal{D}$-modules $\mathcal{M}$ and their microlocalizations,  for which one associates a certain microlocal index 
$\mu eu(\mathcal{M}),$ \cite{SS94a,SS94b} 
and micro-local Hochschild class $hh_{\mathcal{E}}(\mathcal{M})$ \cite{KashiwaraSchapira2014}.

Hochschild homology of a dualizable object in the context of higher algebra
provides a framework for realizing Euler-characteristics, Chern characters, characters of group representatsion and more, via functorial properties. Geometric setting of correspondence categories, one realizes them with derived analogs of loop-spaces and fixed point loci. Universal nonlinear versions of Grothendieck-Riemann-Roch
 theorems, Atiyah-Bott-Lefschetz trace formulas, and Frobenius-Weyl character formulas are obtained this way. We now recast the above results as a trace formula in the context of homotopical algebra and derived geometry.

To this end, consider 
$$HH_*\big(\mathsf{IndCoh}(Y)\big)\simeq \Gamma(\mathcal{L}Y,\omega_Y)\simeq \omega(\mathcal{L}Y),$$
be the Hochschild homology. Let $M\in \mathsf{IndCoh}(Y)$ be dualizable. The \emph{categorical character} is the Hochschild trace map, i.e. the image of the identity endomorphism under
$$\mathrm{Tr}:\mathrm{End}_{\mathsf{IndCoh}(Y)}(M)\to HH_*\big(\mathsf{IndCoh}(Y)\big).$$
We apply this to the global setting of not necessarily affine $X_{DR}$-spaces $Z\in \mathrm{PreStk}_{/X_{DR}}$ admitting deformation theory whose derived moduli stack of solutions $R\mathrm{Sol}(Z)$ is locally almost of finite-type (laft), by putting $Y:=X_{DR}$, viewed as a (derived) ind-inf scheme. We compute the index in Defn.~\ref{DFredIndex}.
\begin{theorem}
    Let $Z\to X_{DR}$ be $D$-afp and bounded. Then, for each solution $u:*\to \RS(Z),$ we have a categorical character $[T_{Z/X_{DR},u}]\in HH_*(D_X-\mathsf{Mod}(X)),$ and 
    $p_*[T_{Z/X_{DR}},u]\simeq \int_{\mathcal{L}p}\big[T_{Z/X,u}]\big]\in HH_*(D-\mathsf{Mod}(pt))\simeq HH_*(\mathsf{Vect}_k),$
    for $p:X\to pt,$ with $p_*$ the de Rham push-forward.
\end{theorem}

In Lemma \ref{prop: BC-micro} we prove a base-change result for non-linear microlocalizations which allows us to prove Lemma \ref{prop: BC-solution}, establishing homotopy-invariance of linearization complexes along solutions as they vary along smooth morphisms of complex manifolds. In Lemma \ref{Chern of decomposition} we compute the micro-local Chern character of the decomposition, $E_T\boxtimes M_X$.

\subsection{Motivation and relation to other works}
The purpose of the current  paper is twofold. 
First, we provide a sheaf-theoretic description of index theorems for Spencer hypercohomology via the geometric theory of non-linear PDEs based on the jet-bundle formalism. This makes central use of microlocal sheaf-theory on manifolds \cite{KS90}.
Various analogs of the Atiyah-Singer index, appearing in the form of \emph{algebraic} index theorems due to Fedosov \cite{Fedosov1996DQ} and Nest–Tsygan \cite{NestTsygan1995AlgebraicIndex}, are related to the classical index theorem through deformation quantization \cite{FormalAnalytic} and non-commutative geometry. This fact demonstrates the index is fundamentally quantization-invariant.

Building on this, a mathematically rigorous bridge between the algebraic index theorem and topological quantum mechanics was constructed in \cite{BVIndex} using the Batalin–Vilkovisky formalism. This program was extended to two-dimensional chiral theories via effective BV quantization \cite{Li2023VertexQME}. The result produces a chiral algebraic index theorem on the torus, using the formalism of \cite{BD04}. We also mention the recent developments, expanding on \cite{FG12}, appearing in \cite{GuiWangWilliams2025Jouanolou}),
which extend chiral algebras to higher-dimensions.

Secondly, we extend and interpret our constructions within a moduli theoretic setting. 

In a follow-up (however, essentially independent) work, we generalize our results to the case of \emph{derived} non-linear PDEs, modeled by sub-crystals $\iota:Z\hookrightarrow q^*q_*E$ of prestacks relative to $X_{DR}$ in the sense of \cite{KSY23,KSY2} (see also \cite{Kry2026}), which model the formal geometry of infinitely prolonged systems of overdetermined non-linear PDEs as differentially structured moduli problem in (formal) derived algebraic geometry.

We then compute the so-called \emph{virtual index theorems} for the corresponding derived moduli spaces of solutions $R\mathrm{Sol}(Z)$.

Within this perspective we use the methods of (formal) shifted symplectic geometry, the theory of categorified traces, and the geometry of derived loop spaces. As observed in \cite[Sect.7]{KSY23}, the derived linearization of a formally integrable non-linear PDE lives naturally over the derived loop stack $\mathscr{L}X$, thus canonically produces a categorified index class, formulated in terms of higher $S^1$-equivariant traces over $\mathscr{L}X.$

This class refines the classical numerical index in much the same way that categorified Donaldson–Thomas invariants refine numerical DT invariants. In this sense, classical Atiyah-Singer type index theorems play the role of “DT classes,” while the categorified index supplied by derived loop geometry plays the role of a “categorified DT invariant” in the realm of PDEs.

This perspective leads naturally to a Grothendieck–Riemann–Roch–type formula over derived loop spaces, expressing the categorified index as the secondary trace of the derived linearization. The construction is homotopy-invariant and dimension-independent, and therefore provides a geometric and deformation-invariant model for Witten’s loop-space index in a rigorous setting. We expect that this framework will ultimately clarify the two-dimensional, and possibly higher-dimensional, refinements anticipated in Witten’s proposals.

\subsection{Context and Motivation}
A second main direction of the present article concerns the microlocal and regularity-theoretic properties of configuration spaces of points. Following \cite[Rmk.~1.2, Thm.~1.5]{KSY2}, we compute higher elliptic regularity for such configuration spaces, extending classical Cauchy problems to factorization spaces on Lorentzian manifolds, with the intended application the global Cauchy problem in General Relativity.

Configuration spaces play a central role in both mathematics and physics, providing a natural geometric setting for subvarieties defined by diagonal conditions in cotangent bundles, such as those of the form, 
\begin{equation}
\label{eqn: Diagonals}
\{(x_i,\xi_i,y_i,\eta_i)| x_i=y_i,\xi_i=\eta_i\},
\end{equation}
for microlocal coodinates $(x;\xi)\in T^*X.$ 

In the context of PDE geometry, they are used to study the propagation of singularities for microfunction solutions
$f=f(x_1,\xi_1,\ldots,x_n,\xi_n),$
and microlocal parametrices on globally hyperbolic spaces. This is achieved by working on complements of the characteristic variety $S:=\mathrm{Char}(P)\subset T^*X$ of a differential operator $P$ i.e, microfunctions on $n$-fold product manifolds $X^n\backslash S^n.$

Mathematically, this is realized by employing the analytic derived Ran space $\mathrm{Ran}(X)$ and microlocal algebraic analysis to study factorization algebras in $D$-modules and their non-linear analogues. Recall that a \emph{factorization $D$-module} \cite{BD04}, is a family $\{M_{X^I}:i\in \mathrm{fSet}\}$ of $D_{X^I}$-modules such that for each $\alpha:I\rightarrow J,$ there are homotopy-coherent system of equivalences, $\Delta(\alpha)^*M_{X^I}\xrightarrow{\simeq} M_{X^J},$ and for each $I=I_1\sqcup I_2,$ we have 
\begin{equation}
\label{eqn: Factorization property}
M_{X^I}\xrightarrow{\simeq} M_{X^{I_1}}\boxtimes M_{X^{I_2}}.
\end{equation}
Microlocal analysis over $\mathrm{Ran}_X$ introduced in \cite{Pau14}, provides a formulation of the causal treatment (in the sense of Epstein-Glaser) of the quantization problem for factorization algebras on Lorentzian manifolds via purely (derived) geometric and categorical language.
Working in this level of generality avoids heavy functional-analytic machinery and instead relies on derived geometry and microlocal methods, extending new methods to approach the general factorization quantization problem of Costello–Gwilliam \cite{CG16,CG16b}. 
We emphasize that the latter approach, originally formulated in the $C^{\infty}$-setting, is suitable for elliptic PDEs on Euclidean manifolds, and works perturbatively i.e. around a fixed background classical solution. 

Since many physically significant systems, such as those arising in General Relativity, are governed by non-linear hyperbolic PDEs or more general mixed-type systems, new methods are needed (cf. \cite{JS16, BMS23}). The methods of algebraic microlocal analysis (e.g., \cite{SS94a, SS94b}) discussed here provide a step towards addressing this gap, providing tools to study:
\begin{itemize}
    \item (\textit{Elliptic regularity}): Roughly speaking, elliptic regularity for a $\D$-module $\mathcal{M}_{\mathsf{P}}$ with underlying (elliptic) operator $\mathsf{P}$ is an equivalence
\begin{equation}
    \label{eqn: Elliptic Regularity}
\mathbb{R}Sol_{X}(\mathcal{M}_{\mathsf{P}},\mathscr{A}_{M})\simeq \mathbb{R}Sol_{X}(\mathcal{M}_{\mathsf{P}},\mathscr{B}_M),
\end{equation}
solutions sheaves valued in analytic functions $\mathcal{A}_M$ on a manifold $M$ with complexification $X$, and solutions valued in Sato's sheaf of hyperfunctions $\mathscr{B}_M$ (\cite{Sato}). Equivalence (\ref{eqn: Elliptic Regularity}) means $\mathsf{P}$ induces an isomorphism 
$\mathsf{P}:\mathscr{B}_M/\mathscr{A}_M\xrightarrow{\simeq} \mathscr{B}_M/\mathscr{A}_M,$
with inverse $\mathsf{H}_{\mathsf{P}}.$ Microlocally, this means microfunction solutions 
annihilate outside of the zero-section of $T_M^*X,$ i.e.
$$\mathbb{R}Sol_{\mathcal{E}}(\mathcal{M}_{\mathsf{P}},\mathscr{C}_M)\simeq 0.$$

For non-elliptic (e.g., hyperbolic) operators, this equivalence fails, but a microlocal reformulation remains viable in both elliptic and hyperbolic contexts.

\item (\textit{Micro-hyperbolicity}):
By viewing $\mathsf{H}_{\mathsf{P}}$ not as an operator acting on a space of distributions, but rather as a micro-differential operator given by a kernel on $X\times X$, supported on the diagonal $\Delta_X.$ 
One then inverts $\mathsf{P}$ outside $\mathrm{Char}(\mathsf{P}),$ providing a microdifferential operator $\mathsf{H}\in \mathcal{E}_X.$ An analog of regularity (\ref{eqn: Elliptic Regularity}), is given an isomorphism 
$$\mathbf{R}\mathcal{H}om_{\mathcal{E}_X}(\mathcal{M}_{\mathsf{P}},\mathcal{C}_M)|_{\mathrm{Char}(\mathcal{M}_{\mathsf{P}})^c}\simeq 0,$$
reflecting that the microlocalization of $\mathcal{M}_{\mathsf{P}}$ has support in $S.$
\end{itemize}

In physical terms, the microlocal kernel $H$ serves as the propagator. It appears in the $\star$-products of quantum field theory. In this picture, they are obtain as pullbacks to the diagonals (\ref{eqn: Diagonals}) of expressions like 
$$g(y_1,\eta_1,\ldots,y_n,\eta_n)\mathsf{H}(x_1,\xi_1)\cdots \mathsf{H}(x_n,\xi_n)f(x_1,\xi_1,\cdots,x_n,\xi_n).$$
However, pull-back (and products of distributions) are not defined unless one imposes conditions on their wave front sets\footnote{This leads to the infamous UV problem in quantum field theory.}. Causal renormalization schemes address this by extending algebra-valued distributions from the complement of the thin diagonal $\Delta_n\subset X^n$, to all $X^n.$ 

The possible extensions are governed by a microlocal scaling degree measuring singularity of the distribution transversal to the submanifold $\Delta_n.$ In this framework, $\mathrm{Ran}(X)$ emerges as the natural geometric setting for organizing these extensions and the associated renormalization.

\subsubsection{Relation to other works}
The Mukai pairing \cite{Caldararu2003I,CaldararuWillerton2010} on $HH_*(X)$ can be written in terms of the HKR isomorphism, 
$\mathrm{hkr}:HH_*(X)\simeq \bigoplus_i H^{i-*}(X,\Omega_X^i),$ based on the commuting diagram in hypercohomologies
\[
\begin{tikzcd}
 \mathbb{H}^{-\bullet}(X,\Delta^*\Delta_*\mathcal{O}_X)\arrow[d]\arrow[r,"D"] &\mathbb{H}^{-\bullet}(X,\Delta^!\Delta_*\omega_X)
 \\
 \bigoplus_{i}H^{i-\bullet}(X,\Omega_X^i)\arrow[r,"\iota"] & \bigoplus_i H^{i-\bullet}(X,\Omega_X^i)\arrow[u,"\mathrm{hkr}"],
\end{tikzcd}
\]
wheren $\Delta:X\to X\times X$ is the diagonal, $\omega_X:=\Omega_X^{\dim X}[\dim X],$ and $D$ is the map in hyperbcohomology indules by $\mathrm{td}:\Delta^*\Delta_*\mathcal{O}_X\to \Delta^!\Delta_*\omega_X.$
Kashiwara-Schapira define Hochschild-classes to various classes of sheaves on $X$ (c.f. \cite[Definition 4.1.3]{KSDQMods}).
This was later used, along with deformation quantization arguments to confirm a conjecture of Schapira-Schneiders concerning the micro-local Euler-class of elliptic pairs \cite{SS94a,SS94b}. It is a consequence of a ``non-commutative'' Riemann-Roch-theorem (eg, \cite{Ramadoss2008}).
One may also associated to an holomorphic bundle $E\to X$, with a global differential operator $P\in \Gamma(X,\mathcal{D}_X(E,E)),$ acting on sections of $E$ a Leftschetz number (supertrace), 
$$L:\Gamma(X,\mathcal{D}_X(E,E))\to \mathbb{C},\hspace{1mm} P\mapsto L(P):=\sum_{j=0}^n(-1)^j\mathrm{Tr}H^j(P),$$
where $H^j(P)$ is the image of $P$ under the algebra map $H^j:\Gamma(X,\mathcal{D}_X(E,E))\to End(H^j(X,E)),$ where $H^j(X,E)$ is the (finite-dimensional) sheaf cohomology. When $P$ is just the identity, this gives the holomorphic Euler characteristic of $E$, expressed through the Riemann–Roch–Hirzebruch theorem as the integral over $X$ of a
 characteristic class. A generalization of this formula was studied in \cite{Felder}, and classical analogs for Spencer complexes appear in \cite{LR}.

A Grothendieck-Riemann-Roch formual for filtered coherent crystals along derived foliations is given in \cite[Theorem 6.1.1]{TVGRR}.
Additionally, letting $H_{\F}(X,E)$ be the foliated/transversal cohomology, then \cite[Corollary 6.1.2]{TVGRR} establishes a Hirzeburch-Riemann-Roch formula for crystals along the foliation $\F$.

\subsection{Organization}
The article is structured as follows:
\begin{itemize}
\item 
\textit{Section~\ref{sec: MicrolocalAnalysis}}: we recall the basic tools from microlocal sheaf theory. The main reference is \cite{KS90}.
\item \textit{Section~\ref{sec: PDEGeometry}}: we recall basic tools of $D$-geometry and homotopical $D$-geometry. We summarize results of \cite{KSY23,KSY2} we will be using.
In Subsect.~\ref{ssec: PDEGeometryIndexes} we prove several of the main theorems.
The main applications e.g. to manifolds with boundary, are found in Subsect.~\ref{ssec: PDEGeometryApplications}.

\item 
\textit{Section~\ref{sec: Analytic Torsion of NPDEs}}: we discuss analytic torsion invariants for formally integrable involutive systems of non-linear PDEs.
\item \textit{Section~\ref{sec: MixedType}}: we propose a definition of mixed-type PDE using microlocal sheaf theory. In Subsect.~\ref{ssec: MixedTypeIndex} we prove an index theorem for mixed-type systems.

\item \textit{Section~\ref{sec: ConfigurationSpaces}}: we extend the results to the configuration space of points.

\item \textit{Section~\ref{sec: GlobalAbstract}}: we compute the $D$-index and conclude by phrasing the moduli-theoretic index for derived spaces of flat sections $R\mathrm{Sol}(Z)$ of $D$-spaces $Z\to X_{DR}$ in the context of categorified enumerative geometry. We have in mind future relations to categorified Donaldson-Thomas invariants, Seiberg-Witten invariants and the virtual dimension of the moduli stacks of coherent sheaves and more generally, rigidified perfect complexes on smooth Calabi-Yau varieties.  
\end{itemize}
\subsection*{Acknowledgments}
The first author is supported by the Simons Foundation, grant SFI-MPS-T-Institutes-00007697, and the Ministry of Education and Science of the Republic of Bulgaria, grant DO1-239/10.12.2024. J.K would like to thank Prof. Kashiwara and Prof. Schapira for an invitation to visit Kyoto University in April 2025, and for helpful discussions.
The third author is supported by grants from Beijing Institute of Mathematical Sciences and Applications (BIMSA),
the Beijing NSF BJNSF-IS24005,
and the China National Science Foundation (NSFC) NSFC-RFIS program W2432008.
He would also like to thank China's National Program of Overseas High Level Talent for generous support.
Finally, he would like to thank NSF AI Institute for Artificial Intelligence and Fundamental Interactions at
Massachusetts Institute of Technology (MIT) which is funded by the US NSF grant under Cooperative Agreement
PHY-2019786.

\section{Microlocal analysis}
\label{sec: MicrolocalAnalysis}
We recall basic microlocal geometry, following \cite{KS90,KSDQMods}. The results are well-known, but we recall the characterization of ellipticity \cite{SS94a,SS94b} and hyperbolicity \cite{JS16} of $\D$-modules (linear PDEs) and extend this to $\D$-algebras (non-linear PDEs). We propose a novel mathematical definition of \emph{mixed-type} system via microlocal stratifications of the derived characteristic variety.

Here and throughout $f:X\to S$ will denoted a smooth projective morphism of smooth proper schemes over $\mathbb{C}.$ In some places, where it will be specified, it will represent a relative analytic submersion over $S$. 
Thus, it is smooth and we have the short exact sequence 
$$0\rightarrow X\times_YT_{Y/S}^*\rightarrow T_{X/S}^*\rightarrow T_{X/Y}^*\rightarrow 0,$$ of vector bundles on $X$. 
For example, in the absolute case when $S$ is a point, by the \emph{microlocal-correspondence} associated with $f$, we mean
\begin{equation}
    \label{eqn: Microlocal correspondence}
    \begin{tikzcd}
    & \arrow[dl,"f_d"] X\times_y T^*Y\arrow[dr,"f_{\pi}"] & 
    \\ 
    T^*X & & T^*Y.
    \end{tikzcd}
\end{equation}

Our conventions and notations follow that in \cite{KS90} for specific details e.g. of non-characteristic morphisms, and the notion of Whitney cone $C_{T_M^*X}(\Lambda)$ along the conormal bundle $T_M^*X$ of a real manifold $M$ with complexification $X$, associated to a close conic subset $\Lambda\subset T^*X.$ See e.g. \cite[Definition 4.1.1, Proposition 4.1.2]{KS90}.

\subsubsection{A microlocal diagram and ellipticity}
\label{ssec: A microlocal diagram and ellipticity}
Let $f:N\rightarrow M$ be a morphism of real analytic manifolds for which $Y=N_{\mathbb{C}},X=M_{\mathbb{C}}$ and $f_{\mathbb{C}}:Y\rightarrow X$ are complexifications. Let $\pi_M:T_M^*X\rightarrow M$ be the conormal bundle to $M$ in $X$. The microlocal correspondence (\ref{eqn: Microlocal correspondence}) for $f$ is then
\begin{equation}
    \label{eqn: Microlocal correspondence complexified}
    T_N^*Y\xleftarrow{f_d}N\times_M T_M^*X\xrightarrow{f_{\pi}}T_M^*X.
\end{equation}
Assume that $N\subset L$ and $g:L\rightarrow X$ gives a closed embedding $g|_N:N\hookrightarrow M:$
\begin{equation}
\label{setupdiag}
\begin{tikzcd}
T^*L& \arrow[l,"g_d"] L\times_XT^*X\arrow[r,"g_\pi"] & T^*X
\\
T_N^*L\arrow[u,hook] & \arrow[l,"(g_N)_d"] N\times_M T_M^*X\arrow[u,hook] \arrow[r,"(g_N)_{\pi
}"] & T_M^*X\arrow[u, hook].
\end{tikzcd}
\end{equation}
Give a coherent $\D_X$-module $\M$, let $\mathrm{Char}(\M)$ be its characteristic variety, a closed $\mathbb{C}^*$-conic cositoropic subset of $T^*X.$
\begin{definition}
A coherent $\mathcal{D}_X$-module $\mathcal{M}$ is \emph{elliptic} (on $M$) if
$\mathring{T}_M^*X\cap \mathrm{Char}(\mathcal{M})=\emptyset$, or equivalently,
$T_M^*X\cap \mathrm{Char}(\mathcal{M})\subset T_X^*X.$
\end{definition}
More generally, a coherent $\D_X$-module $\M$ is said to be elliptic on $L\subset X$ if 
\begin{equation}
    \label{eqn: Elliptic condition}
\mathrm{Char}(\M)\cap T_L^*X\subset T_X^*X.
\end{equation}
\begin{remark}
\label{Remark: Elliptic to sub-man}
Given a coherent $\D_X$-module and $L\subset X$ for which (\ref{eqn: Elliptic condition}) holds, we will say that the pair $(\M,L)$ is elliptic.
\end{remark}

From the short-exact sequence
$$0\rightarrow T_M^*X\rightarrow M\times_XT^*X\rightarrow T^*M\rightarrow 0,$$
where the middle term is isomorphic to $T^*M\oplus iT^*M$ so $T_M^*X\simeq i T^*M,$ an operator $P\in \D_X$ is \emph{elliptic} if $\sigma(P)(x;i\xi)\neq 0$ for $\xi\neq 0.$ 
Let $\mathbf{R}\mathcal{S}ol_{X}(\M):=\mathbf{R}\mathcal{H}om_{\mathcal{D}_X}(\M,\mathcal{O}_X),$ be the functor of derived holomorphic solutions, $\mathscr{A}_M$ the sheaf of complex valued real analytic functions on $M$, that is, $\mathscr{A}_M:=\mathcal{O}_X\otimes\mathbb{C}_M=\mathcal{O}_X|_{M},$ and $\mathscr{B}_M$ the sheaf of Sato's hyperfunctions. Recall the latter is a flabby sheaf concentrated in degree $0$, defined by
$\mathscr{B}_M:=\mathbf{R}\mathcal{H}om(D_X'\mathbb{C}_M,\mathcal{O}_X).$ Note that since
$$D_X'\mathbb{C}_M\simeq or_M[-\mathrm{dim}(X)]\simeq \omega_{M/X}\simeq \omega_M^{\otimes -1},$$
we equivalently have that $$\mathscr{B}_M\simeq \mathbb{R}\Gamma_M(\mathcal{O}_X)\otimes\omega_M\simeq H_M^{\mathrm{dim}(X)}(\mathcal{O}_X)\otimes or_M.$$

Then, the condition of ellipticity implies 
$g:L\rightarrow X$ is non-characteristic for $\mu supp\big(\mathbf{R}\mathcal{S}ol_X(\M)\big).$
Note that here, this sheaf enjoys the well-known elliptic regularity i.e.
$$\mathbf{R}\mathcal{S}ol_X(\M,\mathscr{A}_M)\xrightarrow{\simeq}\mathbf{R}\mathcal{S}ol_X(\M,\mathscr{B}_M),$$
is an isomorphism.
Letting $S_\phi^\bullet := \mathbf{R}\mathcal H om_{\D_X}(\mathbb T_{Z/X_{DR},\phi}, \mathcal{O}_{T\times X})$ be the solution complex of a linearized PDE along a solution $\phi:T\to \RS(Z)$, consider the following. Fix $(x_0;\xi_0)\in T^*X\setminus \{0\}$. Let $f:X\to \mathbb{R}$ be smooth with $
f(x_0)=0,\quad df(x_0)=\xi_0.$
Let $B_\epsilon(x_0)$ be a small ball of radius $\epsilon>0$ around $x_0$. The microlocal stalk of $S_\phi^\bullet$ at $(x_0;\xi_0)$
is
\[
\mu_{(x_0,\xi_0)}(S_\phi^\bullet) := \varinjlim_{\epsilon\to 0} R\Gamma_{\{f\ge 0\}\cap B_\epsilon(x_0)}(S_\phi^\bullet) \in D^b(\mathbb{C}),
\]
which is independent of the choice of $f$ with $df(x_0)=\xi_0$.
\begin{remark}
If $(x_0;\xi_0)$ lies outside $\mathbf{Ch}^{\mathrm{der}}(\mathbb T_{Z/X_{DR},\phi})$, then $\mu_{(x_0,\xi_0)}(S_\phi^\bullet)\simeq 0$.  
- If $(x_0;\xi_0)$ lies in a hyperbolic cone $\lambda$, the stalk is concentrated in degree 0 (finite-dimensional hyperfunction solutions).  
- If $(x_0;\xi_0)$ lies in an elliptic cone $\gamma$, the stalk is finite-dimensional by ellipticity.
\end{remark}
We describe the microlocal stalks of $\Sol(\mathbb{T}_{A,\phi})$ along $N,M,L$ and in the cones $\gamma,\lambda$.
To first describe microloclaization along a real submanifold, let $i_N:N\hookrightarrow X$ be the inclusion. Recall the functor of \emph{specialization} from \cite{KS90}, which yields a complex
\[
\nu_N(\mathscr{B}_M) \in D^b(\mathbf{Sh}(T_NX)),
\]
whose stalk at $(x,\xi)\in T_NX$ is
\[
\bigl(\nu_N \mathscr{B}_M\bigr)_{(x,\xi)} \simeq \varinjlim_{U\ni x} R\Gamma_{ \{ y\in U \mid \eta(y)\in \mathbb{R}_{>0}\xi\} } (U;\mathscr{B}_M),
\]
where the colimit is over neighborhoods $U$ and $\eta$ is a local coordinate lift. 
We also will need to compute stalks at points valued in (compliments of) cones.
For a convex cone $\gamma \subset T_N L$, the microlocal stalk is
\[
\bigl(R\Gamma_\gamma \nu_N \mathscr{B}_M\bigr)_{(x,\xi)}
\simeq
\varinjlim_{V \ni (x,\xi)} R\Gamma_{\gamma\cap V}(\nu_N \mathscr{B}_M|_V),
\]
with $V$ conic neighborhoods of $(x,\xi)$ in $T_N L$.  
If $\xi\notin \gamma$, the stalk vanishes.

\subsubsection{Relation to characteristic variety}

Let $\mathcal{M}^\bullet = \mathbb{T}_{A,\phi}$. Then the microlocal stalk of the hyperfunction solution complex satisfies:
\[
R\mathscr{H}om_{\mathcal D_X}(\mathcal{M}^\bullet, \Gamma_\gamma \mathscr{B}_{NL})_{(x,\xi)}
\;\simeq\;
0, \quad \text{for } \xi \notin g_{Nd}(g_{N\pi}^{-1}(\mathrm{Char}^{\mathrm{der}}(\mathcal{M}^\bullet))).
\]
Similarly for $\mathscr{B}_{M,\lambda}$, we have vanishing outside the preimage of $\lambda$.  
Thus the support of the microlocal Euler class is exactly
\[
\mathrm{supp}(\mu eu(\Gamma_\gamma \mathscr{B}_{NL})) \subset S_\gamma \cap \mathrm{Char}^{\mathrm{der}}(\mathbb{T}_{A,\phi}),
\]
\[
\mathrm{supp}(\mu eu(\mathscr{B}_{M,\lambda})) \subset \lambda \cap \mathrm{Char}^{\mathrm{der}}(\mathbb{T}_{A,\phi}).
\]

\subsubsection{Hyperbolicity}
Let $M,X$ be as in Subsect.~\ref{ssec: A microlocal diagram and ellipticity}. 
Recall the natural embedding 
$T^*M\hookrightarrow T_{T_M^*X}T^*X$, arising from the embedding $T^*M\hookrightarrow T^*T_M^*X$ and the natural isomorphism $T_{T_M^*X}T^*X\simeq T^*T_M^*X.$  
The hyperbolic characteristic variety of $\mathcal{M}$ along $M$,
\begin{equation}
\label{eqn:Hyperbolic Char}
Char_M^{hyp}(\mathcal{M}):=T^*M\cap C_{T_M^*X}\big(\mathrm{Char}(\mathcal{M})\big),
\end{equation}
then $\theta \in T^*M$ is hyperbolic for $\mathcal{M}$ when $\theta  \notin Char_M^{hyp}(\mathcal{M}).$
\begin{remark}
Say $\theta \in T_{T_M^*X}T^*X$ is microhyperbolic for a $\mathcal{D}_X$-module $\mathcal{M}$ if $\theta \notin C_{T_M^*X}\big(\mathrm{Char}(\mathcal{M})\big).$ Under $\theta \in T^*M\hookrightarrow T_{T_M^*X}T^*X$, one says $\theta$ is hyperbolic. If $\theta$ is microhyperbolic (resp. hyperbolic)
$\theta\notin \mu\mathrm{supp}\big(\mathbf{R}\mathcal{S}ol_X(\mathcal{M},\mathcal{C}_M)\big)$ 
(resp. $\theta\notin \mu\mathrm{supp}\big(\mathbf{R}\mathcal{S}ol_X(\mathcal{M},\mathcal{B})\big)$).
\end{remark}

\begin{example}
\label{Hyperbolic example 1}
If $z=x+iy\in X$ so that $M=\{z\in X|y=0\},$ and $\zeta=\xi+i\eta\in T^*X$ we get $T_M^*X=\{y=\xi=0\}.$ A direction
$p_0=(x_0;\theta_0)\in T^*M$ with $\theta_0\neq 0$ is hyperbolic if and only if there exist an open neighbourhood $U$ of $x_0$ in $M$ and open conic neighbourhood $\gamma$ of $\theta_0\in \mathbb{R}^n$ such that
$\sigma(P)(x;\theta+i\eta)\neq 0,$
for all $\eta \in \mathbb{R}^n,x\in U,\theta\in \gamma.$
\end{example}
Analagously to condition \eqref{eqn: Elliptic condition}, following the terminology of Remark \ref{Remark: Elliptic to sub-man}, a sub-manifold $N\subset M$ will be called \emph{hyperbolic} for $\mathcal{M}$ (alternatively, we say that $\M$ is hyperbolic along $N$) if 
\begin{equation}
    \label{eqn: Hyperbolic condition}
T_N^*M\cap Char_M^{hyp}(\mathcal{M})\subset T_M^*M.
\end{equation}
Any nonzero covector of $T_N^*M$ is hyperbolic for $\mathcal{M}.$ Since hyperbolicity is not an absolute notion, as opposed to ellipticity, and requires specifying directions, we use the following

\begin{definition}
    Let $\M$ be a coherent $\D_X$-module, with $X$  a complexification of $M$. For a given sub-manifold $N\subset M$ satisfying \eqref{eqn: Hyperbolic condition}, we will say that the pair $(\M,N\subset M)$ is hyperbolic.
\end{definition}

\begin{example}
\label{Hyperbolic example 2}
Let $N\subset M$ be a real submanifold with complexification $Y\subset X$ and $\mathcal{M}=\mathcal{D}_X/\mathcal{D}_X\cdot P$ a coherent $\mathcal{D}_X$-module determined by $P$ of order $\leq k.$ Assume $(x)=(x_1,x')\in M$ such that $N=\{x_1=0\}$ and suppose that $\mathrm{Char}(\mathcal{M})\cap \mathring{T}^*X \cap \mathring{T}_M^*X$ is a regular involutive submanifold of $\mathring{T}_M^*X$ and $(\M,N\subset M)$ is hyperbolic. These assumptions imply that $\sigma(P)$ is a hyperbolic polynomial with respect to $N$ with real-constant multiplicities. In other words, $\sigma(P)(x;\tau,\eta')$ has $k$ real roots $\tau$ with constant multiplicites where $x,\eta'$ are real.
\end{example}

Hyperfunction and real-analytic solutions 
propagate in hyperbolic directions \cite{KS90}.

\begin{lemma}
\label{Propagation for hyperbolic directions}
Suppose that $P\in\mathcal{D}_X$ with $\mathcal{M}:=\mathcal{D}_X/\mathcal{D}_X\cdot P\in\mathrm{Coh}(\mathcal{D}_X).$ Then

$\mu\mathrm{supp}\big(\mathbf{R}\mathcal{S}ol_X(\mathcal{M},\mathcal{A}_M)\big),\mu\mathrm{supp}\big(\mathbf{R}\mathcal{S}ol_X(\mathcal{M},\mathcal{B}_M)\big)\subset Char_M^{hyp}(\mathcal{M}).$
\end{lemma}
\begin{proof}
    Follows from the isomorphisms,
$$R\Gamma_MR\mathcal{H}om_{\D_X}(\M,\mathcal{O}_X)\simeq R\mathcal{H}om_{\D_X}(\M,R\Gamma_M\mathcal{O}_X),$$
and $R\mathcal{H}om_{\D_X}(\M,\mathcal{O}_X)|_{M}\simeq R\mathcal{H}om_{\D_X}(\M,\mathcal{O}_X|_M),$ by applying \cite[Corollary 6.4.4]{KS90}, i.e.
$$\mathrm{SS}\big(R\Gamma_M R\mathcal{H}om_{\D_X}(\M,\mathcal{O}_X)\big)\subset T^*M\cap C_{T_M^*X}\big(\mathrm{SS}(R\mathcal{H}om_{\D_X}(\M,\mathcal{O}_X)\big).$$
\end{proof}

\subsubsection{Mixed-type $D$-modules}
Letting $M,N,L,Y,X$ be as before. We come to the linear version of a mixed-type PDE. As discussed above, we specify a priori, the regions of microlocal type.
\begin{definition}
\label{defn: LinearMixedType}
Consider a coherent $\D_X$-module $\M$ with
\[
\begin{tikzcd}
& \arrow[dl,"\rho_L", labels=above left] N:=L\cap M\arrow[dr,"\rho_M"]
\\
L& & M,
\end{tikzcd}\hspace{2mm} \text{ such that } TN\simeq \rho_L^*T_L\cap \rho_M^*T_M.
\]
Then $\M$ is a $\D_X$-module of
\emph{mixed-type with domain of definition $N$} if:
    \begin{itemize}
        \item $(\M,L)$ is elliptic,
    
        \item $(\M,N)$ is hyperbolic,

        \item The map $i:Y\rightarrow X$ is non-characteristic for $\M.$
    \end{itemize}
\end{definition}
Equivalently, Definition \ref{defn: LinearMixedType} means, that on the one hand, $g:L\rightarrow X$ is non-characteristic for $\mathbb{R}\mathcal{S}ol_X(\M)$, so in particular $g_d:L\times_X T^*X\rightarrow T^*L$ is proper on $g_{\pi}^{-1}\big(\mu\mathrm{supp}(\mathbb{R}\mathcal{S}ol_X(\M))\big).$ Similarly, for the map induced from $g_N.$ On the other hand, we have that 
$(g_N)_{\pi}:N\times_MT_M^*X\rightarrow T_M^*X,$ in diagram (\ref{setupdiag}) is non-characteristic for 
$C_{T_M^*X}\big(\mu \mathrm{Supp}(\mathbb{R}\mathcal{S}ol_X(\M))\big).$
This last statement has two equivalent reformulations.

\begin{lemma}
The morphism $(g_N)_{\pi}$ is non-characteristic for $C_{T_M^*X}\big(\mu \mathrm{Supp}(\mathbb{R}\mathcal{S}ol_X(\M))\big)$ is equivalent to either one of the following equivalent statements:
\begin{itemize}
    \item $\mathring{T}_N^*M\cap C_{T_M^*X}\big(\mu\mathrm{supp}(\mathbb{R}\mathcal{S}ol_X(\M))\big)=\mathring{T}_N^*M\cap C_{T_M^*X}\big(\mathrm{Char}(\M)\big)=\emptyset,$

    \item $T_N^*X\cap \mathrm{Char}(\M)\subset T_M^*X$ and $\mathring{T}_N^*M\cap C_{\mathring{T}_M^*X}(\mathring{\mathrm{Char}}(\M)).$
\end{itemize}
\end{lemma}

\subsection{Relative characteristic varieties}

Fix a relative analytic submersion over $S$, denoted $f:X\rightarrow Y.$ In particular, it is smooth and we have the short exact sequence $0\rightarrow X\times_YT_{Y/S}^*\rightarrow T_{X/S}^*\rightarrow T_{X/Y}^*\rightarrow 0,$ of vector bundles on $X$. 
\subsubsection{Relative ellipticity} We follow \cite{SS94a,SS94b}. As $f:X\rightarrow S$ is smooth, there exists a subsheaf of algebras $\D_{X|S}$ of $\D_X$ generated by $\Theta_{X/S}$.

Letting $D^b(\D_{X|S}\otimes \D_{X|S})$ be the derived category of perfect left $\D_{X|S}$-bimodules, consider $\M^{\bullet}$ a perfect complex of $\D_X$-modules. We may choose a coherent $\D_{X|S}$-subcomplex $\M_{S}^{\bullet}$ generating $\M^{\bullet}$ and one defines the $f$-\emph{relative characteristic variety}
$$Char_{X|S}(\M^{\bullet}):=Char(\D_X\otimes_{\mathcal{D}_{X|S}}\mathcal{M}_{S}^{\bullet}).$$

If $f:X\rightarrow S$ is not necessarily smooth, decompose it according to its graph embedding
$X\xrightarrow{i}X\times S\xrightarrow{q}S,$
and put 
$$Char_{X|S}(\M^{\bullet}):=i_di_{\pi}^{-1}Char_{X\times S|S}(\mathbf{R}i_!\M^{\bullet}).$$
Since $i$ is a closed embedding we write $i_!$ for $\mathbf{R}i_!$, the $\D$-module push-forward and so explicitly,
$$Char_{X|S}(\M^{\bullet})=i_di_{\pi}^{-1}\big(Char(\mathcal{D}_{X\times S}\otimes_{\mathcal{D}_{X/S|S}}(i_!\mathcal{M}^{\bullet})_S\big),$$
where $(i_!\mathcal{M}^{\bullet})_S$ denotes the corresponding generating sub-module of the $\D_{X\times S}$-module $i_!\mathcal{M}^{\bullet}.$
\begin{definition}
    \label{defn: f-elliptic}
    \normalfont 
    Let $M$ be a real analytic manifold with complexification $X$ and consider $f:X\rightarrow S$ as above. A $\D_{X/S}^{\bullet}$-complex $\M^{\bullet}$ is \emph{$f$-elliptic} if
$Char_{X|S}(\M^{\bullet})\cap T_M^*X\subset T_X^*X.$
\end{definition}
Note that $\D_{X/S}$ is naturally an $\mathcal{O}_X$-module, and there exists a relative induction functor $Mod(\mathcal{O}_X)\rightarrow Mod(\D_{X/S}^{op}),\F\mapsto \mathrm{Ind}_{\D_{X/S}}(\F):=\F\otimes_{\mathcal{O}_X}\D_{X/S},$ with respect to the obvious right $\D_{X/S}$-module structure.
Relative dualizing sheaves and dualities are defined i.e. $D_{X/S}'(\F):=\RHom_{\mathcal{O}_X}(\F,\omega_{X/S}[d_{X/S}]),$ and there is a canonical isomorphism 
$D_{X/S}'(\F)\times_{\mathcal{O}_X}\D_{X/S}\simeq \mathbb{D}_{X/S}(\mathrm{Ind}_{\D_{X/S}}(\F)),$ with 
$\mathbb{D}_{X/S}(\M):=\RHom_{\D_{X/S}}(\M,\mathcal{K}_{X/S}),$ 
where $\mathcal{K}_{X/S}:=\omega_{X/S}[d_{X/S}]\otimes_{\mathcal{O}_X}\D_{X/S}$ is the relative dualzing $\D_{X/S}^{op}$-module.

Consider $g:S_b\rightarrow S$ and consider base-change functor 
$$(-)_b:Sch_{/S}\rightarrow Sch_{/S_b},$$
sending $(X/S)\mapsto (X\times_S S_b/S_b).$
Then, there exists $g_b:X_b:=X\times_S S_b\rightarrow X$ given by the obvious Cartesian square. In particular, we have $X_b\xrightarrow{\epsilon_{X_s}}S_b$ and $\epsilon:X\rightarrow S.$
This induces a ring-homomorphism $g_b^{-1}\D_{X/S}\rightarrow \D_{X_b/S_b},$
thus a base-change for relative $\D^{op}$-modules:
$$D^b(\D_{X/S}^{op})\rightarrow D^b(\D_{X_b/S_b}^{op}),\M\mapsto g_b^{-1}(\M)\otimes_{g_b^{-1}\D_{X/S}}^L\D_{X_b/S_b}.$$

\begin{lemma}
When $S_b=\{s\}$ for some $s\in S$, let $i_s:\{s\}\rightarrow S$ and thus if $\M$ is elliptic, then $\M_{i_s}^{\bullet}$ is relatively elliptic and 
$$\chi\big(R\Gamma(X_{s},\M_{i_s}\otimes_{\D_{X/S}}^L\mathcal{O}_{X_s})\big),$$
is a locally constant function on $S.$
\end{lemma}
\begin{proof}
Indeed, by base-change, if $I_s$ is the ideal of holomorphic functions vanishing at $s\in S,$ then 
$$R\Gamma\big(X_s,\M_{i_s}\otimes_{\D_{X_s}}^{L}\mathcal{O}_{X_s}\big)=[\mathcal{O}_{X_s}/I_s\otimes_{\mathcal{O}_S}^LR\epsilon_*(\M\otimes_{\D_{X/S}}^L\mathcal{O}_X)]_s.$$
\end{proof}
\section{Non-linear PDEs via algebraic analysis}
\label{sec: PDEGeometry}
In this section we extend the microlocal analysis of mixed-type $\D$-modules to the setting of $D$-geometry \cite{KSY23}. See also \cite{Kry2026}.
\begin{remark}[Notations]
\label{SupportNotation}
Let $X$ be smooth finite-type. Its $\infty$-category of $D$-modules is dentoed $\mathrm{Mod}(D_X),$ with natural $t$-structure. Let $\Lambda\subseteq T^*X$ be a Zariski closed conical subset and let
$\mathrm{Mod}_{\Lambda}(D_X)^{\heartsuit,c},$ denote coherent $D_X$-modules $M\in \mathrm{Mod}(D_X)^{\heartsuit,c},$ whsoe singular support $\mathrm{SS}(M)\subset \Lambda.$
Set 
$$\mathrm{Mod}_{\Lambda}(D_X)^{\heartsuit}:=\mathrm{Ind}\big(\mathrm{Mod}_{\Lambda}(D_X)^{\heartsuit,c}\big),$$
for its ind-completion. One has a subcategory $i_{\Lambda}:\mathrm{Mod}_{\Lambda}(D_X)\subset\mathrm{Mod}(D_X)$ consisting of those $D_X$-modules $M$ for which $H^n(M)\in \mathrm{Mod}_{\Lambda}(D_X)^{\heartsuit}.$
Thus $\mathrm{Mod}_{\Lambda}(D_X)$ has an induced $t$-structure, and $i_{\Lambda}$ is $t$-exact.
\end{remark}
The canonical morphism $\mathcal{O}_X\to \mathcal{D}_X$ of sheaves of algebras on $X$ induces maps in Hochschild homology as well as negative cyclic and periodic cyclic homologies, to be denoted $\iota.$
In particular, the following diagram commutes,
\[
\begin{tikzcd}
  \mathrm{HC}_0^{per}(\mathcal{O}_X)\arrow[r,"(-)^{an}\circ \iota"]\arrow[d,"\mathrm{hkr}"] & \mathrm{HC}_0^{per}(\D_X)\arrow[d,"\chi"]
  \\
  \prod_{p=-\infty}^{\infty} H^{2p}(X^{an};\mathbb{C})\arrow[r,"(-\wedge \mathrm{Td}T_X)"]& \prod_{p=-\infty}^{\infty} H^{2n-2p}(X^{an};\mathbb{C}).
\end{tikzcd}
\]
Furthermore, there is a canonical map in $\mathrm{Ho}(\mathrm{DGCat}_{\mathrm{cont}}),$ given by 
$$a_*:\mathsf{Perf}(\D_X^{op})\to \mathsf{Perf}(pt),$$
induced by $\D$-module pushforwards to the point, which is given by $a_*(\M^{an}\otimes_{\D_{X^{an}}}^L\mathcal{O}_{X^{an}}).$ The induced map on Hochschild homology is written the same,
$$a_*:HH_0(\mathsf{Perf}(\D_X^{op}))\to HH_0(pt)\simeq \mathbb{C}.$$
Since $X$ is projective, it is $D$-affine, hence the following diagram commutes,
\[
\begin{tikzcd}
HH_{\bullet}\big(\mathsf{Perf}(\Gamma(X,\D_X)^{op})\big)\arrow[r,"\simeq"] \arrow[d,"\simeq"] & HH_{\bullet}\big(\mathsf{Perf}(\D_X^{op})\big)\arrow[d]
\\
HH_{\bullet}\big(\Gamma(X,\D_X)\big)\arrow[r] & HH_{\bullet}(\D_X).
    \end{tikzcd}
\]
Note also, analytification induced a morphism $(-)^{an}:HH_{\bullet}(\D_X)\to HH_{\bullet}(\D_{X^{an}}),$ and there is a corresponding isomorphism 
$$HH_{\bullet}\big(\mathsf{Perf}(\D_X^{op})\big)\to HH_{\bullet}(\D_X)\xrightarrow{(-)^{an}}HH_{\bullet}(\D_{X^{an}}).$$

\subsection{Microlocal index for $D$-algebras}
\label{ssec: PDEGeometryIndexes}
Let $X$ be a topological space $Z\subset X$ a closed subset and $\A$ a sheaf of algebras. Let $K_Z^i(\A)$ be the $i$-th $K$-group in $\mathrm{Perf}(\A),$ acyclic on the compliment of $Z.$

The Chern map is
$$\mathrm{Ch}_{Z,i}^{\A}:K_Z^i(\A)\rightarrow H_{Z}^{-i}(X,\mathrm{CC}_{\bullet}^-(\A)),$$
and the Euler-class map is 
\begin{equation}
    \label{eqn: Eui}
\mathrm{Eu}_{Z,i}^{\A}:K_Z^i(\A)\rightarrow H_Z^{-i}(X;C_{\bullet}(\A)).
\end{equation}
Both are natural in $X,Z,\A$ and (\ref{eqn: Eui}) coincides with the composition
$$K_Z^i(\A)\xrightarrow{\mathrm{Ch}_{Z,i}^{\A}}H_Z^{-i}(X;\mathrm{CC}_{\bullet}^-(\A))\rightarrow H_Z^{-i}(X;C_{\bullet}(\A)).$$
\subsubsection{$\D$-ellipticity}
Let $\mathrm{Fl}\M$ be an $\AD$-module of finite presentation endowed with a good-$\AD$-filtration. Let $\mathrm{Gr}^F\M$ be the associated graded module, viewed as a finitely-presented $\mathcal{O}_{T^*(X,\EQ)}$-module, with $\EQ=Spec_{\mathcal{D}}(\A).$
Let us set 
$$\mathsf{Char}_{\D}(\M):=supp\big(\mathrm{Gr}^F\M\big)\subset T^*X\times_X Spec_{\D}(\A).$$
Let $M$ be a real-analytic manifold with complexification $X.$ Then $\M$ is \emph{$\D$-elliptic} (relative to $p_{\infty}:Spec(\A)\rightarrow X$) if 
\begin{equation}
    \label{eqn: D-elliptic}
\mathsf{Char}_{\D}(\M)\cap T_M^*X\times_X Spec(\A)\subset M\times T_X^*X\times_X Spec(\A).
\end{equation}
Remark that $\mathsf{Char}_{\D}(Sym(\M))$ is the smallest conic, additively closed subset containing 
$\mathsf{Char}_{\D}(\M)$ for the fiberwise additive group structure of $T^*X$. That is, as $\D$-modules:
$Ch(\M)\subset Ch\big(Sym(\M)$. In particular, for $\A=Sym(\M)$, since $\mathbb{L}_{\A}\simeq Sym(\M)\otimes \M,$ one has 
$$\mathsf{Char}_{\D}(\A)=Ch(Sym(\M)\otimes \M)\subset Ch\big(Sym(\M))+_XCh(\M),$$
where for any two $V_1,V_2\subset T^*X,$ we have 
$V_1+_XV_2:=p_0(V_1\times_X V_2)$ where $p_0:T^*X\times_X T^*X\rightarrow T^*X$ is the addition map $(x;\xi_1),(x;\xi_2)\mapsto (x;\xi_1+\xi_2).$
Thus, 
$$V_1+_XV_2\simeq \big\{(x;\xi_1+\xi_2)| (x;\xi_i)\in V_i,i=1,2\big\}.$$

\subsection{$K$-theoretic compatibility}
There is an identification
$\mathrm{Spec}_X\big(\mathrm{Gr}\A[\mathcal{D}_X]\big)\simeq T^*\mathrm{Spec}_{\mathcal{D}_X}(\mathcal{A}),$ and there exists an embedding 
$T^*X\hookrightarrow T^*\mathrm{Spec}_{\mathcal{D}_X}(\mathcal{A}),$ which at the level of algebras, is a morphism of Poisson algebras since the procedure of lifting $\D_X$-valued sections $P$ induces
an embedding of graded Lie algebras
$\mathrm{Gr}(\mathcal{D}_X)\hookrightarrow \mathrm{Gr}(\mathcal{A}[\mathcal{D}_X]).$

We impose a simplifying condition.

\begin{assumption}
\label{ass: Simplifying}
Let $X$ be quasi-projective and suppose $\mathcal{A}$ has the properties that: (i) $\mathbb{T}_{\A}^{\bullet}$ has a good filtration $\mathrm{Fl}(\mathbb{T}_{\A}^{\bullet})$ by vector $\D_X$-bundles compatible with the dg-$\A[\mathcal{D}_X]$-module structure and with the action of $\mathrm{Fl}(\D_X)$ and (ii) for each solution $\varphi,$ that  $\mu(\mathbb{T}_{\A,\varphi})$ has bounded $\mathcal{O}_{T^*X}$-coherent cohomology. 
\end{assumption}

The Assumptions \ref{ass: Simplifying} are satisfied, for example, if $\A$ is homotopically of finite $\D$-presentation i.e. a compact $\D$-algebra, since in this case we may write $\mathbb{T}_{\A}$ as a bounded complex of finite rank $\A[\D]$-modules.

Let $\Lambda\subset T^*X$ and consider 
$K_{\Lambda}^0(T^*X),$ the Grothendieck group of the category of perfect complexes of $\mathcal{O}_{T^*X}$-modules supported in $\Lambda.$ Equivalently, those complexes which are acyclic on $T^*X\backslash \Lambda.$ Set $K^0(T^*X):=K_{T^*X}^0(T^*X).$

Suppose that $Char_{\mathcal{D},\varphi}(\A)\subset \Lambda.$ There is an induced class 
$$\mu_{\Lambda}(\mathbb{T}_{\A,\varphi})\in K_{\Lambda}^0(T^*X),$$
and by considering the Chern character $ch:K^0(T^*X)\rightarrow \bigoplus_i H^{2i}(T^*X;\mathbb{C}),$ or rather its natural extension with supports
$$\mathrm{Ch}_{\Lambda}:K_{\Lambda}^0(T^*X)\rightarrow \bigoplus_{i} H_{\Lambda}^{2i}(T^*X;\mathbb{C}),$$
we have a class
$\mathrm{Ch}_{\Lambda}\big(\mu_{\Lambda}(\mathbb{T}_{\A,\varphi})\big).$
\begin{definition}
\label{defn: D-Chern character}
Let $\A$ be a commutative dg-$\D_X$-algebra satisfying Asssumptions \ref{ass: Simplifying}. Then the class $\mathrm{Ch}_{\Lambda}\big(\mu_{\Lambda}(\mathbb{T}_{\A,\varphi})\big)$ is called the \emph{$\D$-geometric Chern character} of $\A$ along $\varphi.$
\end{definition} 
The \emph{microlocal Chern character} of $\A$ along $\varphi$ is:
\begin{equation}
    \label{eqn: MicroChernA}
\mu \mathrm{Ch}_{\Lambda}(\A;\varphi):=\mathrm{Ch}_{\varphi^{-1}\Lambda}\big(\sigma_{\Lambda}(\mathbb{T}_{\A,\varphi})\big)\cup \pi^* td_X(TX).
\end{equation}
Here, we have used an isomorphism $\pi^*:H_{supp}^0(X;\omega_X)\xrightarrow{\simeq} H_{\pi^{-1}supp}^0(T^*X;\pi^{-1}\omega_X).$
\begin{remark}
For $\F\in Coh(X)$ with $supp(\F)\subset X,$ since $char(\F\otimes_{\mathcal{O}_X}\D_X)\simeq \pi^{-1}supp(\F),$ then $\mu eu(\F\otimes_{\mathcal{O}_X}\D_X)\simeq \pi^* eu(\F\otimes_{\mathcal{O}_X}\D_X).$
\end{remark}
We now give a compatible formulation via the $K$-theory of perfect $\mathcal{A}[\mathcal{D}_X]$-modules. 
\begin{lemma}
\label{prop: K-theory commute1}
Let $\A$ be a compact $\D$-algebra and suppose $\mathsf{Char}_{\D}(\A)\subset \Lambda\subset Spec_X(\pi^*\mathcal{A}).$
Then, for each solution $\varphi\in \mathsf{Spec}_{X_{DR}}(\mathcal{A}),$ one has a commuting diagram:
\[
\begin{tikzcd}
K_{\Lambda}^0(\mathcal{A}[\mathcal{D}_X])\arrow[r]\arrow[d,"\phi^*"]& K_{\Lambda}^0(\mathcal{A}[\mathcal{E}_X])\arrow[d,"\phi^*"]
\\
K_{\phi^{-1}\Lambda}^0(\mathcal{D}_X)\arrow[r] & K_{\phi^{-1}\Lambda}^0(\mathcal{E}_X),
\end{tikzcd}
\]
where both horizontal arrows are extension of scalars.
\end{lemma}
\begin{proof}
Since $\pi^{-1}\D_X\hookrightarrow \mathcal{E}_X,$ is a faithfully flat morphism of filtered $\pi^{-1}\mathcal{O}_X$-algebras,
by extension of scalars, we have both $K^0(\D_X)\rightarrow K^0(\mathcal{E}_X),$ a map from Grothendieck group of perfect complexes of $\D$-modules to perfect complexes of $\mathcal{E}_X$-modules, as well as 
$K_{\Lambda}^0(\D_X)\rightarrow K_{\Lambda}^0(\mathcal{E}_X),$
for perfect complexes of $\D$-modules $\mathcal{M}^{\bullet}$ with $Char(\mathcal{M}^{\bullet})\subset \Lambda,$ and perfect complexes of $\mathcal{E}_X$-modules supported on $\Lambda.$
Similarly, given a compact object $\mathcal{A}^{\bullet}\in \mathrm{cdga}_{\mathcal{D}_X}^{\leq 0},$ with $\mathcal{A}^{\bullet}[\mathcal{D}_X]:=\mathcal{A}^{\bullet}\otimes_{\mathcal{O}_X}\mathcal{D}_X,$ it follows from the inclusion $\pi^{-1}\mathcal{O}_X\hookrightarrow \pi^{-1}\mathcal{D}_X,$ we have a canonical morphism
$$\mathcal{A}^{\bullet}\otimes_{\mathcal{O}_X}\mathcal{D}_X\rightarrow \pi_{\mathcal{D}}^*\mathcal{A}^{\bullet}\otimes_{\pi^{-1}\mathcal{O}_X}\pi^{-1}\mathcal{D}_X,$$
which extends further to coefficients in $\mathcal{E}_X$, due to the embedding $\pi^{-1}\mathcal{D}_X\hookrightarrow \mathcal{E}_X.$
$$\pi_{\mathcal{D}}^*\mathcal{A}^{\bullet}\otimes_{\pi^{-1}\mathcal{O}_X}\pi^{-1}\mathcal{D}_X\hookrightarrow \pi_{\mathcal{D}}^*\mathcal{A}^{\bullet}\otimes_{\pi^{-1}\mathcal{D}_X}\mathcal{E}_X.$$

    Note $\mathrm{supp}(\varphi^*\mu_{\mathcal{D},V}\mathbb{T}_{\mathcal{A}}^{\ell})\subset\varphi^{-1} \mathrm{supp}(\mu_{\mathcal{D},V}\mathbb{T}_{\mathcal{A}}^{\ell}).$ 
By our assumptions on $\mathcal{A}$, $\mathbb{L}_{\mathcal{A}}$ is perfect with dual $\mathbb{T}_{\mathcal{A}}^{\ell},$ and since $\varphi^*$ is a surjection, $Char_V^1(\mathbb{T}_{\mathcal{A},\varphi}^{\ell})$ and $\varphi^{-1}Char_{\mathcal{D},V}^1(\mathcal{A})$ coincide as analytic sub-spaces.
\end{proof}

\begin{theorem}
    \label{thm: C}
    Consider a morphism $f:X\rightarrow Y$ of complex analytic manifolds and consider the associated microlocal correspondence \emph{(\ref{eqn: Microlocal correspondence})} . Let $\A_Y$ be a compact dg-$\D$-algebra on $Y$ and suppose that 
\begin{equation}
\label{eqn: Condition1}
\mathring{\pi}_X^{-1}(p)\cap C_{T_Y^*Y}\big(\mathsf{Char}_{T_Y^*Y}^1(\varphi^*\mathbb{T}_{\mathcal{B}})\big)=\emptyset,
\end{equation}
for all $p\in T_Y^*Y$ with $\mathring{\pi}_X:\mathring{T}_X^*Y\times_Y T_Y^*Y\rightarrow T_Y^*Y$ the natural projection and where $\mathring{T}_X^*Y$ denotes the removal of the zero-section. Then for every solution $\varphi$, we have 
$$f_{d*}f_{\pi}^*\sigma_{\Lambda_{Y}}^{\mathcal{D}}(\A_Y;\varphi)=\sigma_{f_{\mu}^{-1}\Lambda_Y}^{\mathcal{D}}(f^*\A_Y,\gamma(\varphi)),$$
holds in $K_{f_{\mu}^{-1}\Lambda_Y}^0(T^*X),$ and in particular, 
$$f^*\mathrm{Ind}_{\D}(\A_Y;\varphi)=\mathrm{Ind}_{\D}(f^*\A_Y;\gamma(\varphi)).$$
\end{theorem}
\begin{proof}
Recall
\begin{equation*}
    \begin{tikzcd}
    Sol_{\mathcal{D}_Y}(\mathcal{A}^{\bullet})\times_Y T^*X\arrow[d] & \arrow[l,"f_d"] Sol_{\mathcal{D}_Y}(\mathcal{A}^{\bullet})\times_Y X\times_Y T^*Y\arrow[d]\arrow[r,"f_{\pi}"] & Sol_{\mathcal{D}_Y}(\mathcal{A}^{\bullet})\times_Y T^*Y\arrow[d]
    \\
    X\arrow[r,"\mathbf{1}_X"] & X\arrow[r,"f"] & Y
    \end{tikzcd}
\end{equation*}
By assumptions on $f:X\rightarrow Y$, we have for any $\varphi:Y\rightarrow \A_Y,$ that $SS(\mathcal{A}_Y;\varphi)\subset T^*Y.$ Then, by (\ref{eqn: Condition1}) i.e $f$ is $\D_X$-NC and  $f_d$ is proper on $f_{\pi}^{-1}(\mathsf{Char}_{\D}(\A,\varphi)).$ Therefore, by Proposition \ref{prop: K-theory commute1} there exists a well-defined morphism in $K$-theory with supports
$$f_{\pi}^*:K_{\mathsf{Char}_{\D}(\A,\varphi)}^{an}(T^*Y)\rightarrow K^{an}(X\times_Y T^*Y).$$
For convenience of notation, set 
$$Char_{\D;X}(\A_Y;\varphi):=f_d\big(f_{\pi}^{-1}\varphi^{-1}\mathsf{Char}_{\D}(\A)\big)=f_d\big(f_{\pi}^{-1}\mathsf{Char}_{\D}(\A;\varphi)\big).$$
Then, 
\begin{align*}
&\mathrm{ch}_{Char_{\D;X}(\A_Y;\varphi)}\sigma_{Char_{\D;X}(\A_Y;\varphi)}(f_{\D}^*\varphi^*\mathbb{T}_{\A}^{\ell})\cup td_X(T^*X)
\\
&=\mathrm{ch}_{Char_{\D;X}(\A_Y;\varphi)}\big[(f_d)_!f_{\pi}^*\sigma_{\mathsf{Char}_{\D}(\A_Y;\varphi))}(\varphi^*\mathbb{T}_{\A}^{\ell})\big]\cup td_X(T^*X)
\\
&=(f_d)_!\big[f_{\pi}^*\mathrm{ch}_{\mathsf{Char}_{\D}(\A_Y;\varphi)}\sigma_{\mathsf{Char}_{\D}(\A_Y;\varphi)}(\varphi^*\mathbb{T}_{\A}^{\ell})\cup td_X(X\times_Y T^*Y)\big].
\end{align*}
But since
$$td_X(X\times_YT^*Y)=\pi_X^*td_X(T^*X)\cup\pi^*f^*td_Y(T^*Y),$$
one has that 
\begin{align*}
&\mathrm{ch}_{Char_{\D;X}(\A_Y;\varphi)}\sigma_{Char_{\D;X}(\A_Y;\varphi)}(f_{\D}^*\varphi^*\mathbb{T}_{\A}^{\ell})\cup \pi_X^*td_X(T^*X)\cup \pi_X^*td_X(TX)
\\
&= f_{d!}\big[f_{\pi}^*\mathrm{ch}_{\mathsf{Char}_{\D}(\A;\varphi)}\sigma_{\mathsf{Char}_{\D}(\A;\varphi)}(\varphi^*\mathbb{T}_{\A}^{\ell})\cup \pi^*td_X(TX)\cup \pi^*f^*td_Y(T^*Y)\big]
\\
&=(f_d)_!f_{\pi}^*\big[\mathrm{ch}_{\mathsf{Char}_{\D}(\A_Y;\varphi)}\sigma_{\mathsf{Char}_{\D}(\A_Y;\varphi)}(\varphi^*\mathbb{T}_{\A}^{\ell})\cup \pi_Y^*td_Y(T^*Y)\big]\cup \pi_X^*td_X(TX).
\end{align*}

\end{proof}

We will now give two applications of this result.

\subsection{$\D$-geometric Atiyah-Singer index}
\label{ssec: Relation to AS}
Consider $\mathcal{A}=Sym^{*}(\M^{\bullet}),$ for a complex of $\D$-modules $\M^{\bullet}$ with
bounded good cohomology. Let $\Lambda$ be a given closed conic subvariety of $T^*X$ containing $Char(\M^{\bullet}).$ Theorem \ref{thm: C} agrees with the non-commutative notion of micro-local Euler class
$$\mu eu(-):K_{\Lambda}^0(\D_X)\rightarrow K_{\Lambda}^0(\mathcal{E}_X)\rightarrow H_{\Lambda}^{2d}(T^*X;\mathbb{C}).$$
Indeed in this case $\mathbb{L}_{\A}\simeq Sym^*(\M^{\bullet})\otimes \M^{\bullet},$ and since $\A$ is $X$-flat, by pull-back along $\varphi,$ and by letting $[-]^h$ be the homogeneous component of a graded object of degree $h$, we reduce to
\begin{equation}
 \mu eu(\M^{\bullet}):=\big[ch_{\Lambda}\big(\pi^{-1}\mathrm{Gr}_{\bullet}^{\mathrm{F}}(\M^{\bullet})\otimes_{\pi^{-1}\mathrm{Gr}_{\bullet}^{\mathrm{F}}\D_X}\mathcal{O}_{T^*X}\big)\cup \pi^* Td(TX)\big]^{2d}\in H_{\Lambda}^{2d}(T^*X;\mathbb{C}).
 \end{equation}

For $\D$-elliptic $\AD$-modules $\M$, we may specialize to the case of a free $\D$-algebra and recover a version of the classical Atiyah-Singer index \cite{AtiyahSinger1963}.
\begin{theorem}
\label{ASinger}
Let $M$ be a smooth real analytic manifold with complexification $X,$ and consider a $\mathcal{D}$-algebra $\A=Sym^*(\M^{\bullet})$ and suppose its $\mathcal{D}_X$-space of solutions $\mathsf{Spec}_{X_{DR}}(\mathcal{A})$ satisfies at each point,
\begin{equation}
\label{eqn: Elliptic condition 1}
\big(Char_{\mathcal{D}}(\mathcal{\A})\times \mathsf{Spec}_{X_{DR}}(\mathcal{\A})\big)\cap T_M^*X\subset Sol_{\mathcal{D}}(\mathcal{\A})\underset{X}{\times} M\times T_X^*X.
\end{equation}
Then, letting $0_M:M\hookrightarrow T_M^*X$ denote the zero-section and $j:T_M^*X\hookrightarrow T^*X,$ the canonical embedding, the
$$\chi_{\D}(\A)=\int_M 0_M^*\big[j^*\mathrm{ch}\sigma_{\mathrm{Char}(\M)}(\M)\big]\cup td_M(TM^{\mathbb{C}}).$$
\end{theorem}

\begin{proof}
Consider  
$\chi_{\D}(\A)=\chi(R\Gamma(X,\RHom_{\AD}(\mathbb{L}_{\A},\mathcal{O}_X)\big).$
Now, condition (\ref{eqn: Elliptic condition 1}), reduces to stating that 
\begin{equation}
\label{eqn: Elliptic condition 2}
supp\big(\mu\mathbb{T}_{\mathcal{\A},\varphi}^{\ell})\cap T_M^*X\subset M\times_X T_X^*X,
\end{equation}
holds for every solution $\varphi$.
Since $\A$ is compact, the argument may be reduced to the case of a free $\D$-algebra. 
Thus, we may assume that $\A=Sym^*(\M)$ for a perfect $\D_X$-module $\M,$ and can even assume $\M$ is itself elliptic as a $\D$-module. Indeed, condition (\ref{eqn: Elliptic condition 2}) directly implies $\M$ is an elliptic $\D$-module since for each $\varphi,$  pull-back of $\mathsf{Char}_{\D}\big(Sym^*(\M))$ gives 
$supp(\mu \varphi^*Sym^*(\M)\otimes \M),$ which reduces directly to $\mathrm{Char}(\M).$ 
Therefore consider the pair $(\M,\mathbb{C}_M)$ with $\M$ elliptic on $M$, with complexification $X.$ Let $0_M:M\hookrightarrow T_M^*X$ denote the zero-section and consider the embedding $j:T_M^*X\hookrightarrow T^*X.$ Then, to prove our result, we need to show that 
$$\chi_{\D}(\A)=\chi_{AS}(\M),$$
with $\chi_{AS}$ the Atiyah-Singer index. 
Since
$$\RHom_{\AD}(\mathbb{L}_{\A},\mathcal{O}_X)\simeq \RHom_{\mathrm{Sym}_{\mathcal{O}_X}^*(\M)\otimes \D}(\mathrm{Sym}_{\mathcal{O}_X}^*(\M)\otimes \M,\mathcal{O}_X),$$
we see that 
$\chi_{\D}(\A)=\chi\big(RSol_{X}(\M,\mathcal{O}_X)\big),$ which by duality and by taking global sections over $M$, is given by 
$$\chi\big(R\Gamma(M,\M\otimes_{\D}^L\mathcal{O}_X)\big)=\int_{T^*X}\mu eu(\M)\cup \mu eu(\mathbb{C}_M),$$
for the constant sheaf $\mathbb{C}_M.$
Then from (\ref{eqn: Elliptic condition 1}) condition (\ref{eqn: Elliptic condition 2}) implies $T_M^*X\cap \mathrm{Char}(\M)\subset M,$ and therefore we have that 
$$\int_{T_M^*X}j^*\mu eu(\M)=\int_M 0_M^* j^*\mu eu(\M).$$
One may readily verify this is given by
$$\int_M 0_M^*j^*\big[\mathrm{ch}(\sigma_{\mathrm{Char}(\M)})\cup \pi^* td_X(TX)\big] =\int_M 0_M^*\big[j^*\mathrm{ch}\sigma_{\mathrm{Char}(\M)}(\M)\big]\cup td_M(TM^{\mathbb{C}}).$$
Then, the result follows since the right-hand side is exactly $\chi_{AS}(\M).$
\end{proof}

\subsubsection{Boundary}
\label{ssec: Boundary}
Let $\partial X\subset X$ be the boundary. In this subsection, we give an index via relative Spencer hypercohomology but mention for completeness this does in fact belong to the larger class of \emph{boundary solution problems}. If this boundary problem is elliptic, in the sense its linearization
 at each solution determines an elliptic boundary problem, then the tangent complex is perfect and its Euler-characteristic produces boundary index values at each solution.
 
Let $E,F\in VBun(X),H\in VBun(\partial X).$ Consider $P:J_X^kE\to F$, a $k$-th order PDE and boundary condition is a map 
$B:J_X^rE\times_X \partial X\to H, r\leq k.$ Homotopy pullback $Sol_{\partial}:=0\times_{RMaps_{/X}(X,F)\times RMaps_{/\partial X}(\partial X,H)}^h RMaps_{/X}(X,E),$ via the zero map and $P\times B,$ is the space of solutions to the boundary problem.

 Let $\mathcal{D}_{\partial X\to X}:=\mathcal{O}_{\partial X}\otimes_{i^{-1}\mathcal{O}_X}i^{-1}\D_X$ be the $(\mathcal{D}_{\partial X},i^{-1}\D_X)$-bimodule and denote the pull-back and push-forwards as 
$$\mathbf{L}i_{\partial}^*\M:=\D_{\partial X\to X}\otimes_{i^{-1}\D_X}^Li^{-1}\M,\hspace{2mm} i_{*}\M:=\mathbf{R}i_*\big(\M\otimes_{\D_{\partial X}}^L\D_{\partial X\to X}\big),$$
respectively.
Analogously to the usual Spencer resolution $\mathrm{DR}_X^{\bullet}(\mathcal{D}_X)$ we consider $$DR_{\partial X}^{\bullet}(\mathcal{D}_X)\simeq i_{\partial *}\big(\Omega_{\partial X}^{\bullet}\otimes_{i_{\partial}^{-1}\mathcal{O}_X}^{\mathbb{L}}i_{\partial}^{-1}\mathcal{D}_X\big),$$ and note the existence of a morphism of sheaves on $X,$
$$\widetilde{i}_{\partial}^*:\mathrm{DR}_X^{\bullet}(\mathcal{D}_X)\rightarrow DR_{\partial X}^{\bullet}(\mathcal{D}_X),$$
induced by the pull-back of forms along $i_{\partial}.$ 
The following is given by \cite[Prop.~9.1]{KSY2}.
\begin{lemma}
    The sheaf of smooth densities $Dens(X/\partial X)$ on $X$ relative to the boundary $\partial X$, has a Spencer resolution as a $\mathcal{D}_X^{op}$-module and there is a quasi-isomorphism $Cone(\widetilde{i}_{\partial}^*)[d_X-1]\xrightarrow{\simeq} Dens(X/\partial X).$

\end{lemma}
We then compute that
\begin{lemma}
    \label{prop: BoundaryIndex}
   Consider $X,\partial X$ and $f:X\to S$ as above, with $\mathcal{B}_{X/S}$ the $D$-algebra of functions of a formally integrable $D$-scheme. Then, the relation 
$$\big[\mathbf{R}f_*\mathrm{Cone}(\widetilde{i}_{\partial}^*)\big]=\big[\mathbf{R}f_*\mathcal{S}p_{X/S}^{\bullet}(\mathcal{B}_{X/S})\big]-\big[\mathbf{R}f_*i_{*}\mathcal{S}p_{\partial X/S}^{\bullet}(\mathbf{L}\iota_{\partial}^*\mathcal{B}_{X/S})\big],$$
holds in $K_0(S).$
\end{lemma}
In other words, we have
$\mathrm{Ind}^{rel}(\mathcal{B}_{X/S})=\mathrm{Ind}(\mathcal{B}_{X/S})-\mathrm{Ind}^{\partial}(\mathcal{B}|_{\partial X}).$
\begin{proof}
   Since $\mathrm{ch}:K_0(S)\to H^{2*}(S,\mathbb{Q})$ is additive we obtain 
   $$\mathrm{ch}(\mathbf{R}f_*\mathrm{Cone}(\widetilde{i}_{\partial}^*))=\mathrm{ch}\big(\mathbf{R}f_*\mathcal{S}p_{X/S}(\mathcal{B}_{X/S})\big)-\mathrm{ch}\big(\mathbf{R}f_*i_*\mathcal{S}p_{\partial X/S}^{\bullet}(\mathbf{L}i_{\partial}^*\mathcal{B}_{X/S})\big).$$
   Now, Proposition \ref{GRRSpencerDelta} gives that 
   $$\mathrm{ch}(\mathbf{R}f_*\mathcal{S}p_{X/S}^{\bullet}(\mathcal{B}_{X/S})\big)=f_*\big(\mathrm{ch}(\mathcal{S}p_{X/S}(\mathcal{B}))\cdot \mathrm{Td}(T_{X/S})\big),$$
   and similarly, letting $f^{\partial}:\partial X\xrightarrow{i_{\partial}}X\xrightarrow{f}S,$ we have 
   $$\mathrm{ch}\big(\mathbf{R}f_*^{\partial}\mathcal{S}p_{\partial X/S}^{\bullet}(\mathbf{L}i^*\mathcal{B})\big)=f_{*}^{\partial}\big(\mathrm{ch}(\mathcal{S}p_{\partial X/S}^{\bullet}(\mathbf{L}i^*\mathcal{B}))\cdot \mathrm{Td}(T_{\partial X/S})\big).$$
   We may write
   $$f_{*}^{\partial}(\mathrm{ch}(\mathcal{S}p_{\partial X/S}(\mathbf{L}i^*\mathcal{B}))\cdot \mathrm{Td}(T_{\partial X/S})\big)=f_*\bigg(i_*(\mathrm{ch}(\mathcal{S}p_{\partial X/S}^{\bullet}(\mathbf{L}i^*\mathcal{B})))\cdot \mathrm{Td}_{T_{\partial X/S}}\bigg),$$
   which is equivalent to $f_*\big(\mathrm{ch}(i_*\mathcal{S}p_{\partial X/S}^{\bullet}(\mathbf{L}i^*\mathcal{B})\cdot \mathrm{Td}(T_{X/S})\big).$
   Then, we see that 
   $$\mathrm{ch}\big(\mathbf{R}f_*\mathrm{Cone}(\widetilde{i}_{\partial}^*))=f_*\big(\mathrm{ch}(\mathcal{S}p_{X/S}(\mathcal{B}))-\mathrm{ch}(i_{*}\mathcal{S}p_{\partial X/S}^{\bullet}(\mathbf{L}i^*\mathcal{B}))\cdot \mathrm{Td}(T_{X/S})\big),$$
   thus we obtain
\begin{align*}
\mathrm{Ind}^{rel}&=\int_{T^*(X/S)}\mathrm{ch}(\mathcal{S}p_{X/S}(\mathrm{gr}_K(\mathcal{B}))\wedge \mathrm{Td}(\pi^*T_{X/S})
\\
&-\int_{T^*(\partial X/S)}\mathrm{ch}(\mathcal{S}p_{\partial X/S}^{\bullet}(\mathbf{L}i^*\mathrm{gr}_{K}(\mathcal{B})))\wedge \mathrm{Td}(T_{\partial X/S}).\end{align*}
\end{proof}
To conclude, recall the notion of boundary $\D$-scheme given in \cite[Sect 9.1]{KSY2}.
It defines an object of $\mathrm{CAlg}(\D_{\partial X}).$ One may compute its index from the absolute hypercohomology over $\partial X.$ In general, it does not agree with the index above, which is taken relative to the boundary.

\subsection{Applications: Elliptic Lie-systems}
\label{ssec: PDEGeometryApplications}
Let $(X,\omega)$ be a compact Kähler manifold, and $E\to X$ a holomorphic vector bundle. 
Let $\EQ^k\subset \mathrm{Jets}_X^k(Der(E))$ be a formally integrable geometric structure of order $k$, and let
\begin{equation}
\label{eq:D-ideal}
0 \longrightarrow \mathfrak{I}_k \longrightarrow \mathcal{O}(\mathrm{Jets}_X^k(Der(E))) \longrightarrow \mathcal{O}_{\EQ^k} \longrightarrow 0
\end{equation}
be the associated $\D_X$-ideal sequence. Assume $\EQ^k$ is elliptic.

Via the standard order filtration $F^\bullet$ on the jet algebra:
\[
F^j \mathcal{O}(\mathrm{Jets}_X^\infty(Der(E))) := \mathcal{O}(\mathrm{Jets}_X^j(Der(E))).
\]

The ideal inherits the filtration $
F^j \mathfrak{I}_k := F^j \mathcal{O}(\mathrm{Jets}_X^\infty(Der(E))) \cap \mathfrak{I}_k$ and we let the associated graded to this filtration be $
\mathrm{gr}^j_F(\mathfrak{I}_k) := F^j \mathfrak{I}_k / F^{j-1} \mathfrak{I}_k.$

The geometric symbol is then $
\mathfrak{g}^k := \mathrm{gr}^k_F(\mathfrak{I}_k) \subset Sym^k(T_X^*)\otimes Der(E),$
which is coherent over $\mathcal{O}_X$.
In other words,  Its conormal sheaf
\[
\mathcal{C}_{Z^k/\mathrm{Jets}_X^k(Der(E))} := \mathcal{I}_k/\mathcal{I}_k^2 \subset \Omega_{\mathrm{Jets}_X^k(Der(E))}^1 \otimes \mathcal{O}_{Z^k}
\]
encodes the infinitesimal structure of the PDE. Denote by $Pr_\ell(Z^k)\subset \mathrm{Jets}_X^{k+\ell}(Der(E))$ the $\ell$-th prolongation. The geometric symbol is
\[
\mathfrak{g}^k := \ker\big(\eta_{k,k-1}|_{Z^k}:\mathrm{Jets}_X^k(Der(E))\to \mathrm{Jets}_X^{k-1}(Der(E))\big) \subset Sym^k(\Omega_X^1)\otimes Der(E).
\]
Its $\ell$-th prolongation is
\[
(\mathfrak{g}^k)^{(\ell)} \subset Sym^{k+\ell}(\Omega_X^1) \otimes Der(E),
\]
 total characteristic $\D$-module defined by all $Ch_k\subset Sym^k(\Omega_X^1)\otimes Der(E),$ with all prolongations: 
$$Pr_1(Ch_k):=(T^*X\otimes Ch_k)\cap Sym^{k+1}T^*X\otimes Der(E),$$
and more generally,
$$Pr_{r+1}(Ch)=Pr_1\big(\otimes_{i=1}^rT^*X\otimes Ch_k)\cap (Sym^{k+r}T^*X\otimes DerE).$$

\subsubsection{Topological index}
We work relative to a fibration $f:X\rightarrow S$ and consider the restriction of the $\D$-subscheme $\EQ$ to fibers of $f$. Noting the exact sequence of vector bundles $0\rightarrow \pi^{-1}E\rightarrow TE\rightarrow \pi^{-1}TX\rightarrow 0,$ then in terms of $\mathcal{I}$ which defines $\EQ,$ the restriction $\EQ_{X/S}$ over the fibers is given by the image of $\mathcal{I}$ in $\D_{X/S}\otimes \Omega_{X/S}^1.$
Indeed, $0\rightarrow \D_X\otimes f^*\Omega_S^1\rightarrow \D_X\otimes\Omega_X^1\xrightarrow{r}\D_X\otimes\Omega_{X/S}^1\rightarrow 0,$ then $r(\mathcal{I})\cap \D_{X/S}\otimes \Omega_{X/S}^1$ is the $\D$-submodule associated to the equations of restriction to fibers.
Let $\EQ_{X/S}$ denote the family of fiber-wise restricted system (cf, \eqref{FamilyPDE}) namely, $\{\EQ_s^{(i)}\hookrightarrow \mathrm{Jets}_{X_s}^i(E_s)|s\in S\}.$
If $\EQ$ is $\D$-involutive, then $\EQ_{X/S}$ is $\D$-involutive. Indeed, if $\EQ$ is involutive (resp. formally integrable) in the standard manner, then $\EQ_s$ is involutive (resp. formally integrable) for all $s\in S.$

Consider the fiber-wise Spencer complex:
$$0\rightarrow \EQ_{s}^{(\ell)}\xrightarrow{D_{s}}\EQ_{s}^{(\ell-1)}\otimes \Omega_{X/S}^1\rightarrow\cdots\rightarrow \EQ_{s}^{(\ell-r)}\otimes\Omega_{X/S}^r\rightarrow 0,$$
where $r=rank(f).$ There is a natural symbol class
$[\sigma(\EQ_{X/S})]\in K^0(T_{X/S}),$ and supposing the fibers are compact, denoting a generic fiber by $X_{\infty}$, then this class yields an element of $K^0(T_{X_{\infty}}X).$

The index formula for $\EQ_{X/S}$ in families is the following topological index.
\begin{lemma}
\label{NC Rel Elliptic}
Let $\EQ$ be a formally integrable $\D$-algebraic PDE and $f:X\rightarrow S$ a non-characteristic fiber bundle. Suppose that $\EQ_{X/S}$ is relatively elliptic for $f.$ Then
\begin{equation}
\label{eqn: TopInd}
Ch\big(\mathrm{ind}(\EQ_{X/S})\big)=(-1)^{\mathrm{dim}(X_{\infty})}\int_{X_{\infty}}ch(\sigma(\EQ_{X/S}))\wedge Td_{\mathbb{C}}(T_{X_{\infty}}X),
\end{equation}
holds, where $ch:K^0(S)\rightarrow H^*(S),$ and with 
$$\mathrm{ind}(\EQ_{X/S}):=\big[\chi(\EQ_{X/S})\big]=\sum_{i\geq 0}(-1)^i\big[\mathcal{H}^i(\EQ_{X/S})\big]\in K^0(S),$$
where we have set 
$$\int_{X_{\infty}}:H^*(T_{X_{\infty}}X)\rightarrow H^{*-\mathrm{dim}(X_{\infty})}(S),$$
the fiber integration and with $Td_{\mathbb{C}}(T_{X_{\infty}}X)$ the Todd class of the complexification of the fiber-wise tangent bundle.
\end{lemma}

\subsection{Hypercohomology of geometric structures}
\label{ssec: Hyp of geom str}
 begin by considering the order filtration $F^\bullet \mathcal{B}_{X/S}$ on $\mathcal{B}_{X/S}$, which is exhaustive, bounded below, and coherent over $\mathcal{O}_X$. Tensoring with the relative Spencer complex $\mathcal{S}p_{X/S}^\bullet$ gives a filtered complex of $\mathcal{O}_X$-modules
\[
K^p \mathcal{S}p_{X/S}^\bullet(\mathcal{B}_{X/S}) := \sum_{q+k \ge p} \mathcal{S}p_{X/S}^k(F^q \mathcal{B}_{X/S}).
\]
The associated graded pieces are
\[
\mathrm{gr}_K^q \mathcal{S}p_{X/S}^\bullet(\mathcal{B}_{X/S}) \simeq \mathcal{S}p_{X/S}^\bullet(\mathrm{gr}_K^q \mathcal{B}_{X/S}), \quad \mathrm{gr}_K^q \mathcal{B}_{X/S} := F^q \mathcal{B}_{X/S}/F^{q+1} \mathcal{B}_{X/S},
\]
which are coherent $\mathcal{O}_X$-modules. By formal integrability, the Spencer differential preserves this filtration.

Next, consider the exact sequence of tangent sheaves
\[
0 \to \Theta_{X/S} \to \Theta_X \to f^*\Theta_S \to 0.
\]
This induces a Leray-type filtration $L^\bullet$, as originally defined by Katz-Oda \cite{KatzOda1968}, on the Spencer complex by
\[
L^p \mathcal{S}p_{X}^\bullet(\mathcal{B}_{X/S}) := \mathrm{Ker}\Big(\wedge^p \Theta_X \otimes \mathcal{S}p_X^{\bullet-p} \to \bigoplus_{r>p} f^* \wedge^r \Theta_S \otimes \wedge^{r-p} \Theta_X \otimes \mathcal{S}p_X^\bullet\Big),
\]
so that $
\mathrm{gr}_L^p \mathcal{S}p_X^\bullet(\mathcal{B}_{X/S}) \simeq \wedge^p f^*\Theta_S \otimes \mathcal{S}p_{X/S}^{\bullet-p}(\mathcal{B}_{X/S}).$

The total filtration $F_\mathrm{Tot}^p := \sum_{i+j=p} L^i \cap K^j$ is bounded, exhaustive, and preserved by the Spencer differential. Therefore, it defines a spectral sequence converging to the hypercohomology
\[
\mathbb{H}^\bullet(S, \mathbf{R}f_* \mathcal{S}p_{X/S}^\bullet(\mathcal{B}_{X/S})).
\]

The associated graded of $F_\mathrm{Tot}^\bullet$ is
\[
\mathrm{gr}_{F_\mathrm{Tot}}^p \mathcal{S}p_{X/S}^\bullet(\mathcal{B}_{X/S}) \simeq \bigoplus_{i+j=p} f^* \wedge^i \Theta_S \otimes \mathcal{S}p_{X/S}^{\bullet-i}(\mathrm{gr}_K^j \mathcal{B}_{X/S}).
\]

Applying $\mathbf{R}f_*$ and the projection formula gives
\[
\mathbf{R}f_* \mathrm{gr}_{F_\mathrm{Tot}}^p \mathcal{S}p_{X/S}^{p+q}(\mathcal{B}_{X/S}) \simeq \bigoplus_{i+j=p} \wedge^i \Theta_S \otimes \mathbf{R}^{q-i} f_* \mathcal{S}p_{X/S}^{\bullet-i}(\mathrm{gr}_K^j \mathcal{B}_{X/S}).
\]
Taking hypercohomology over $S$ identifies the $E_1$-page:
\[
E_1^{p,q} \simeq \bigoplus_{i+j=p} \mathcal{H}^{q-i}\Big(S, \wedge^i \Theta_S \otimes \mathbf{R}^{q-i} f_* \mathcal{S}p_{X/S}^{\bullet-i}(\mathrm{gr}_K^j \mathcal{B}_{X/S})\Big),
\]
with differential $d_1$ induced by the Spencer differential along horizontal vector fields, which reduces to the Gauss–Manin connection for $p=0$.

Finally, non-characteristic base change ensures that for each $s\in S$, the restriction
\[
\mathcal{S}p_{X/S}^\bullet(\mathcal{B}_{X/S})|_{X_s} \simeq \mathcal{S}p_{X_s}^\bullet(i_s^*\mathcal{B}_{X/S}),
\]
and hence $
\mathbf{R}f_* \mathcal{S}p_{X/S}^\bullet(\mathcal{B}_{X/S}) \otimes_{\mathcal{O}_S} k(s) \simeq \mathscr{H}_{Sp}^\bullet(X_s; i_s^*\mathcal{B}_{X/S}),$
as required.

\subsubsection{Ellipticity}
Let $M$ be smooth real analytic manifold with complexification $X$. Let $\mathcal{A}_E\rightarrow \mathcal{B}$ be the $\D$-PDE associated to a geometric structure of order $k$ i.e. 
$$A_E:=\mathcal{O}\big(\mathrm{Jets}_X^{\infty}Der(E)\big),$$
and at the defining level $k$:
\begin{equation}
    \label{eqn: Gauge structure diagram}
\begin{tikzcd}
    \mathfrak{g}^k\simeq Ker(\hat{\sigma}_k^{\EQ})\arrow[d]\arrow[r]& \EQ_k\arrow[d]\arrow[dr,"\hat{\sigma}_k",dotted] \arrow[r]& \EQ_k^0 \arrow[d]& 
    \\
    \mathrm{Jets}_X^kEnd(E)\simeq ker(\hat{\sigma}_k)\arrow[d]\arrow[r]& \mathrm{Jets}_X^k(DerE)\arrow[d] \arrow[r,"\hat{\sigma}_k"]& \mathrm{Jets}_X^k(\Theta_X)\arrow[d]
    \\
    End(E)\arrow[r] & Der(E)\arrow[r,"\sigma"]& \Theta_X
\end{tikzcd}
\end{equation}
where $\EQ_k^0$ is the Lie-equation on $X$ i.e. $\EQ_k^0\rightarrow \mathrm{Jets}_X^k(T_X)$ is a PDE and its sheaf of sections is an $\mathcal{O}_X$-Lie subalgebra of $\mathrm{Jets}_X^k(\Theta_X).$

\begin{definition}
    \normalfont 
Let $k\in \mathbb{N}$, and consider $\mathrm{Spec}_{\mathcal{D}}(\mathcal{B}),$ be the $\D$-scheme associated to a holomorphic geometric structure of order $k$ in $E.$
It is \emph{elliptic} if for every solution $\varphi:X\rightarrow \mathrm{Spec}_{\mathcal{D}}(\mathcal{B})$, the support condition $supp(\mu\Theta_{\mathcal{B},\varphi})\cap T_M^*X\subset M\times_X T_X^*X$ holds.
\end{definition}

More invariantly, we ask that
$$Char_{\mathcal{D}}(\mathcal{B})\times \mathrm{Spec}_{\mathcal{D}}(\mathcal{B})\cap T_M^*X\subset \mathrm{Spec}_{\mathcal{D}}(\mathcal{B})\times_XM\times T_X^*X.$$

This gives a global geometric characterization of ellipticity: if $\mathcal{B}$ is explicity given by an operator $P$, it means the symbol of the operator has no real characteristic covectors.

Suppose $\EQ_0\subset \mathrm{Jets}_X^k(DerE)$ is a $k$-th order geometric structure. Then $\EQ_0$ is elliptic if its geometric symbol $\mathcal{N}^k\subset Sym^k(\Omega_X^1)\otimes Der(E)$ has no real characteristic covectors, i.e. for any $\xi_x\neq 0\in \Omega_{X,x}^1$ we have that 
$\mathcal{N}_x^k\cap Der_x(E)\otimes \xi_x^k=0,$
for any $x\in X.$
\subsubsection{}
The index of a $k$-th order geometric structure with defining ideal $\mathfrak{I}_k$ is the Euler-characteristic of the Spencer complex:
$$\chi(\mathfrak{I}_k):=\sum_{i\geq 0}(-1)^i\mathrm{dim} H^i(\mathfrak{I}_k).$$
Via the Spencer complex for $\mathrm{Jets}_{X}^k(\Theta_X)$ and for the gauge structure $\mathfrak{g}^k$ by restriction of Spencer complex of $ker(\hat{\sigma}_k)$, as in (\ref{eqn: Gauge structure diagram}), we have the following.
\begin{lemma} The following relation holds,
$Ind(\EQ_k)=Ind(\mathfrak{g}^k)+ Ind(\EQ_k^0).$
\end{lemma}
\begin{proof}
This is immediate from the the usual additivity of Euler-characteristics for split-exact sequences and thus 
$Ind(\EQ_k)= Ind(\mathfrak{g}^k)+ Ind(\EQ_k^0),$
follows directly from (\ref{eqn: Gauge structure diagram}), where we set $Ind(\EQ_k):=\chi(\mathfrak{I}_k),$
 for $\EQ_k$ with ideal $\mathfrak{I}_k.$
\end{proof}

\subsubsection{Topological index of elliptic geometric structures}
Following \cite{LR} closely, we consider an elliptic geometric structure $\EQ$ in a holomorphic vector bundle $E\rightarrow X$. Without loss of generality we may take $\EQ$ to be of order $\leq 1.$ The following result computes the topological index in terms of its stable Spencer cohomology (cf.  \ref{eqn: FiberwiseRelSpener}), which follows from Lemma \ref{NC Rel Elliptic}.
\begin{lemma}
\label{prop: TopIndex}
Given $\EQ^1\subset \mathrm{Jets}^1(DerE),$ then $\EQ_{f}^1\subset \mathrm{Jets}^1(DerE|_s),$ and recalling 
$\mathrm{Ind}(\EQ):=\sum_{i\geq 0}(-1)^i\mathrm{dim}_{\mathbb{C}}\mathcal{H}^i(\EQ),$ we have 
$$\mathrm{ch}[\chi(\EQ^1)]=(-1)^{\mathrm{dim}(X_{\infty})}\chi(S)\sum_{r=0}^{2k+\ell_0}\sum_{i=0}^{2k}
(-1)^{i+k}\int_{X_{\infty}}\mathrm{ch}(\mathcal{H}_{\mathbb{C}}^{r-i,i}(\mathfrak{g}_f)\wedge \mathrm{Td}_{\mathbb{C}}(T_{X_{\infty}}X)\frac{\mathrm{Td}_{\mathbb{C}}S}{\widetilde{e}(S)},$$
where $\ell_0=\ell_0(rankE,dimX,order\EQ),$ is the degree of involutivity, obtained via the $\delta$-Poincare lemma in stable spencer cohomology.
\end{lemma}

\subsection{Application: Finite-type systems}
A \emph{compatibility condition} of a system of PDE's 
$E\subset J^k(n,m)$ is an algebraic equation defining 
$\pi_{k+\ell,k}(E)$ which is algebraically independent of $F=0$ and 
which is satisfied by all formal solutions. The Frobenius theorem states the following.
\begin{theorem}
\label{FTypeTheorem}
Let $E\subset J^k(n,m)$ be a PDE of finite type $\ell$. Then every 
solution is uniquely determined by its $(k+\ell-1)$-jet. If in 
addition $E$ has no compatibility conditions, then for every 
$\xi\in E^{(k+\ell)}$ there exists a local solution 
$u\in Sol(E)$ satisfying $j_x^{k+\ell}u=\xi$.
\end{theorem}
If a PDE is of finite type, then all of its 
prolongations are also of finite type. An analogous 
generalization holds for infinite type systems under the assumption 
of analyticity of $E$.

\begin{theorem}
\label{FTypeTheorem2}
    If $\mathcal{E}$ is finite-type, 
    $\mathrm{dim}\mathrm{Sol}(\mathcal{E})\leq \mathrm{sup}_{x\in X}\sum_i \mathrm{dim} g_i(x).$
\end{theorem}

An alternative characterization is as follows: $P:\Gamma(E)\to \Gamma(V)$ is finite-type if there exists $\ell_0<\infty$ and a bundle morphism $k:J^{\ell_0}E\to J^{\ell_0+1}V,$ and a differential operator $Q:\Gamma(V)\to \Gamma(J^{\ell_0+1}E),$ for which 
\begin{equation}
    \label{eqn: FT-condition}
j^{\ell_0+1}s-k(j^{\ell_0}s)=Q\big(P(s)\big),
\end{equation}
for every $s\in \Gamma(E).$
\begin{lemma}
\label{FTypeVBundle}
    Suppose $P:E\to V$ is a differential operator between vector bundles over $X.$ Assume $P(s)=0$ is finite type. Then it is equivalent to a flat connection equation $Du'=0,$ on a canonical finite–rank bundle $U'\to X.$
\end{lemma}
\begin{proof}
    The bundle $U'$ is the canonical quotient bundle determined by 
    $$\mathrm{coker}(\mathrm{Im}(Q)\subset J^{\ell_0+1}E\xrightarrow{\pi^{\ell_0+1}_{\ell_0}}J^{\ell_0}E).$$
    Namely, it is the quotient bundle obtained by modding out relations appearing in images of $Q$ after projection to order $\ell_0$,
    $$U'\simeq J^{\ell_0}E/\pi_{\ell_0}(\mathrm{Im}Q).$$
 Thus, on the equation manifold $Z_P$ corresponding to $P(s)=0,$ from \eqref{eqn: FT-condition} one has $j^{\ell_0+1}s=k(j^{\ell_0}s),$ so indeed, any $(\ell_0+1)$-st jet is determined by its $\ell_0$-th jet. Define
 \begin{equation}
     \label{eqn: FT Conn}
     D:=S\circ \big(\pi^{\ell_0+1}_{\ell_0}\circ k)^{-1}:U'\to \Omega_X^1\otimes U'.
 \end{equation}
 Hence, $k$, after modding out $\mathrm{Im}(Q),$ induces $U'\to J^{\ell_0+1}E/\mathrm{Im}Q,$which is composed with the Spencer differential $S:J^{\ell_0+1}E\to \Omega_X^1\otimes J^{\ell_0}E,$ and projecting back to $U'.$
 Since local sections $s$ coming from actual solutions of $P(s)=0,$ give $j^{\ell_0+1}s=k(j^{\ell_0}s),$ then 
 $S(j^{\ell_0+1}s)=S\circ k(j^{\ell_0}s),$ but since $S(j^{\ell_0+1}s)=d_{DR}(j^{\ell_0}s),$ sections coming from solutions thus correspond to $u'\in \Gamma(U')$ satisfying
 $D u’=0.$
\end{proof}

\subsection{Moduli interpretation}
We briefly interpret the above result via the moduli of $D$-schemes \cite{KSh1}.
The \emph{stable locus of the $D$-Hilbert functor}, is explicitly understood as the set of points $[\mathcal{E}]$ with continuous rational Poincaré series,
\begin{equation}
    \label{eqn: Rational Poincare}
\mathrm{P}_{G}^{\mathcal{E}}(z)=\sum_{k=0}^{\infty}\mathrm{tr.deg}\big(\mathfrak{R}(\mathcal{E}_\theta^k)^{G_{\theta}^k})-\mathrm{tr.deg}(\mathfrak{R}(\mathcal{E}_{\theta}^{k-1})^{G_{\theta}^{k-1}})\cdot z^k.
\end{equation}

We will put,
$$\mathrm{DHilb}^{P,{\mathrm{f.t}}}(m,n):=\{[\mathscr{E}_{\infty}]\in \mathrm{DHilb}^{P}(m,n) | \exists \ell_0 \text{ with } \mathfrak{g}^{\ell_0+1}=0,\mathfrak{g}^{\ell_0}\neq \emptyset\},$$
thus consisting of those points for which, $\pi_{\ell_0}^{\ell_0+1}:\mathscr{E}_{\ell_0+1}\simeq \mathscr{E}_{\ell_0},$ and from Lemma \ref{FTypeVBundle}, there exists a connection
$\nabla:=D\circ (\pi_{\ell_0}^{\ell+1})^{-1}:\mathscr{E}_{\ell_0}\to \Omega_X^1\otimes \mathscr{E}_{\ell_0},$
satisfying
$\nabla\circ\nabla=0.$
\begin{lemma}
$\mathrm{DHilb}^{P,\mathrm{f.t}}(n,m)$ is finite-type and representable by a quasi-projective scheme.
\end{lemma}
\begin{proof}
    The result follows from Theorems \ref{FTypeTheorem} and \ref{FTypeTheorem2}. By Lemma \ref{FTypeVBundle}, we have that $\mathrm{DHilb}^{P,\mathrm{f.t}}(n,m)\hookrightarrow \mathrm{Quot}^P\big(J^{\ell_0}(E,n)).$ Since $J^{\ell_0}(E,n)$ is quasi-projective, the Quot-functor is representable by a quasi-projective scheme.
\end{proof}

Fix $\Lambda\subset T^*X$ a closed $\mathbb{C}^{\times}$-conic analytic susbet. Look at PDEs with $\mathrm{Char}\subset \Lambda$
$$\mathrm{DHilb}^{P,\mathrm{ft},\Lambda}(n,m):=\mathrm{DHilb}^{P,\mathrm{ft}}(n,m)\cap \mathrm{DHilb}_{\Lambda}^P(n,m),$$
with
$\mathrm{DHilb}_{\Lambda}^P(m,n):=\{\mathscr{E}_{\infty}\subset J^{\infty}(m,n),\text{algebraic }P^{\mathcal{E}}=P,\mathrm{Ch}(\mathcal{E},\theta)\subset \Lambda\}.$

In this case, the Poincar\'e series \eqref{eqn: Rational Poincare} decomposed by the corresponding microlocal stratification, 
$$\mathrm{DHilb}_{\Lambda}^P(m,n)\simeq\coprod_{\alpha\subset \{1,\ldots,r\}}\mathrm{DHilb}_{\alpha}^P(m,n),\hspace{2mm}P_{\mathrm{DHilb}_{\Lambda}^P}=\sum_{\alpha\subset \{1,\ldots,r\}}P_{\alpha}(z),z\in\mathbb{C}.$$

For recovering interesting index theorems from the Euler-characteristic of the tangent complex, we equip the space of dependent variables and independent variables with additional geometric information. 
For example, assume the manifold of independent variables $(g,M)$ is Riemannian manifold of dimension $2n$, with a complexification $X$ and $(E,h)\to X$ is a hermitian bundle. 
    
For a point
$[\mathscr{E}]\in\mathrm{DHilb}^{P,\mathrm{ft},\mathrm{ell}}(2n,m),$ Spencer-complex has a simple form:
$$0\to \Gamma^{\infty}(\mathscr{E}_{}^{\ell})\to \Gamma(T_M^*\otimes \mathscr{E}_{}^{\ell})\to \cdots \to \Gamma(\wedge^{2n}T^*M\otimes \mathscr{E}_{}^{\ell})\to 0.$$

Ellipticity implies we may introduce
$$\Delta_k^L:=\nabla_k^*\circ \nabla_k+\nabla_{k-1}\circ \nabla_{k-1}^*,\hspace{2mm} \nabla_k:\mathscr{E}\otimes \wedge^k T_M^*\to \mathscr{E}\otimes \wedge^{k+1}T_M^*.$$

Following \cite{RaySinger1971},
for each $s\in\mathbb{C}$ define regularized determinant via $\zeta$-functions,
$$\zeta_k(s):=\sum_{\lambda_i\in \mathrm{Spec}\Delta_k\backslash \{0\}}\lambda_i^{-s},\hspace{2mm} at\hspace{1.5mm} \Re(s)> \dim M.$$

The second author gave a formula for the Ray-Singer analytic torsion of the Spencer complex of an elliptic finite-type PDE.
\begin{lemma}[\cite{LR}] There is a real number, 
$T_X(\mathscr{E}):=\prod_{k=0}^{\dim M}\big(\det \Delta_k^L\big)^{(-1)^{k+1}k/2},$
 which is an invariant of $X.$
\end{lemma}
We now extend this to the infinite-type setting and in the following section, give a formula for the Ray-Singer analytic torsion for a formally integrable PDE.

\section{Analytic torsion invariants of integrable PDEs}
\label{sec: Analytic Torsion of NPDEs}
In this section we compute Ray-Singer torsion for Spencer cohomology of formally integrable and involutive PDEs. We compute Spencer (hypercohomology) under degenerations and relate these two results to the computation of the BCOV invariant of a Calabi–Yau manifold \cite{eriksson2018bcov}. We first recall some basics facts, following closely the treatment in \cite{Yoshikawa2004,Yoshikawa2007,Yoshikawa2015},\cite{FangLuYoshikawa2008}, and \cite{Zucker1979} and the original works \cite{Quillen2,Quillen1}.
\subsection{Analytic torsion and Quillen metrics}
\label{ssec: AnalyticTorsion}
Let $(M,g)$ be a compact K\"ahler manifold of dimension $d$ with K\"ahler form $\omega$. Let 
$
\Delta_{p,q} = (\bar\partial + \bar\partial^*)^2$
be the Hodge-Kodaira Laplacian acting on $C^\infty$ $(p,q)$-forms on $M$, or equivalently $(0,q)$-forms on $M$ with values in $\Omega^p_M$, where $\Omega^1_M$ is the holomorphic cotangent bundle of $M$ and $
\Omega^p_M := \Lambda^p \Omega^1_M.$ 
Let $\sigma(p,q) \subset \mathbb{R}_{\ge 0}$ be the set of eigenvalues of $\Delta_{p,q}$. The spectral zeta function of $\Delta_{p,q}$ is defined as
\[
\zeta_{p,q}(s) := \sum_{\lambda \in \sigma(p,q) \setminus \{0\}} \lambda^{-s} \dim E(\lambda, p,q),
\]
where $E(\lambda, p,q)$ is the eigenspace of $\Delta_{p,q}$ corresponding to the eigenvalue $\lambda$. Then $\zeta_{p,q}(s)$ converges on the half-plane $\{ s \in \mathbb{C} \mid \Re s > \dim M \}$, extends to a meromorphic function on $\mathbb{C}$, and is holomorphic at $s=0$.

By Ray--Singer \cite{RaySinger1971}, the analytic torsion of $(M, \Omega^p_M)$ is the real number defined as
\[
\tau(M, \Omega^p_M) := \exp \Biggl\{- \sum_{q \ge 0} (-1)^q q \, \zeta'_{p,q}(0) \Biggr\}.
\] 
Obviously, $\tau(M, \Omega^p_M)$ depends not only on the complex structure of $M$ but also on the metric $g$. When emphasizing the dependence on the metric, we write $\tau(M, \Omega^p_M, g)$.
The following combination of analytic torsions was introduced in \cite{BCOV1994}.

\begin{definition}
The \emph{BCOV torsion} of $(M,g)$ is the real number defined as
\[
T_{\rm BCOV}(M,g) := \prod_{p \ge 0} \tau(M, \Omega^p_M)^{(-1)^p} 
= \exp \Biggl\{- \sum_{p,q \ge 0} (-1)^{p+q} pq \, \zeta'_{p,q}(0) \Biggr\}.
\]
If $\gamma$ is the K\"ahler form of $g$, we often write $T_{\rm BCOV}(M,\gamma)$ for $T_{\rm BCOV}(M,g)$.
\end{definition}

In general, $T_{\rm BCOV}(M,g)$ depends on the choice of K\"ahler metric $g$ and hence is not a holomorphic invariant of $M$. When $M$ is a Calabi--Yau threefold, it is possible to construct a holomorphic invariant of $M$ from $T_{\rm BCOV}(M,g)$ by multiplying a correction factor.

\subsection{Calabi--Yau threefolds and BCOV invariants}

A compact connected K\"ahler manifold $X$ is \emph{Calabi--Yau} if $h^{0,q}(X) = 0$ for $0 < q < \dim X$ and $K_X \cong \mathcal{O}_X$, where $K_X$ is the canonical line bundle of $X$. Our particular interest is the case where $X$ is a threefold.

Let $X$ be a Calabi--Yau threefold. Let $g = \sum_{i,j} g_{i\bar{j}} dz_i \otimes d\bar{z}_j$
be a K\"ahler metric on $X$, with $
\gamma = \gamma_g := \sqrt{-1} \sum_{i,j} g_{i\bar{j}} dz_i \wedge d\bar{z}_j$ the corresponding K\"ahler form. Define $\mathrm{Vol}(X,\gamma) := \frac{1}{(2\pi)^3} \int_X \frac{\gamma^3}{3!}.$

The covolume of $H^2(X,\mathbb{Z})_{\rm free} := H^2(X,\mathbb{Z}) / \mathrm{Torsion}$ with respect to $[\gamma]$ is defined as
\[
\mathrm{Vol}_{L^2}(H^2(X,\mathbb{Z}),[\gamma]) := \det \bigl( \langle e_i, e_j \rangle_{L^2,[\gamma]} \bigr)_{1 \le i,j \le b_2(X)},
\]
where $\{ e_1, \dots, e_{b_2(X)} \}$ is a basis of $H^2(X,\mathbb{Z})_{\rm free} = \mathrm{Im}\{ H^2(X,\mathbb{Z}) \to H^2(X,\mathbb{R}) \}$ over $\mathbb{Z}$ and $\langle \cdot, \cdot \rangle_{L^2,[\gamma]}$ is the inner product on $H^2(X,\mathbb{R})$ induced by integration of harmonic forms. Namely, if $H e_i$ denotes the harmonic representative of $e_i \in H^2(X,\mathbb{R})$ with respect to $\gamma$ and $\ast$ denotes the Hodge star operator, then
\begin{equation}
    \label{L2Metric1}
\langle e_i, e_j \rangle_{L^2,[\gamma]} := \frac{1}{(2\pi)^3} \int_X H e_i \wedge \ast H e_j.
\end{equation}

\begin{remark}
The covolume $\mathrm{Vol}_{L^2}(H^2(X,\mathbb{Z}),[\gamma])$ is then the volume of the real torus 

$H^2(X,\mathbb{R}) / H^2(X,\mathbb{Z})_{\rm free}$ with respect to the $L^2$-metric \eqref{L2Metric1}.
\end{remark}

As the correction term to the BCOV torsion $T_{\rm BCOV}(X,\gamma)$, we introduce a Bott--Chern term.

\begin{definition}
For a Calabi--Yau threefold $X$ equipped with a K\"ahler form $\gamma$, define
\[
A(X,\gamma) := \exp \Biggl[
- \frac{1}{12} \int_X \log \frac{ \sqrt{-1} \eta \wedge \bar{\eta} }{ \frac{\gamma^3}{3!} / \mathrm{Vol}(X,\gamma) \, \|\eta\|_{L^2}^2 } \, c_3(X,\gamma)
\Biggr],
\]
where $c_3(X,\gamma)$ is the top Chern form of $(X,\gamma)$, $\eta \in H^0(X,K_X) \setminus \{0\}$ is a nowhere vanishing canonical form on $X$, and
\[
\|\eta\|_{L^2} := \frac{1}{(2\pi)^3} \int_X \sqrt{-1} \, \eta \wedge \bar{\eta}.
\]
\end{definition}

Obviously, $A(X,\gamma)$ is independent of the choice of $\eta \in H^0(X,K_X) \setminus \{0\}$. Note that $A(X,\gamma) = 1$ if $\gamma$ is Ricci-flat.

\begin{definition}
The \emph{BCOV invariant} of $X$ is the real number defined as
\[
\tau_{\rm BCOV}(X) := \mathrm{Vol}(X,\gamma)^{-3+\chi(X)/12} \, \mathrm{Vol}_{L^2}(H^2(X,\mathbb{Z}),[\gamma])^{-1} \, T_{\rm BCOV}(X,\gamma) \, A(X,\gamma),
\]
where $\chi(X)$ denotes the topological Euler number of $X$.
\end{definition}

Applying the curvature formula for Quillen metrics, given a Calabi--Yau threefold $X$, $\tau_{\rm BCOV}(X)$ is independent of the choice of a K\"ahler form on $X$.

\subsection{Equivariant Analytic Torsion and Quillen Metrics}

Let $G$ be a compact Lie group. Let $\mathcal{G}^\vee$ denote the set of equivalence classes of complex irreducible representations of $G$. For $W \in \mathcal{G}^\vee$, the corresponding irreducible character is denoted by $\chi_W$.  

Let $V$ be a compact K\"ahler manifold, and assume that $G$ acts on $V$ as holomorphic automorphisms. Let $h_V$ be a $G$-invariant K\"ahler metric on $V$. Let $F$ be a $G$-equivariant holomorphic vector bundle on $V$ endowed with a $G$-invariant Hermitian metric $h_F$. We denote for simplicity these pairs by $
V := (V,h_V), F := (F,h_F).$
Then $A^{p,q}(V,F)$ is endowed with the $L^2$ metric $(\cdot,\cdot)_{L^2}$ with respect to $h_V$ and $h_F$, which is $G$-invariant under the standard $G$-action on $A^{p,q}(V,F)$.

Let 
\[
\Box^F_{p,q} := (\bar\partial_F + \bar\partial_F^*)^2
\] 
be the Hodge--Kodaira Laplacian acting on $A^{p,q}(V,F)$. As in Subsect. \ref{ssec: AnalyticTorsion}, denote by $\sigma(\Box^F_{p,q})$ the eigenvalues of $\Box^F_{p,q}$ and let $E(\lambda; \Box^F_{p,q})$ be the eigenspace corresponding to $\lambda \in \sigma(\Box^F_{p,q})$. Since $G$ preserves the metrics $h_V$ and $h_F$, $\Box^F_{p,q}$ commutes with the $G$-action on $A^{p,q}(V,F)$, and thus $G$ acts on $E(\lambda; \Box^F_{p,q})$.

For $g \in G$ and $\mathrm{Re}\, s \gg 0$, define the equivariant zeta function
\[
\zeta^G(g)(s) := \sum_{q \ge 0} (-1)^q q \sum_{\lambda \in \sigma(\Box^F_{p,q})\setminus\{0\}} \lambda^{-s} \, \mathrm{Tr}\, g|_{E(\lambda;\Box^F_{p,q})}.
\]
It is classical that $\zeta^G(g)(s)$ extends to a meromorphic function on $\mathbb{C}$ and is holomorphic at $s=0$. For $g \in G$, the \emph{equivariant analytic torsion} of $(V,F)$ is defined by
\[
\log \tau^G(V,F)(g) := - \zeta^{G\,\prime}(g)(0).
\]
Since $F$ is $G$-equivariant, $G$ acts on the cohomology $H(V,F) := \bigoplus_{q\in\mathbb{Z}} H^q(V,F)$
and preserves the $\mathbb{Z}$-grading. Consider the isotypic decomposition of the $\mathbb{Z}$-graded $G$-space:
\[
H(V,F) \simeq \bigoplus_{W \in \mathcal{G}^\vee} \mathrm{Hom}_G(W,H(V,F)) \otimes W.
\]
Since $H(V,F)$ is finite-dimensional, $\mathrm{Hom}_G(W,H(V,F)) = 0$ except for finitely many $W$. Define
$$
\lambda_W(F) := \bigotimes_{q \ge 0} \big(\det \mathrm{Hom}_G(W,H^q(V,F)) \otimes W\big)^{(-1)^q},$$
and consider the equivariant determinant of the cohomology of $F$ given by
$$
\lambda_G(F) := \bigotimes_{W \in \mathcal{G}^\vee} \lambda_W(F).
$$
A vector $\alpha = (\alpha_W)_{W \in \mathcal{G}^\vee} \in \lambda_G(F)$ is said to be \emph{admissible} if $\alpha_W \neq 0$ for all $W$ and $\alpha_W = 1_{\lambda_W(F)}$ for all but finitely many $W$. The set of admissible elements of $\lambda_G(F)$ is identified with $
\bigoplus_{W \in \mathcal{G}^\vee} \lambda_W(F)^\times,$ with $\lambda_W(F)^\times := \lambda_W(F) \setminus \{0\}.$
By Hodge theory, there is an isomorphism of $\mathbb{Z}$-graded $G$-spaces
\[
H(V,F) \simeq \bigoplus_{q \ge 0} E(0;\Box^F_{0,q}),
\] 
and the $G$-invariant metric induced from the $L^2$ metric is denoted $h_{H(V,F)}$. The isotypic decomposition is orthogonal with respect to $h_{H(V,F)}$. Let $\|\cdot\|_{L^2,\lambda_W(F)}$ be the Hermitian metric on $\lambda_W(F)$ induced from $h_{H(V,F)}$.  

For an admissible $\alpha = (\alpha_W)_{W \in \mathcal{G}^\vee} \in \bigoplus_{W \in \mathcal{G}^\vee} \lambda_W(F)^\times$, define
\[
\log \|\alpha\|^2_{Q,\lambda_G(F)}(g) := -\zeta^{G\,\prime}(g)(0) + \sum_{W \in \mathcal{G}^\vee} \frac{\chi_W(g)}{\dim W} \log \|\alpha_W\|^2_{L^2,\lambda_W(F)}.
\]
This function is called the \emph{equivariant Quillen metric} on $\lambda_G(F)$ with respect to $h_V$ and $h_F$.

\begin{notate}
Let $c_i(F,h_F) \in \bigoplus_{p\ge 0} A^{p,p}(V)$ denote the $i$-th Chern form of $(F,h_F)$. For $g \in G$, let $Td_g(F,h_F), \mathrm{ch}_g(F,h_F) \in \bigoplus_{p\ge 0} A^{p,p}(V^g)$ denote the equivariant Todd form and equivariant Chern character form of $(F,h_F)$, respectively.
\end{notate}
Let $M$ be a connected complex manifold with a holomorphic $G$-action and $B$ a connected complex manifold with trivial $G$-action. Let $\pi: M \to B$ be a flat proper morphism, $G$-equivariant, with a $G$-equivariant ample line bundle. Let $F$ be a $G$-equivariant holomorphic vector bundle on $M$, and let $h_M$ and $h_F$ be $G$-invariant metrics. Define fibers $
M_b := \pi^{-1}(b), \qquad F_b := F|_{M_b}, \quad b \in B,$ and consider the following isotypic decomposition of direct image sheaves. 
For $W \in \mathcal{G}^\vee$ and $U \subset B$ open, define
$$
\mathrm{Hom}_G(W,R^q\pi_* F)(U) := \{\varphi \in \mathrm{Hom}(W|_{\pi^{-1}(U)}, R^q\pi_* F|_{\pi^{-1}(U)}): g\varphi = \varphi g, \forall g\in G\}.
$$
Direct image sheaves $R^q \pi_* F$ decompose according to $$
R^q \pi_* F \simeq \bigoplus_{W \in \mathcal{G}^\vee} (R^q \pi_* F)_W, \qquad (R^q \pi_* F)_W := \mathrm{Hom}_G(W,R^q\pi_*F) \otimes W.$$
For each $W \in \mathcal{G}^\vee$, define
$
\lambda_W(F) := \bigotimes_{q\ge 0} \det(R^q \pi_* F)^{(-1)^q}_W,$ where $\lambda_G(F) := \bigotimes_{W\in\mathcal{G}^\vee} \lambda_W(F).$
For an admissible holomorphic section $\sigma = (\sigma_W)_{W \in \mathcal{G}^\vee}$ of $\lambda_G(F)$ over $U \subset B$, $\sigma$ is nowhere vanishing for all $W$ and $\sigma_W = 1_{\lambda_W(F)}$ for all but finitely many $W$.

\subsubsection{Smoothness}
Assume $\pi:M\to B$ is a $G$-equivariant proper holomorphic submersion. Let $TM/B$ denote the relative tangent bundle and $g \in G$. Then $\lambda_G(F)$ carries the equivariant Quillen metric $\|\cdot\|_{Q,\lambda_G(F)}(g)$ such that
\[
\|\cdot\|_{Q,\lambda_G(F)}(g)(b) := \|\cdot\|_{Q,\lambda_G(F_b)}(g), \quad b\in B.
\]

\begin{lemma}
Let $U \subset B$ be a small open set such that $\lambda_W(F_i)|_U \simeq \mathcal{O}_U$ for all $i$ and $W \in \mathcal{G}^\vee$. Let $\sigma = (\sigma_W)_{W\in\mathcal{G}^\vee}$ be an admissible holomorphic section of $\lambda_G(F)|_U$. Then
\[
\log \|\sigma\|^2_{Q,\lambda_G(F)}(g)
\]
is a smooth function on $U$.
\end{lemma}

\subsection{Analytic-Torsion for infinite-type PDEs}
For a given formally integrable PDE which is not necessarily of finite-type, the topological Spencer complex defines the following Ray-Singer torsion via the zeta-regularized Laplacians.
\begin{theorem}
\label{thm: RSingerSpencer}
    Given a formally integrable involutive system of PDEs, $\{Z_{\ell}\}_{\ell\geq 0}.$ Assume $(g,M)$ is a Riemannian manifold of dimension $2n$, and $E$ is given a Hermitian metric $h$ inducing metrics $\{h_{\ell}\}$ on each $Z_{\ell}$, compatible with $g.$ Then there is a family of local Laplacian operators,
    $$\Delta_k^{r-i,i}:=D_{k}^{(i-1)}\circ D_{k}^{(i-1),*}+D_k^{(i),*}\circ D_k^{(i)}:\mathcal{S}_{k}^{r-i,i}\to \mathcal{S}_{k}^{r-i,i},$$
    for each $k,$ where $D_{-1}:=\emptyset$ and $D_{2n+1}:=\emptyset.$ 
    We have
    $$\zeta_{k}^{r-i,i}(s):=\sum_{\lambda \in Spec(\Delta_k^{r-i,i})\backslash\{0\}}\lambda^{-s},\mathrm{Re}(s)>n,$$
    and the regularized $\zeta$-determinant.
    $$\mathrm{log}\big(\det\Delta_k^{r-i,i}\big):=-\frac{d}{ds}\zeta_k^{r-i,i}(s)\big|_{s=0},$$
    is well-defined since $\zeta_k(s)$ is holomorphic at $s=0.$ The Ray-Singer analytic torsion is 
    $$T_{Z}^{RS}(M)=\prod_{k=0}^{2n}\prod_{r=0}^{2n+\ell_0}\prod_{i=0}^{2n}\big(\det \Delta_k^{r-i,i})^{(-1)^{i+2}i/2}.$$
\end{theorem}
\begin{proof}
    The formal Laplacian is defined on each $\mathcal{S}_{k}^{r-i,i}:=\wedge^iT_M^*\otimes Z^{k-(r-i)},$ by taking the level-wise formal adjoint of the Spencer operator with respect to an $L^2$-norm defined by the choices of metrics:
    $$\big<\omega_1\otimes \theta_1,\omega_2\otimes\theta_2\big>_{L^2}:=\int_M(\omega_1\wedge \star\overline{\omega}_2)h_{k-i}(\theta_1,\theta_2),$$
    for $\omega\otimes\theta \in \wedge^iT_M^*\otimes Z^{k-i}.$
    Then, 
    $$\big<D_k^{(i)}(\varphi),\psi\big>_{L^2}=\big<\varphi,D_{k}^{(i),*}(\psi)\big>_{L^2}.$$
    Explicitly, we define $D_k^{(i),*}:=(-1)^{n(i+1)+1}\star D_k^{(i)}\star :\wedge^{i+1}T_M^*\otimes Z^{k-i-1}\to \wedge^iT_M^*\otimes Z^{k-i}.$
    Then,
    $$D_k^{(i-1),*}:\wedge^iT_M^*\otimes Z^{k-i}\to \wedge^{i-1}T_M^*\otimes Z^{k-i+1}.$$
    Recall that 
    $$\mathrm{Spec}(\Delta_k^{r-i,i})=\big\{\lambda\in\mathbb{R}_{\geq 0}|\exists s\neq 0,\hspace{2mm}with\hspace{1mm} \Delta_k^{r-i,i}(s)=\lambda s,s\in C^{\infty}(M,\mathcal{S}_k^{r-i,i})\big\}.$$
\end{proof}
The main example to connect with known results, is that Theorem \ref{thm: RSingerSpencer} permits us to recover Quillen's metric on the canonical line bundle associated with the zeta regularized determined of the $\overline{\partial}$-Laplacian actin on $(p,q)$-forms on a Calabi-Yau $n$-fold $X.$ Consequently, we may reproduce the so-called Quillen-BCOV metric introduced in \cite{FangLuYoshikawa2008}, on the so-called BCOV-line bundle i.e. the weighted product of determinants of Hodge bundles,
\begin{equation}
\label{eqn: BCOVLine}
\lambda_{\mathrm{BCOV}}(X):=\bigotimes_{p,q}\mathrm{det} H^q(X,\Omega_X^p)^{(-1)^{p+q}p}.
\end{equation}
Note that \eqref{eqn: BCOVLine} is independent of the choice of K\"ahler structure, and in \cite{eriksson2018bcov} a renormalized $L^2$-metric on \eqref{eqn: BCOVLine} was introduced, called the \emph{$L^2$-BCOV metric}; for an element $s\in \lambda_{BCOV}(X)$ define the following norm,
$$h_{L^2-BCOV}(s,s):=h_{L^2}(s,s)\cdot \prod_{k=1}^{2\dim_{\mathbb{C}}X}\mathrm{Vol}_{L^2}\big(H^k(X,\mathbb{Z}),\omega)^{(-1)^{k+1}k/2},$$
which is the product of $L^2$-metirc on $\lambda_{BCOV}$ provided by Hodge theory, and $\mathrm{Vol}_{L^2}$ is the covolume of the lattice $H^k(X,\mathbb{Z})_{nt}\subset H^k(X,\mathbb{R})$ with repscet to the $L^2$-product, where $H_{nt}$ is the maximal torsion free quotient. 
Following \cite[Proposition 5.6]{eriksson2018bcov}, this is invariant of the choice of K\"ahler metric and consequently we have a new invariant of a system of PDEs,
$$\tau_{BCOV}(Z):=h_{Q,BCOV}/h_{L^2-BCOV}.$$
We will compute the behaviour of Spencer complex of formally integrable PDEs under semistable degeneration, and give a curvature formula for the case of a holomorphic elliptic geometric structure $\mathcal{E}^1\subset J^1(\mathscr{D}_X(E)).$
\subsection{Degenerations}
Let $\pi:\mathcal{X}\to S$ be a proper holomorphic submersion of complex manifolds, with compact K\"ahler fibers $X_s:=\pi^{-1}(s),s\in S.$ Consider the PDE
$$\mathcal{E}^1:\big\{\nabla_{\mathcal{X}/S}(f)=0,f\in \Gamma_S(X,E),$$
for a bundle $E$. The Spencer complex $\mathrm{Sp}(\mathcal{E}^1)$ is a resolution of $\mathrm{Sol}(\mathcal{E}^1)$ i.e. of $\ker\nabla_{\mathcal{X}/S}.$
It is an elliptic complex of $\mathcal{O}_{\mathcal{X}}$-modules whose terms are locally free. Thus, $R\pi_*\mathrm{Sp}^{\bullet}(\mathcal{E}^1)$ is a perfect complex of $\mathcal{O}_S$-modules, with determinant line $\lambda_{Sp}(\mathcal{E}).$
It is endowed with a $L^2$-metric $h_{L^2}$ by choosing a K\"ahler metric on the fibers, and a hermitian metric $h_E$ on $E$, compatible with the connection $\nabla$. Then, each term $\mathrm{Sp}^k(\mathcal{E})|_{X_s}$ obtains a fiber-wise $L^2$-inner product via Hodge theory. The associated Quillen metric determined for each $s\in S$ by the Spencer-type Laplacians are given by Theorem \ref{thm: RSingerSpencer}.

\begin{lemma}
   Let $\pi:\mathcal{X}\to S$ be a proper holomorphic submersion of complex manifolds, with compact K\"ahler fibers $X_s:=\pi^{-1}(s),s\in S$ of dimension $n$. Consider a holomorphic vector bundle $E\to (\mathcal{X}/S)$ equipped with a relative flat connection $\nabla.$
Then the curvature of the Quillen metric on $\lambda_{Sp}(\mathcal{E})$ is given by
$$c\big(\lambda_{Sp}(\mathcal{E}),h_{Q}\big)=\int_{\mathcal{X}/S}\mathrm{ch}_2\big((-1)^{p+q}H_{\mathbb{C}}^{p,q}(\mathfrak{g}_f)\big)\mathrm{Td}(T_{\mathcal{X}/S})+\frac{1}{12}c_1(\omega_{\mathcal{X}/S})\cup c_1(T_{\mathcal{X}/S}).$$
\end{lemma}
\begin{proof}
We use that $\mathrm{ch}(Sp(\mathcal{E}))=\sum_k(-1)^k\mathrm{ch}(Sp^k(\mathcal{E}),$ and then the GRR formula, given by Lemma \ref{GRRSpencerDelta}. Then, use Lemma \ref{prop: TopIndex}.
\end{proof}

\section{Mixed-type $D$-algebras}
\label{sec: MixedType}
In this section we define mixed-type $D$-spaces and prove the index theorem for them.

 Let $Z\to X_{DR}$ be a $D_X$-scheme. Recall the diagram \eqref{setupdiag}. In particular, we have
 $$
T_N^*L \xleftarrow{(g_N)_d} N\times_M T_M^*X\xrightarrow{(g_N)_{\pi
}} T_M^*X.$$ 

\begin{definition}
\label{MixedDef}
    The $D_X$-scheme $Z$ is said to be \emph{mixed-type for} $\Omega:=(M,N,L,X,Y)$ if if there exist test families of solutions $u_T,v_T$ (take $T=\mathrm{Spec}(\mathbb{C})$, for simplicity)
such that: 
    \begin{enumerate}
        \item $\mathrm{Char}^{der}(Z,u)\cap T_L^*X\subset T_X^*X;$ 
        \item $T_N^*M\cap \bigg(T^*M\cap C_{T_M^*X}\big(\mathrm{Char}^{der}(Z,v)\big)\bigg)\subset T_M^*M,$
        \item In a neighbourhood of $N$, there exists closed cones $\lambda,\lambda'$ such that 
        $$\mathrm{Char}^{der}(Z,w)\cap T_M^*X\subset \lambda\cup \lambda',$$
        with $\lambda\cap \lambda'=M\times_X T_X^*X.$
    \end{enumerate}
\end{definition}
Condition (a) says that the behaviour is elliptic along $L$ near $N$, meaning no nonzero real conormals along $L$, while (b) means that characteristic directions meet the real conormal transversely, translating to the statement that $N$ is hyperbolic for $T_{Z,v}$ on $M$ (one also says $T_{Z,v}$ is hyperbolic along $N$ in $M$).

This definition is meaningful, and occurs in practice. Indeed, Let \(M = N \times \mathbb{R}\) be a real-analytic manifold and \(X = Y \times \mathbb{C}\) its complexification. 
Suppose \(P\) is a differential operator on \(X\) that is hyperbolic on \(M\) in the direction \(N \times \{0\}\), as in the case of the wave operator on spacetime. 
Let \(L = N \times i\mathbb{R}\). 
In many cases (eg, for the wave operator $P$), it becomes elliptic on \(L\). 
The passage from \(M\) to \(L\) allows one to study properties of \(P\) on \(M\) via its elliptic counterpart on \(L\).

We will discuss it further after our result is proven (see Example \ref{globhyp}).
\subsection{Mixed-type index}
\label{ssec: MixedTypeIndex}
The microlocal gluing condition of Definition \ref{MixedDef} (iii) allows us to prove the following.
\begin{theorem}
\label{theorem: Mixed index}
    Suppose $Z\to X_{DR}$ is mixed-type with respect to $\Omega$. Suppose $\gamma\subset T_NL$ is an open convex cone such that $\overline{\gamma}$ contains the zero-section $N$ and $g_{N\pi}^{-1}(\lambda)=g_{Nd}^{-1}(\gamma^{\circ a}).$
    Then, the hypercohomologies 
    $$H^i\big(R\Gamma(M;R\mathrm{Sol}_{X}(T_{Z,\phi},\mathscr{B}_{M,\lambda}))\big),$$
    are finite-dimensional, thus the Euler-characteristic is well-defined and coincides with the hyperbolic index,
    $$\mathrm{Ind}^{hyp}_{(M,\lambda)}(Z,\phi)=\int_{T^*X} \mu\mathrm{Eu}(T_{Z,\phi})\cup \mu \mathrm{Eu}\big(\Gamma_{\gamma}(\mathscr{B}_{NL})\big).$$
\end{theorem}
\begin{proof}
For simplicity, let us set $\mathcal{S}(Z,-):=R\mathrm{Sol}_X(T_{Z,\phi},-),$ so, for example $\mathcal{S}(Z,\mathscr{C}_M)=R\mathrm{Sol}_X(T_{Z,\phi},\mathscr{C}_M).$
The microlocal Euler classes are defined by
\[
\mu eu(\Gamma_\gamma \mathscr{B}_{NL})
\in H^{2n}_{S_\gamma}(T^*X;\pi^{-1}\omega_X),
\qquad
\mu eu(\mathscr{B}_{M,\lambda})
\in H^{2n}_{\mathrm{Char}^{\mathrm{der}}(\mathbb{T}_{A,\phi})\cap T_M^*X}(T^*X;\pi^{-1}\omega_X),
\]
where $S_\gamma=\overline{\gamma^{\circ a}}$ is the closure of the polar-cone of $\gamma$.

First, note that $g^{-1}\mathcal{S}(Z,\mathcal{O}_X)\simeq \mathcal{S}(Z,\mathscr{B}_L).$ Thus,
$$Rg_{Nd!}g_{N\pi}^{-1}\mathcal{S}(Z,\mathscr{C}_M)\simeq \mathcal{S}(Z,\mu_N\mathscr{B}_L)\otimes \omega_{L/N}.$$
Since 
$$Rg_{Nd!}g_{N\pi}^{-1}R\Gamma_{\lambda}\mathcal{S}(Z,\mathscr{C}_M)\simeq Rg_{Nd!}R\Gamma_{g_{N\pi}^{-1}(\lambda)}g_{N\pi}^{-1}\mathcal{S}(Z,\mathscr{C}_M)\simeq R\Gamma_{\gamma^{\circ a}}Rg_{Nd!}g_{N\pi}^{-1}\mathcal{S}(Z,\mathscr{C}_M),$$
we have that there is a canonical identification
$$Rg_{Nd!}g_{N\pi}^{-1}\mathcal{S}(Z,\Gamma_{\lambda}\mathscr{C}_M)\simeq \mathcal{S}(Z,\Gamma_{\gamma^{\circ a}}\mu_N\mathscr{B}_L)\otimes \omega_{L/N}.$$
Apply $R\pi_{N*}$ to this and use that $g_{Nd}$ is proper on the support of $\mathcal{S}(Z,\mathscr{C}_M),$ we get that 
$$R\pi_{N*}Rg_{Nd!}g_{N\pi}^{-1}\mathcal{S}(Z,\Gamma_{\lambda}\mathscr{C}_M)\simeq R\pi_*g_{N\pi}^{-1}\mathcal{S}(Z,\Gamma_{\lambda}\mathscr{C}_M)\simeq (R\pi_{M*}\mathcal{S}(Z,\Gamma_{\lambda}(\mathscr{C}_M))|_N.$$
Thus, there is a canonical isomorphism,
$$\mathcal{S}(Z,\mathscr{B}_{M\lambda})|_N\simeq \mathcal{S}(Z,\pi_{N*}\Gamma_{\gamma^{\circ a}}\mu_N\mathscr{B}_L)\otimes \omega_{L/N}.$$

In other words, 
We have 
\[
\begin{tikzcd}
    R\mathscr{H}om_{\D_X}(\M,\mathscr{B}_{M,\lambda})|_{N}\arrow[dr,"\rho_{MN}"]\arrow[rr,"\simeq"] & & R\mathscr{H}om_{\D_X}(\M,\Gamma_{\gamma}\mathscr{B}_{NL})\arrow[dl,"b_{\gamma,N}"]
    \\
    & R\mathscr{H}om_{\D_X}(\M,\widetilde{\mathscr{B}}_N)[d] &,
\end{tikzcd}
\]
where $\widetilde{\mathscr{B}}_N:=R\Gamma_N(\mathcal{O}_X)\otimes or_N[dim_M]\simeq H_N^n(\mathcal{O}_X)\otimes or_N.$
Then, there exists $\nu\in H_{g_{N\pi}^{-1}(\lambda)}(N\times_MT_M^*X),$ such that $g_{Nd,*}(\nu)=\mu\mathrm{Eu}(\Gamma_{\gamma}\mathscr{B}_{NL}).$ Since $g_{N\pi}^{-1}(\lambda)=g_{Nd}^{-1}(\gamma^{\circ a}),$ by hypothesis, then 
$$\int_{T^*X} \mu\mathrm{Eu}(Z,\phi)\cup \mu\mathrm{Eu}(\Gamma_{\gamma}\mathscr{B}_{NL})=\int_{T^*X}\mu\mathrm{Eu}(Z,\phi)\cup g_{Nd*}(\nu),$$
which by the projection formula and using the hypothesis gives 
$$\int_{N\times_M T_M^*X}g_{Nd}^*(\mu\mathrm{Eu}(Z,\phi))\cup \nu=\int_{T_M^*X}\mu\mathrm{Eu}(Z,\phi)\cup g_{N\pi*}(\nu)=\int_{T^*X}\mu\mathrm{Eu}(Z,\phi)\cup \mu\mathrm{Eu}(\mathscr{B}_{M\lambda}).$$
\end{proof}

\begin{remark}
    Near a clean intersection $N=M\cap L$, with coordinates $(x_N,x_M,x_L)$, the microlocal stalks can be computed as iterated cones:
\[
R\Gamma_\gamma \nu_N(\mathscr{B}_M)_{(x_N,\xi)}
\simeq
\mathrm{Cone}\Bigl(
R\Gamma_{\gamma\cap T_M^*X} (\mathscr{B}_M|_{x_N}) 
\to R\Gamma_{\gamma\cap T_L^*X} (\mathscr B_L|_{x_N}) 
\Bigr).
\]
The cones encode the boundary values morphisms
\[
R\mathscr{H}om(\mathbb{T}_{A,\phi},\mathscr{B}_{M,\lambda})|_N
\longrightarrow
R\mathscr{H}om(\mathbb{T}_{A,\phi},\Gamma_\gamma \mathscr{B}_{NL})
\longrightarrow
R\mathscr{H}om(\mathbb{T}_{A,\phi},\widetilde{\mathscr{B}}_N)[d].
\]
\(\widetilde{\mathscr{B}}_N := R\Gamma_N(\mathcal O_X) \otimes or_N[dim M]\) is the local cohomology along $N$.
Since $\mathrm{Char}^{\mathrm{der}}(\mathbb{T}_{A,\phi})\cap S_\gamma$ is compact in a neighbourhood of $N$, the stalks of $R\Gamma_\gamma \mathscr{B}_{NL}$ and $R\Gamma_\lambda \mathscr{B}_{M,\lambda}$ are finite-dimensional, ensuring that the microlocal Euler class pairing is well-defined.
\end{remark}
 The example of interest comes from wave operator on globally hyperbolic space-times.

 \begin{example}
     \label{globhyp}
     Let $M=N\times \mathbb{R},L=N\times i\mathbb{R}$ and $Y$ a complexification of $N$ with $X=Y\times\mathbb{C}.$ Fix coordinates $t+is\in\mathbb{C}$ and $(t+is\tau+i\sigma)\in T^*\mathbb{C}.$ Then $(x;i\eta)\in T_N^*Y,$ and $N$ is identified with $N\times \{0\}$ in $X.$ Let $P\in\mathcal{D}_X^{\leq m},$ be an operator of order $\leq m$ which is elliptic on $L$ and weakly-hyperbolic on $M$ in the $\pm dt,$ directions.
     Globally hyperbolic space-time $N\times\mathbb{R}$ possess wave-operators with this property. Let
     $$L^+:=N\times \{t+is|t=0,s>0\},$$
     and put
     $$\Lambda:=T_N^*Y\times \{(t+is;\tau+i\sigma)|s=0,\tau=0,\sigma \leq 0\}.$$
     In this case, the map $g_{Nd}$ is given by $(x,0;i\eta,i\sigma)\mapsto (x;\sigma),$ and $\gamma=L^+$. The $\D$-module is $\M_P:=\D_X/\D_X\cdot P,$ and
     the isomorphism $\mathcal{H}om_{\D_X}(\M_P,\mathscr{B}_{M,\lambda})|_{N}\simeq \mathcal{H}om_{\D_X}(\M_P,\mathscr{B}_{L^+})|_N$ holds and both admit boundary value morphisms to $\mathcal{E}xt_{\D_X}^1(\M_P,\widetilde{\mathscr{B}}_N)\simeq \widetilde{\mathscr{B}}_N/P\cdot \widetilde{\mathscr{B}}_N.$
 \end{example}

\section{Extension to Configuration Spaces}
\label{sec: ConfigurationSpaces}
\subsection{The cotangent diagram}
For a diagram $\mathcal{X}=\{X_i\}$ denoted $T^*\mathcal{X}(i):=T^*X_i$ and understand $T^*\mathcal{X}(i\rightarrow j)$ to be given by the diagram
\begin{equation}
\label{eqn: Cotangent Diagram}
\begin{tikzcd}
T^*X_j \arrow[d,"\pi_{j}"] & \arrow[l] X_j\times_{X_i}T^*X_i\arrow[d,"\pi_j'"] \arrow[r]& T^*X_i\arrow[d,"\pi_i"]
\\
X_j & \arrow[l,"id_{X_j}"] X_j\arrow[r,"\xi_g"] & X_i
\end{tikzcd}
\end{equation}
In particular $T^*\mathcal{X}$ is the collection of all such diagrams (\ref{eqn: Cotangent Diagram}).

The following makes this statement more precise.
\begin{lemma}
For $i\xrightarrow{g} j$ and $j\xrightarrow{h} k$ in $\mathrm{I}$ which are composeable, one has an induced morphism 
$\gamma:T^*\mathcal{X}(i\rightarrow j)\rightarrow T^*\mathcal{X}(j\rightarrow k)$ such that 
$$T^*\mathcal{X}=\underset{g\in\mathrm{Ar}(\mathrm{I}),\gamma}{\mathrm{colim}}\hspace{.2mm}T^*\mathcal{X}(i\xrightarrow{g} j),$$
where we form the homotopy-colimit over all diagrams with faces of the form (\ref{eqn: Cotangent Diagram}) formed by morphisms $\gamma.$
\end{lemma}
\begin{proof}
It suffices to notice that the arrows $g,h$ define morphisms of varieties $\xi_g:X_j\rightarrow X_i,\xi_h:X_k\rightarrow X_j$ which, along with the cotangent maps $T^*\xi_g,T^*\xi_h$ deteremine $\gamma$.
Indeed, they provide a commutative cube\footnote{We have omitted the obvious left-hand side of the cube.}
 \begin{equation*}
\label{eqn: Cotangent Cube}
\begin{tikzcd}[row sep=scriptsize, column sep=scriptsize]
&X_k\times_{X_j}T^*X_j\arrow[dl, "\gamma"] \arrow[rr, ""] \arrow[dd, "\pi_k'"] & & T^*X_j\arrow[dl, ""] \arrow[dd, "\pi_j"] \\
X_j\times_{X_i}T^*X_i\arrow[rr, ""] \arrow[dd, "\pi_j'"] & & T^*X_i\\
& X_k\arrow[dl, "\xi_h"] \arrow[rr, "\xi_h"] & & X_j \arrow[dl, "\xi_g"] \\
X_j \arrow[rr, "\xi_g"] & & X_i \arrow[from=uu, crossing over]\\
\end{tikzcd}
\end{equation*}
from which we identify $\gamma$ as
$T^*\xi_g':=\xi_h\times T^*\xi_h:X_k\times_{X_j}T^*X_j\rightarrow X_j\times_{X_i}T^*X_i.$

\end{proof}

Consider a morphism $\alpha:\mathcal{X}\rightarrow \mathcal{Y}$ of diagrams i.e. for all $g:i\rightarrow j$ we have that
\[
\begin{tikzcd}
X_j\arrow[d,"\alpha_j"]\arrow[r,"\xi_g^{\mathcal{X}}"] & X_i\arrow[d,"\alpha_i"]
\\
Y_j\arrow[r,"\xi_g^{\mathcal{Y}}"]& Y_i
\end{tikzcd}
\]
commutes. Put for simplicity the induced map
$$\alpha_{g:i\rightarrow j}:=\xi_g^{\mathcal{Y}}\circ \alpha_j:X_j\rightarrow Y_i.$$

The morphism $\alpha$ defines the commutative diagram at the level of cotangent spaces given by
\begin{equation}
\label{eqn: Cotangent space diagram}
\begin{tikzcd}
T^*X_j& \arrow[l] X_j\times_{X_i}T^*X_i\arrow[r] & T^*X_i
\\
X_j\times_{Y_j}T^*Y_j\arrow[u,"\alpha_{jd}"]
\arrow[d,"\alpha_{j\pi}"]
& \arrow[l]X_j\times_{Y_i}T^*Y_i\arrow[ul,"\Delta_{\alpha_gd}"]\arrow[u]\arrow[d]\arrow[dr,"\Delta_{\alpha_g\pi}"] \arrow[r] & X_i\times_{Y_i}T^*Y_i\arrow[u,"\alpha_{id}"]\arrow[d,"\alpha_{i\pi}"]
\\
T^*Y_j & \arrow[l] Y_j\times_{Y_i}T^*Y_i\arrow[r] & T^*Y_i
\end{tikzcd},
\end{equation}
where we have defined two natural morphisms
$$\Delta_{\alpha_g,d}:X_j\times_{Y_i}T^*Y_i\rightarrow T^*X_j,\hspace{2mm}\text{ and }\hspace{1mm} \Delta_{\alpha_g,\pi}:X_j\times_{Y_i}T^*Y_i\rightarrow T^*Y_i.$$

\subsection{Extending non-characteristic maps}
\label{ssec: Extending NC}
Consider as above a diagram of smooth $k$-schemes $\mathcal{X}:I\rightarrow Sch_k$ where $\mathcal{X}(i):=X_i\in Sch_k$ for each $i$ and for every arrow in $I$ we obtain the evident map of $k$-schemes.

\subsubsection{Non-Characteristic morphisms}
For $Z$ a closed sub-manifold $X$ denote the normal bundle by $T_ZX$, and the conormal bundle by $T_Z^*X.$ Denote by $T_X^*X$ the zero-section.

For a smooth map $f:X\rightarrow Y,$ there is a diagram
\begin{equation}
    \label{eqn: Usual CKK Diagram}
    \begin{tikzcd}
    & \arrow[dl,"f_d"] T^*Y\times_Y X\arrow[dr,"f_{\pi}", labels=below left]& 
    \\
    T^*X & & T^*Y
    \end{tikzcd}
\end{equation}
If $\Lambda\subset T^*Y$ is a closed $\mathbb{R}^+$-conic subset, one says that $f$ is \emph{non-characteristic} for $\Lambda,$ or that $\Lambda$ is non-characteristic for $f$ if one has that
\begin{equation}
    \label{eqn: Classical NC condition}
    f_{\pi}^{-1}\big(\Lambda\big)\cap T_X^*Y\subset X\times_Y T_Y^*Y.
\end{equation}
Notice that $T_X^*Y=\mathrm{ker}(f_d)=f_d^{-1}\big(T_X^*X\big).$

The appropriate extension to diagrams 
$\overline{X}:=\mathcal{X},$ whose domain $\mathrm{I}$ is an essentially small category for which every $g:i_1\rightarrow i_2,$ gives a map $\xi_g:X_{i_1}\rightarrow X_{i_2},$ is as follows. First we define $T^*\overline{X}$ to be the totality of all induced diagrams:
\[
\begin{tikzcd}
T^*X_{i_1}\arrow[d,"\pi_{i_1}"] & \arrow[l,"(\xi_g)_d"] X_{i_1}\times_{X_{i_2}}T^*X_{i_2}\arrow[d] \arrow[r,"(\xi_g)_{\pi}"] & T^*X_{i_2}\arrow[d,"\pi_{i_2}"]
\\
X_{i_1}\arrow[r,"\mathbf{1}_{X_{i_1}}"] & X_{i_1}\arrow[r,"\xi_g"] & X_{i_2}
\end{tikzcd}
\]
One puts $T_{X_{i_1}}^*X_{i_2}:=\mathrm{ker}(\xi_{g})_d=\big(\xi_g\big)_d^{-1}T_{X_{i_1}}^*X_{i_1}.$
\vspace{2mm}

A family $\overline{\Lambda}:=\lbrace \Lambda_i\rbrace_{i\in \mathrm{I}},$ of closed conic subsets $\Lambda_i\subset T^*X_i$ for every $i\in\mathrm{I}$, satisfying 
\begin{equation}
    \label{eqn: Conic for Cotangent Diagram}
    \Lambda_{i_1}\subset \big(\xi_g\big)_d\big(\xi_g\big)_{\pi}^{-1}\Lambda_{i_2},
\end{equation}
is said to be a closed conic subset for $T^*\overline{X}.$
If $\overline{\Lambda}=\lbrace\Lambda_i\rbrace_{i\in\mathrm{I}}$ is a closed conic, it is said to be \emph{transversal to the identity} if:
\vspace{2mm}

For all $g:i_1\rightarrow i_2$ in $\mathrm{I}$, the map $\xi_g:X_{i_1}\rightarrow X_{i_2}$ is non-characteristic with respect to $\Lambda_{i_2}.$ 
\vspace{1mm}

In other words, we have that 
$$(\xi_g)_d^{-1}T_{X_{i_1}}^*X_{i_1}\cap (\xi_g)_{\pi}^{-1}\Lambda_{i_2}\subset X_{i_1}\times_{X_{i_2}}T_{X_{i_2}}^*X_{i_2}.$$
Letting $\lambda:\overline{X}\rightarrow \overline{Y}$ be a morphism of diagrams, the corresponding non-characteristic condition is the obvious one.
\subsection{Characteristic varieties on configuration spaces}
\label{ssec: Char Vars on Config Space}
Consider $\mathcal{M}:=\{\mathcal{M}^{I}\}_{I\in fSet}\in \mathsf{Coh}(\mathcal{D}_{Ran_X})$ and define its characteristic variety to be the family
\begin{equation}
    \label{RanChar}\mathcal{C}har(\mathcal{M}):=\big\{\mathcal{C}har(\mathcal{M})^{I}:=\mathrm{Char}(\mathcal{M}^I)\subset T^*(X^I)\big\}_{I\in fSet}.
    \end{equation}
Then the main result of this subsection is the following 
\begin{theorem}
    \label{theorem: Factorization Characteristics}
Let $\mathcal{M}^{(Ran)}\in \mathrm{Coh}(\mathcal{D}_{Ran_X})$ and consider $\mathcal{C}har(\mathcal{M})\rightarrow Ran_X$, the levelwise co-Cartesian fibration which is moreover a closed $\mathbb{C}_{Ran_X}^{\times}$-conic subset of $T^*Ran_X.$ Then if $\mathcal{M}\in \mathsf{CohFAlg}(X),$ is of the form $\mathcal{M}^I:=\mathcal{M}^{\boxtimes I}$ for some $\mathcal{M},$ one has that $\mathcal{C}har(\mathcal{M})$ is a factorization $k$-scheme over $X.$
\end{theorem}
\begin{proof}
This follows since if $\mathcal{M}^{(I)}$ is of the form $\mathcal{M}^{\boxtimes I}$ for some $\mathcal{D}_X$-module $\mathcal{M},$ then we need only assume that $\mathcal{M}$ itself is coherent. Indeed, it follows from the classical fact that if $\mathcal{M},\EuScript{N}$ are coherent $\mathcal{D}_X$ and $\mathcal{D}_Y$-modules, then their external tensor product is a coherent $\mathcal{D}_{X\times Y}=\mathcal{D}_X\boxtimes\mathcal{D}_Y$-module.

To obtain factorization structure maps, let $\alpha:I\rightarrow J$ be a partition and $j(\alpha):U(\alpha)\hookrightarrow X^I$ the corresponding open embedding.
Make the identifications: for $I=\sqcup_{j\in J} I_j$ with $I_j=\alpha^{-1}(j)$ for a surjection $\alpha,$ we have
$$X^I\simeq \prod_{j\in J} X^{I_j},\hspace{3mm}\text{ and }\hspace{2mm} T^*(X^I)\simeq (T^*X)^{I}\simeq T^*\big(\prod_{j\in J}X^{I_j}\big)\simeq \prod_{j\in J}T^*(X^{I_j}).$$

For $j(\alpha):U(\alpha)\hookrightarrow X^I,$ consider the microlocal diagram (c.f \ref{eqn: Usual CKK Diagram}):
\[
\begin{tikzcd}
    & \arrow[dl,"j(\alpha)_d", labels=above left] U(\alpha)\times_{X^I} T^*(X^I)\arrow[dr,"j(\alpha)_{\pi}"] & 
    \\
    T^*U(\alpha) & & T^*(X^I)\simeq \prod_{j\in J}T^*X^{I_j}.
\end{tikzcd}
\]
We want to show there are isomorphisms
$$\mathcal{C}har(\mathcal{M})^{(I)}|_{U(\alpha)}\rightarrow \mathcal{C}har(\mathcal{M})^{(I_1)}\times\cdots\times \mathcal{C}har(\mathcal{M})^{I_n}|_{U(\alpha)},$$
where $n:=|J|.$ To this end, unpacking \eqref{RanChar} over each stratum, we observe that
\begin{eqnarray*}
\mathcal{C}har(\mathcal{M})^{(I)}|_{U(\alpha)}&\simeq& j(\alpha)_{\pi}^{-1}\mathcal{C}har(\mathcal{M}^{(I)})
\\
&\simeq& \mathcal{C}har\big(j(\alpha)^*\mathcal{M}^{(I)}\big)
\\
&\xrightarrow{\simeq}& \mathcal{C}har\bigg(j(\alpha)^*\big(\boxtimes_{j\in J}\mathcal{M}^{(I_j)}\big)\bigg),
\end{eqnarray*}
where we used the factorization equivalence for $\mathcal{M}.$ Since $j(\alpha)$ is an open embedding we have
$$\mathcal{C}har\bigg(j(\alpha)^!\big(\boxtimes_{j\in J}\mathcal{M}^{(I_j)}\big)\bigg)\subset j(\alpha)_dj(\alpha)_{\pi}^{-1}\mathcal{C}har(\big(\boxtimes_{j\in J}\mathcal{M}^{(I_j)}\big)),$$
and that this is in fact an equality.
By well-known properties\footnote{Specifically, that for any complex of sheaves $F$ on $M$ and $G$ on $N$, 
$\mu supp(F\boxtimes^L G)\subset \mu supp(F)\times \mu supp(G).$ If
$f:N\rightarrow M$ is proper on $supp(F),$ then $Rf_!F\xrightarrow{\simeq} Rf_*F$ and 
$\mu supp(Rf_*F)\subset f_{\pi}f_d^{-1}\mu supp(F).$
Moreover, if $f$ is a closed embedding, $\mu supp(Rf_*F)=f_{\pi}f_d^{-1}\mu supp(F).$} of characteristic varieties (see for example \cite{KS90}) we get
$$j(\alpha)_dj(\alpha)_{\pi}^{-1}\mathcal{C}har\big(\boxtimes_{j\in J}\mathcal{M}^{(I_j)}\big)=j(\alpha)_dj(\alpha)_{\pi}^{-1}\big(\mathcal{C}har(\mathcal{M}^{(I_1)})\times\cdots\times \mathcal{C}har(\mathcal{M}^{(I_n)}).$$
This coincides with the image of the restriction of
$\prod_{j\in J}\mathcal{C}har(\mathcal{M}^{(I_j)})|_{U(\alpha)},$
to $T^
*U(\alpha).$  
We have used the isomorphisms
$$\mathcal{C}har(\mathcal{M})^{(I\sqcup J)}|_{(X^{I\sqcup J})_{disj}}=\mathrm{Char}(\mathcal{M}^{(I\sqcup J)})|_{(X^{I\sqcup J})_{disj}}=\mathrm{Char}(\mathcal{M}^I)\times \mathrm{Char}(\mathcal{M}^J),$$
as subsets of $T^*(X^I)\times T^*(X^J)\simeq T^*(X^{I\sqcup J})$.
Moreover, $\Delta_{\alpha}:X^J\hookrightarrow X^I,$ gives $\Delta_{\alpha,\pi}^*:T^*X^I\times_{X^I}X^J\to T^*X^J,$ so $\Delta_{\alpha,\pi}^*\mathrm{Char}(\mathcal{M}^I)\subset \mathrm{Char}(\mathcal{M}^J).$
\end{proof}
\begin{lemma}
\label{proposition: Hyperfunction on Ran}
Hyperfunctions upgrade to the structure of a factorization $D$-module.
\end{lemma}
\begin{proof}
    Recall $R\Gamma_{M^I}(\mathcal{O}_{X^I})\simeq R\mathcal{H}om_{X^I}(\mathbb{C}_{M^I},\mathcal{O}_{X^I}),$ so apply $\Delta(\alpha)^!$ there is a canonical isomorphism in $D^b(X^J),$ via Grothendieck duality,
    $$\Delta^!(\alpha)R\mathcal{H}om_{X^I}(\mathbb{C}_M^I,\mathcal{O}_{X^I})\simeq R\mathcal{H}om_{X^J}(\Delta(\alpha)^{-1}\mathbb{C}_{M^I},\Delta(\alpha)^!\mathcal{O}_{X^I}),$$
    but since $\Delta(\alpha)^{-1}\mathbb{C}_{M^I}\simeq \mathbb{C}_{\Delta^{-1}M^I}\simeq\mathbb{C}_{M^J},$ we get
    $$R\mathcal{H}om_{X^J}(\mathbb{C}_{M^J},\Delta(\alpha)^!\mathcal{O}_{X^I})\simeq R\Gamma_{M^J}(\Delta^!\mathcal{O}_{X^I})\simeq R\Gamma_{M^J}(\omega_{X^J/X^I}[-c_{I/J}],$$
    where $c_{I/J}$ is the complex codimension of $X^J$ in $X^I.$
    Equivalently, 
    $$\Delta^!R\mathcal{H}om_{D_{X^I}}(D_{X^I}'\mathbb{C}_{M^I},\mathcal{O}_{X^I})\simeq R\mathcal{H}om_{D_X^J}(\Delta^{-1}D_{X^I}'\mathbb{C}_{M^I},\Delta^!\mathcal{O}_{X^I}),$$
    and since $\Delta^{-1}D_{X^I}'\mathbb{C}_{M^I}\simeq D_{X^J}'\Delta^{-1}\mathbb{C}_{M^I}\simeq D_{X^J}'\mathbb{C}_{M^J},$
    the right-hand side agrees. Consequently, we have for each $\alpha:I\twoheadrightarrow J,$ an equivalence
    $$\Delta(\alpha)^!\big(R\mathcal{H}om_{D_{X^I}}(D_{X^I}'\mathbb{C}_{M^I},\mathcal{O}_{X^I})\otimes \omega_{M^I}\big)\xrightarrow{\simeq}R\mathcal{H}om_{D_{X^J}}(D_{X^J}'\mathbb{C}_{M^J},\mathcal{O}_{X^J})\otimes \omega_{M^J}.$$
    This establishes that $\{\mathcal{B}_{M^I}\}$ define a $\D$-module on the Ran space. We now prove the factorization axioms. This follows from K\"unneth formula in local cohomology, restricted to the disjoint locus.
Explicitly, for $M^I\simeq M^{I_1}\times M^{I_2}\subset X^{I_1}\times X^I_2,$ 
one has 
$$H_{M^I}^{n\cdot |I|}(\mathcal{O}_{X^I})\simeq H_{M^{I_1}}^{n\cdot |I_1}(\mathcal{O}_{X^{I_1}})\otimes H_{M^{I_2}}^{n\cdot |I_2|}(\mathcal{O}_{X^{I_2}}),$$
from which the result follows.
\end{proof}
We apply Proposition \ref{proposition: Hyperfunction on Ran} to the case of causal manifolds, hence by \cite{JS16} also to Lorentzian manifolds. 
To this end, let $\Lambda \subset T^*X,$ be a closed convex analytic subset and let $\mathrm{Mod}_{\Lambda}(X)$, denote sheaves with support $\mathrm{supp}(F)\subset \Lambda\cap T^*X.$ 
Let $(M,\gamma)$ be causal with polar cone $\gamma^{\circ}$ and put 
$$\mathrm{Mod}_{[\gamma]}(\mathcal{D}_{\mathrm{Ran}_X}):=\big\{\{F_I\}_{I\in \mathrm{fSet}}|\mathrm{SS}(F_I)\subset (\gamma^{\circ})^I\big\}.$$
This is well-defined due to the behavior of singular supports under exterior tensor product \cite[Proposition 5.4.1]{KS90} i.e. $\mathrm{SS}(F\boxtimes G)\subset \mathrm{SS}(F)\times \mathrm{SS}(G).$ We will say that $F=\{F_I\}$ is \emph{causal} if it satisfies (\ref{eqn: Factorization property}) over \emph{space-like disjoints}. 
For each $x\in X$, put $\gamma_x^{\circ}:=\{\xi\in T_x^*X|\xi(v_x)\geq 0,\forall v_x\in \gamma_x\}$ and for $I\in \mathrm{fSet},$ put $\Gamma_I^{\circ}:=\prod_{i\in I}\gamma_{x_i}^{\circ}\subset T_{(x_i)}^*(X^I).$
Consider the cone-support condition
\begin{equation}
    \label{eqn: Cone-support}\mathrm{SS}(F_{X^I})\subset \Gamma_I^{\circ},\forall I\in \mathrm{fSet}.
\end{equation}
\begin{lemma}
    \label{proposition: Hyperfunction on Ran gamma version}
    The assignment $I\to \mathscr{B}_{\mathrm{Ran}_M}^{[\gamma]}(I):=\{u\mathcal{B}_{M^I}|SS(u)\subset \gamma^I\},$ is an object of $\mathrm{Mod}_{[\gamma]}(\D_{\mathrm{Ran}_X}),$ (cf. Notation \ref{SupportNotation}). In other words, $\mathscr{B}_{\mathrm{Ran}_M}^{[\gamma]}$ is a commutative factorization $\D$-algebra.
\end{lemma}

\begin{corollary}
\label{cor: Ran Elliptic Regularity}
Suppose $\M$ is elliptic. Then, for each $I\in \mathrm{fSet},$ there are homotopy equivalences,
$$\M^I\otimes _{\D_{X^I}}^L\mathcal{O}_{X^I}\xrightarrow{\simeq }\M^I\otimes_{\D_X^I}^L\mathscr{B}_{M^I},$$
compatible with diagonal embeddings.
\end{corollary}

\section{The $D$-geometric index}
\label{sec: GlobalAbstract}
In this section we answer affirmatively a conjecture posed in \cite{KSY23}, recalled as follows
\begin{conjecture}
     Let $X$ be a smooth scheme. The non-linear $\D$-Fredholm index can be expressed via Koszul duality,
 by the right-hand side of the GRR formula on loop stack of $X$. Moreover, the Euler-characteristic of its derived global section of its symbol complex may be expressed as a $K$-theoretic intersection
 number between $X$ and the cycle defined by the associated graded symbol morphism, inside the total space of the cotangent stack.
\end{conjecture}

Assume $Z\to X_{DR}$ is $D$-afp and let $\Lambda\subset q^*Z\times_X T^*X$ be a closed conic subset and let 
$j_{\Lambda}:(q^*Z\times_X T^*X)\backslash \Lambda\rightarrow q^*Z\times_X T^*X,$ be the inclusion of the complement.
Working with ind-coherent sheaves, we have 
$$j_{\Lambda}^*:\mathsf{IndCoh}(q^*Z\times_X T^*X)\rightarrow \mathsf{IndCoh}\big(q^*Z\times_X T^*X\backslash \Lambda).$$
\begin{definition}
    The category of micro-supported ind-sheaves on $Z$ is $\mathsf{IndCoh}_{\Lambda}(q^*Z\times_X T^*X):=\mathrm{hofib}(j_{\Lambda}^*).$
    \end{definition}
If $f:X'\to X$ is non-char morphisms of analytic spaces, we have commutativity of the following

\[
\begin{tikzcd} q^*Z \times_X T^*X \backslash \Lambda \arrow[r, "j_\Lambda"] \arrow[d, "f^*"] & q^*Z \times_X T^*X \arrow[d, "f^*"] \\ q'^*Z' \times_{X'} T^*X' \backslash f^{-1} \Lambda \arrow[r, "j_{f^{-1}\Lambda}"] & q'^*Z' \times_{X'} T^*X' \end{tikzcd}
\]
That is, by Beck-Chevalley we get $f^*\circ j_{\Lambda}^*\xrightarrow{\sim}j_{f^{-1}\Lambda}^*\circ f^*,$ equivalences of which follows from the diagram
\[
\begin{tikzcd} \mathsf{IndCoh}_\Lambda(q^*Z \times_X T^*X) \arrow[r] \arrow[d, "f^*"] & \mathsf{IndCoh}(q^*Z \times_X T^*X) \arrow[d, "f^*"] \arrow[r, "j_\Lambda^*"] & \mathsf{IndCoh}((q^*Z \times_X T^*X) \backslash \Lambda) \arrow[d, "f^*"] \\ \mathsf{IndCoh}_{f^{-1}\Lambda}(q'^*Z' \times_{X'} T^*X') \arrow[r] & \mathsf{IndCoh}(q'^*Z' \times_{X'} T^*X') \arrow[r, "j_{f^{-1}\Lambda}^*"] & \mathsf{IndCoh}((q'^*Z' \times_{X'} T^*X') \backslash f^{-1}\Lambda) \end{tikzcd}
\]

In particular, we find: 

\[
\begin{tikzcd} \mathsf{IndCoh}_\Lambda(q^*Z \times_X T^*X) \arrow[r, "a_*"] \arrow[d, hook] & \mathrm{Perf}_\Lambda(pt) \arrow[r, "\chi"] & \mathbb{Z} \\ \mathsf{IndCoh}(q^*Z \times_X T^*X) \arrow[ur, dashed, swap, "a_*"] \end{tikzcd}
\]

Finally,

\[
\begin{tikzcd}[column sep=huge] \varphi^*\mathbb{T}_{\EQ} \in \mathsf{IndCoh}_\Lambda(q^*Z \times_X T^*X) \arrow[r, "a_*"] \arrow[d, "\sigma"] & a_* \varphi^*\mathbb{T}_{\EQ} \in \mathrm{Perf}_\Lambda(pt) \arrow[d, "\sigma_*"] \arrow[r, "\chi"] & \mathbb{Z} \\ \sigma(\varphi^*\mathbb{T}_{\EQ}) \in \mathrm{Coh}_\Lambda(T^*X) \arrow[r, "a_*"] & a_* \sigma(\varphi^*\mathbb{T}_{\EQ}) \in \mathrm{Perf}_\Lambda(pt) \arrow[ur, "\chi" swap] \end{tikzcd}
\]

\begin{lemma}
\label{prop: BC-micro}
    Suppose $f:Y\to X$ is smooth, with $Z\to X_{DR}$ a bounded $D_X$-afp generalized PDE. Denote by $Z_Y$ the pullback generalized PDE. Then,
    $$\mathrm{Char}(Z_Y/Y_{DR};u_T)=(id_T\times f_d)(id_T\times f_{\pi}^{-1})(\mathrm{Char}(Z/X_{DR};u_T)).$$
\end{lemma}
\begin{proof}
    Recall $X$ is a smooth projective $\mathbb{C}$-scheme, and let $Z\to X_{DR}$ be $D_X$-afp. Let $T$ be any classical test scheme with $u_T:T\to Z$ a solution. Let $f:Y\to X$ be smooth with corresponding de Rham map $f_{DR}:Y_{DR}\to X_{DR}.$
Put, $ev_{u_T}:T\times X_{DR}\to Z,$ and $ev_{u_T}^Y:T\times Y_{DR}\to Z_Y,$ where $Z_Y:=f_{DR}^*Z.$ Note the Cartesian diagrams,
\[
\begin{tikzcd} Z_Y \ar[r]\ar[d] & Z \ar[d] \\ Y_{DR} \ar[r,"f_{DR}"] & X_{DR} \end{tikzcd}\hspace{2mm} \text{and }\hspace{1mm}
\begin{tikzcd} T\times Y_{DR} \ar[r,"\,id_T\times f_{DR}\,"] \ar[d,"\mathrm{ev}^Y_{u_T}"'] & T\times X_{DR} \ar[d,"\mathrm{ev}_{u_T}"] \\ Z_Y \ar[r] & Z \end{tikzcd}
\]
$$ev_{u_T}^{Y,!,\mathsf{IndCoh}}(\mathbb{T}_{Z_Y/Y_{DR}})\simeq (id_T\times f_{DR})_{\mathsf{IndCoh}}^!\big(ev_{u_T}^{!,\mathsf{IndCoh}}\mathbb{T}_{Z/X_{DR}}\big)\in \mathsf{IndCoh}(T\times Y_{DR})^{\omega}.$$
Moreover,
$$\mathbb{L}_{Z_Y/Y_{DR},u_T}\simeq \mathbb{D}_{T\times Y_{DR}}(ev_{u_T}^{Y,!}\mathbb{T}_{Z_Y/Y_{DR}})\simeq \mathbb{D}_{T\times Y_{DR}}\big((id_T\times f_{DR})_{\mathsf{IndCoh}}^!ev_{u_T}^!\mathbb{T}_{Z/X_{DR}}\big),$$
and as $f$ is smooth, $(id_T\times f_{DR})^!$ coincides with $(id_T\times f_{DR})^*[d],$ where $d$ is the relative dimension of $f:Y\to X.$

Since $X$ is smooth and projective, $\mathsf{IndCoh}(X_{DR})^{\omega}\simeq \mathsf{Perf}(X_{DR})\simeq \mathsf{Coh}(X_{DR})$ and $\mathsf{IndCoh}(T\times X_{DR})^{\omega}\simeq \mathsf{Coh}(T)\otimes \mathsf{IndCoh}(X_{DR})^{\omega}$. Thus any object decomposes as an exterior tensor product of a coherent complex on $T$ and a perfect complex on $X_{DR}$.
We have
$$ev_{u_T}^{Y,!}\mathbb{T}_{Z_Y/Y_{DR}}\simeq(id_T\times f_{DR})^!(E_T\boxtimes M_X)\simeq E_T\boxtimes f_{DR}^!(M_X).$$
Therefore,
$$\mathbb{L}_{Z_Y/Y_{DR},u_T}\simeq \mathbb{D}_T(E_T)\boxtimes \mathbb{D}_Y(f_{DR}^!M_X).$$
Consider the diagram.
\[
\begin{tikzcd} Y\times_X T^*X \ar[r,"f_\pi"] \ar[d,"{^t f'}"'] & T^*X \ar[d,"\pi_X"] \\ T^*Y \ar[r,"\pi_Y"] & Y \end{tikzcd}\]
Since $\mathrm{SS}(E_T\boxtimes f_{DR}^!M_X)\simeq \mathrm{supp}(E_T)\times \mathrm{Char}(f_{DR}^!M_X)\subset T\times_Y T^*Y,$ since $f$ is smooth, we have
$\mathrm{Char}(f_{DR}^!M_x)\simeq df\big(f_{\pi}^{-1}\mathrm{Char}(M_X)\big)\subset T^*Y,$ therefore,
$$\mathrm{Char}(Z_Y/Y_{DR};u_T)=\mathrm{SS}\big(ev_{u_T}^{Y,!}\mathbb{T}_{Z_Y/Y_{DR}}\big)\simeq \mathrm{supp}(E_T)\times df(f_{\pi}^{-1}\mathrm{Char}(M_X)\big),$$
thus is given by
$$(id_T\times df)\big(\mathrm{supp}(E_T)\times f_{\pi}^{-1}\mathrm{Char}(M_X)\big)\simeq (id_T\times df)(id_T\times f_{\pi}^{-1})\mathrm{Char}(Z/X_{DR};u_T).$$

Moreover, let $f:Y\to X$ be a smooth morphism of smooth projective schemes. Consider the base change
\[
\begin{tikzcd}
T \times Y_{DR} \arrow[r,"id_T\times f_{DR}"] \arrow[dr,swap,"ev_{u_T}"] & T \times X_{DR} \arrow[d,"ev_{u_T}"]\\
& Z
\end{tikzcd}
\]
Then the pullback
\[
(id_T\times f_{DR})^! \, ev_{u_T}^{!,\mathsf{IndCoh}} \mathbb{T}_{Z/X_{DR}} \in \mathsf{IndCoh}(T\times Y_{DR})^{\omega}
\]
satisfies the decomposition
\[
ev_{u_T}^{!,\mathsf{IndCoh}} \mathbb{T}_{f_{DR}^*Z/Y_{DR}}\simeq E_T\boxtimes f_{DR}^! M_X,
\]
so that the characteristic variety is preserved under smooth pullback:
\[
\mathrm{Char}(f_{DR}^*Z/Y_{DR};u_T)\simeq \mathrm{supp}(E_T)\times \mathrm{Char}(f_{DR}^! M_X)\simeq \mathrm{supp}(E_T)\times \mathrm{Char}(M_X).
\]
\end{proof}
From this result, we have 
$\mu \mathbb{T}_{Z/X_{DR},u_T}\in D_{coh}^b(T\times T^*X),$ which provides
a class in $K$-theory with support,
$$[\mu\mathbb{T}_{Z/X_{DR},u_T}]\in K_{\mathbf{char}}^0(T\times T^*X).$$
We require a result allowing us to base-change solutions.
\begin{lemma}\label{prop: BC-solution}
    Let $g:T'\to T$ be a morphism of test affine schemes. Then
    $$[\mu\mathbb{T}_{Z/X_{DR}},u_{T'}]=(g\times id_{T^*X})^*[\mu \mathbb{T}_{Z/X_{DR},u_T}],$$
    holds in $K_{\mathbf{char}(u_{T'})}^0(T'\times T^*X).$
\end{lemma}
\begin{proof}
Put $u_{T'}:=u_T\circ (g\times id_X),$ and consider  
\[
\begin{tikzcd}
T'\times X_{DR}\arrow[r,"g\times\mathrm{id}_{X_{DR}}"] \arrow[dr,"\mathrm{ev}_{u_T'}"] & T\times X_{DR}\arrow[d,"\mathrm{ev}_{u_T}"]\\ & Z. \end{tikzcd}\]
Set $\F_T:=\mathrm{ev}_{u_T}^{!,\mathrm{IC}}(\mathbb{T}_{Z/X_{DR}}),$ and similarly for $\F_{T'},$ for simplicity.
Then,
$$\mathrm{ev}_{u_T'}^{!,\mathrm{IC}}\simeq (g\times id_X)_{\mathrm{IC}}^!\circ \mathrm{ev}_{u_T}^{!,\mathrm{IC}}.$$
We have 
$$\mu_{\D}^V\big((g\times \mathrm{id})^{!,\mathrm{IC}}\F_T\big)\simeq (g\times \mathrm{id}_{T^*X})^*(\mu_{\D}^V\F_T),$$
and
$$\mu_{\D}^V(\F_{T'})\simeq \mu_{\D}^V\big(g\times \mathrm{id})^!\F_T\big)\simeq (g\times \mathrm{id}_{T^*X})^*\mu_{\D}^V(\F_T).$$
Therefore, 
$$\mu\mathbb{T}_{Z/X_{DR},u_{T'}}\simeq (g\times\mathrm{id}_{T^*X})^*\big(\mu\mathbb{T}_{Z/X_{DR},u_T}\big),$$in $D_{coh}^b(T'\times T^*X).$ Finally, note the support condition is preserved since $\mathbf{char}(u_{T'})=\mathrm{SS}(\F_{T'})\subseteq (g\times \mathrm{id}_{T^*X})^{-1}\mathrm{SS}(\F_{T})=(g\times \mathrm{id}_{T^*X})^{-1}\mathbf{char}(u_T),$ 
thus we have the well-defined map in $K$-theory, 
$(g\times \mathrm{id}_{T^*X})^*:K^0_{\mathbf{char}(u_T)}(T\times T^*X) \to K^0_{\mathbf{char}(u_{T'})}(T'\times T^*X).$
\end{proof}
From Lemma \ref{prop: BC-micro} and Lemma \ref{prop: BC-solution}, assume that 
$\F_{T}\simeq E_{T}\boxtimes M_X,$ for each test affine $T.$ Then, 
$$\kappa[u_T]=[\mu_{\D}(E_T\boxtimes M_X)]\simeq [E_T]\boxtimes [\mu M_X]\in K_{\mathrm{supp}\times \mathrm{SS}}^0(T\times T^*X),$$
where $[E_T]\in K_{\mathrm{supp}}^0(T)$ and $[\mu M_X]\in K_{\mathrm{SS}(M_X)}^0(T^*X).$
The Chern map with supports,
$$\mathrm{ch}_{\mathrm{supp}\times \mathrm{SS}}:K_{\mathrm{supp}\times \mathrm{SS}}^0(T\times T^*X)\to H_{\mathrm{supp}\times \mathrm{SS}}^{\bullet}(T\times T^*X,\mathbb{C}).$$
Standard properties of Chern maps yield
$$\mathrm{Ch}_{\mathrm{supp}\times \mathrm{SS}}(\kappa[u_T])=\mathrm{ch}_{\mathrm{supp}}([E_T])\boxtimes \mathrm{ch}_{\mathrm{SS}}([\mu M_X])\in H_{\mathrm{supp}}^{\bullet}(T)\otimes_{\mathbb{C}} H_{\mathrm{SS}(M_X)}^{\bullet}(T^*X,\mathbb{C}).$$
From Kashiwara-Schapira, 
$$\mu ch(M_X):=\mathrm{ch}_{\mathrm{SS}}([\mu M_X])\in H_{\mathrm{SS}}^{\bullet}(T^*X;\mathbb{C}),$$
is the microlocal Euler class of the coherent $D_X$-module $M_X.$
To summarize, we have
\begin{lemma}
\label{Chern of decomposition}
    $\mu \mathrm{Ch}_{\mathbf{char},u_T}\simeq \mathrm{ch}_{\mathrm{supp}}(E_T)\boxtimes \mu ch(M_X).$
\end{lemma}
Let $p_T:T\times T^*X\to T$ and $p_{T^*X}:T\times T^*X\to T^*X,$ be the projections. Then, we have the following polynomiality of a certain generating function related to the solution-wise index.
\begin{lemma}
Let $(X,\omega)$ be compact K\"ahler and let $t$ be an indeterminate variable, and fix non-zero complex numbers $\rho=(\rho_1,\ldots,\rho_{d_X}) \in (\mathbb{C}^{\times})^{\dim X}$. For each $T$-family of solutions $u_T$, of a $D_X$-afp $D$-scheme, put
$$Z_t(u_T):=\sum_{\ell=0}^{d_X}\rho_{\ell}t^{d-\ell}\bigg<\mathrm{ch}_{\mathrm{supp}}E_T\boxtimes \mu ch(M_X),p_T^*(1)\wedge p_{T^*X}^*(\pi^*\omega)^{\ell}\bigg>_{T\times T^*X}.$$
Then, $Z_t$ is a polynomial in $t$, with leading coefficient proportional to $\chi(T,E_T)\cdot \mu\mathrm{ind}(M_X).$
\end{lemma}
\begin{proof}
Given a solution $u:*\to R\mathrm{Sol}_X(Z),$ with compact support condition $\Lambda,$ 
$$\mu\mathrm{Ch}_{\Lambda}^{u}(Z):=\mathrm{ch}_{u^{-1}\Lambda}(\sigma_{\Lambda}\mathbb{T}_{Z,u})\cup \pi^*\mathrm{Td}_X(TX)\in H_{\Lambda}^*(T^*X;\mathbb{C}).$$
Decompose, 
$\mu\mathrm{Ch}_{\Lambda}^u(Z):=\sum_{p=0}^{2d_X}\mu\mathrm{ch}^{(2p)}(Z,u),$
note $\pi^*[\omega]^{\ell}$ pairs with degree $2p$ contributions:
$$\big<\mu\mathrm{Ch}_{\Lambda}^u(Z),\pi^*\omega^{\ell}\big>=\int_{T^*X}\mu\mathrm{ch}^{2(2d-\ell)}\wedge \pi^*\omega^{\ell}.$$
Thus,
$$Z_t(u):=\sum_{\ell=0}^{d_X}\rho_{\ell}t^{d-\ell}\int_{T^*X}\bigg(\sum_{p=0}^{2d}\mu\mathrm{ch}^{(2p)}\bigg)\wedge \pi^*\omega^{\ell}.$$
We see
$Z_t(u)=\sum_{k=0}^{d_X}a_k(u)\cdot t^{d-k},$
and in particular,
$$a_0(u)=\rho_0\int_{T^*X}\mu\mathrm{ch}^{(2d)}\wedge \pi^*\omega^d,\hspace{1mm} \mu\mathrm{ch}^{(2d)}=\big[\mathrm{ch}_{u^{-1}\Lambda}(\sigma_{\Lambda}\mathbb{T}_{,u})\cup \pi^*\mathrm{Td}_X(TX)\big]^{(2d)}.$$

Evaluating the pairing gives
$$\sum_{\ell}^d\rho_{\ell}t^{d-\ell}\bigg(\int_t\mathrm{ch}_{\mathrm{supp}}(E_T)\bigg)\cdot \bigg(\int_{T^*X}\mu ch(M_X)\wedge \pi^*\omega^{\ell}\bigg).$$
Via GRR, we obtain 
$$\sum_{\ell}^d\rho_{\ell}t^{d-\ell}\chi(T,E_T)\cdot\bigg(\int_{T^*X}\mu ch(M_X)\wedge \pi^*\omega^{\ell}\bigg).$$
Writing
$$Z_t(u_T)=\sum_{\ell=0}^{d}\rho_{\ell}t^{d-\ell}A_{\ell}=\rho_0t^d A_d +O(t^{d-1}),$$
where
$$A_{\ell}:=\bigg(\int_t \mathrm{ch}_{\mathrm{supp}}(E_T)\bigg)\cdot \bigg(\int_{T^*X}\mu ch(M_X)\wedge \pi^*\omega^{\ell}\bigg),$$
we see
$$Z_t(u_T)=\rho_0 t^d\chi(T,E_T)\cdot \mu\mathrm{ind}(M_X)+O(t^{d-1}),$$
since $A_0=\chi(T,E_T)\cdot \mu\mathrm{ind}(M_X).$
\end{proof}

\subsection{Outlook: Virtual index theory for PDEs}
\label{sec: Virtual}
We summarize the main results of the sequel. Consider 
$$HH_*\big(\mathsf{IndCoh}(Y)\big)\simeq \Gamma(\mathcal{L}Y,\omega_Y)\simeq \omega(\mathcal{L}Y),$$
be the Hochschild homology. Let $M\in \mathsf{IndCoh}(Y)$ be dualizable. The \emph{categorical character} is the Hochschild trace map, i.e. the image of the identity endomorphism under
$$\mathrm{Tr}:\mathrm{End}_{\mathsf{IndCoh}(Y))}(M)\to HH_*\big(\mathsf{IndCoh}(Y)\big).$$
In other words,
$$\mathbf{1}\xrightarrow{coev}M\otimes M^{\vee}\xrightarrow{\mathrm{Id}_M\otimes\mathrm{Id}_{M^{\vee}}} M\otimes M^{\vee}\xrightarrow{ev}\mathbf{1},$$
represents the trace of the identity as element of endomorphism of the unit. Identifying the latter with elements of $HH_*$, produces the character,
$$[M]:=\mathrm{Tr}\big(id_{M}\big)\in HH_*\big(\mathsf{IndCoh}(Y)\big).$$
This is functorial for a large class of maps e.g. ind-proper morphisms $f:Y\to Y',$ giving a push-forward,
$$f_*:HH_*(\mathsf{IndCoh}(Y))\to HH_*(\mathsf{IndCoh}(Y')).$$

We apply this to the global setting of not necessarily affine $X_{DR}$-spaces $Z\in \mathrm{PreStk}_{/X_{DR}}$ admitting deformation theory whose derived moduli stack of solutions $R\mathrm{Sol}(Z)$ is locally almost of finite-type (laft), by putting $Y:=X_{DR}$, viewed as a (derived) ind-inf scheme.

For a solution $\varphi$, linearization is given by pull-back $\varphi^*\mathbb{T}_{\EQ}$, which is a $\D$-module on $X.$ The $\mathcal{D}$-Fredholm index of $\EQ$ (along $\varphi$) is the Euler
characteristic of its linearization complex
$$\mathrm{Index}_{\varphi}(\EQ):=\chi\big(a_*(\varphi^*\mathbb{T}_{\EQ}^{\ell})\big),$$ 
where $a:X\rightarrow pt.$
Similarly, letting $R\mathrm{Sol}_X(\EQ)$ denote its pre-stack of solutions, the $\mathcal{D}$\emph{-geometric Euler-Poincaré} index (along $\varphi)$ is the corresponding index for the pre-stack $R\mathrm{Sol}_X(\EQ).$

\begin{theorem}
    Let $Z\to X_{DR}$ be $D$-afp and bounded. Then, for each solution $u:*\to \RS(Z),$ we have a categorical character $[T_{Z/X_{DR},u}]\in HH_*(D_X-\mathsf{Mod}(X)),$ and 
    $p_*[T_{Z/X_{DR}},u]\simeq \int_{\mathcal{L}p}\big[T_{Z/X,u}]\big]\in HH_*(D-\mathsf{Mod}(pt))\simeq HH_*(\mathsf{Vect}_k),$
    for $p:X\to pt,$ with $p_*$ the de Rham push-forward.
\end{theorem}
\begin{proof}{Sketch of proof}
Consider the commuting diagram:
\[
\begin{tikzcd} HH_*(\mathcal{D}_X\text{-}\mathrm{Mod}(X)) \arrow[r, "p_*"] \arrow[d, "\simeq"] & HH_*(\mathcal{D}\text{-}\mathrm{Mod}(\mathrm{pt})) \arrow[d, "\simeq"] \\ \omega(\mathcal{L}X) \arrow[r, "\mathcal{L}p_*"] & \omega(\mathcal{L}\mathrm{pt}), \end{tikzcd}
\]
Since $Z$ is $D_X$-afp and bounded, $Z\to Z^{\leq n}$ relative to $X_{DR}$ induces for each solution $u$ a solution of the truncation $u_n:*\to \RS(Z^{\leq n}),$ and via the sequence
$$j_n^*\mathbb{L}_{Z^{\leq n}/X_{DR}}\to \mathbb{L}_{Z/X_{DR}}\to \mathbb{L}_{Z/Z^{\leq n}},$$
via pull-back along $u$, since $u=j_n\circ u_n,$ for each $n$, we have, by dualizing an equality of categorical characters, 
$$p_*[T_{Z/X_{DR},u}]=p_*\big[u_n^*T_{Z^{\leq n}/X_{DR}}\big]+p_*[u^T_{Z/Z^{\leq n}}].$$
Moreover, since
$$p_*[T_{Z/X_{DR},u}]=p_*(Tr(id_{T_{Z/X_{DR},u}}))\simeq Tr(p_*id_{T_{Z/X_{DR}},u})\simeq Tr(id_{p_*T_{Z/X_{DR},u}}),$$
the result follows.
\end{proof}

\bibliographystyle{amsalpha}
\bibliography{Bibliography}

\providecommand{\bysame}{\leavevmode\hbox to3em{\hrulefill}\thinspace}
\providecommand{\MR}{\relax\ifhmode\unskip\space\fi MR }
\providecommand{\MRhref}[2]{%
  \href{http://www.ams.org/mathscinet-getitem?mr=#1}{#2}
}
\providecommand{\href}[2]{#2}
\begin{thebibliography}{BCOV94}

\bibitem[AS63]{AtiyahSinger1963}
Michael~F. Atiyah and Isadore~M. Singer, \emph{The index of elliptic operators on compact manifolds}, Bulletin of the American Mathematical Society \textbf{69} (1963), 422--433.

\bibitem[BCOV94]{BCOV1994}
Michael Bershadsky, Sergio Cecotti, Hirosi Ooguri, and Cumrun Vafa, \emph{Kodaira-spencer theory of gravity and exact results for quantum string amplitudes}, Communications in Mathematical Physics \textbf{165} (1994), no.~2, 311--427.

\bibitem[BD04]{BD04}
A.~Beilinson and V.~Drinfeld, \emph{Chiral algebras}, Colloquium Publications, vol.~51, American Mathematical Society, 2004.

\bibitem[BF86a]{BismutFreed1986EllipticFamiliesI}
Jean-Michel Bismut and Daniel~S. Freed, \emph{The analysis of elliptic families. {I}. {M}etrics and connections on determinant bundles}, Communications in Mathematical Physics \textbf{106} (1986), no.~1, 159--176.

\bibitem[BF86b]{BismutFreed1986EllipticFamiliesII}
\bysame, \emph{The analysis of elliptic families. {II}. {D}irac operators, eta invariants, and the holonomy theorem}, Communications in Mathematical Physics \textbf{107} (1986), no.~1, 103--163.

\bibitem[Bis86]{Bismut1986FamiliesDirac}
Jean-Michel Bismut, \emph{The {A}tiyah--{S}inger index theorem for families of {D}irac operators: Two heat equation proofs}, Inventiones Mathematicae \textbf{83} (1986), no.~1, 91--151.

\bibitem[BMS23]{BMS23}
Marco Benini, Giorgio Musante, and Alexander Schenkel, \emph{Green hyperbolic complexes on lorentzian manifolds}, Comm. Math. Phys. \textbf{403} (2023), no.~2, 699--744.

\bibitem[BN13]{BenZviNadler2013SecondaryTraces}
David Ben‑Zvi and David Nadler, \emph{Secondary traces}, arXiv preprint arXiv:1305.7177, 2013.

\bibitem[BR21]{BR}
Ugo Bruzzo and Vladimir~N. Rubtsov, \emph{Koszul complexes and spectral sequences associated with lie algebroids}, S{\~a}o Paulo Journal of Mathematical Sciences \textbf{15} (2021), no.~2, 495--504.

\bibitem[Bru17]{B}
Ugo Bruzzo, \emph{Lie algebroid cohomology as a derived functor}, Journal of Algebra \textbf{483} (2017), 245--261.

\bibitem[BZN13]{BenZviNadler2013NonlinearTraces}
David Ben-Zvi and David Nadler, \emph{Nonlinear traces}, arXiv preprint arXiv:1305.7175, 2013.

\bibitem[CG16a]{CG16}
Kevin Costello and Owen Gwilliam, \emph{Factorization algebras in quantum field theory, vol. 1}, Cambridge University Press, 2016.

\bibitem[CG16b]{CG16b}
\bysame, \emph{Factorization algebras in quantum field theory, vol. 2}, 2016.

\bibitem[C\u03]{Caldararu2003I}
Andrei C\u{a}ld\u{a}raru, \emph{The mukai pairing, i: The hochschild structure}, arXiv preprint arXiv:math/0308079v2, 2003.

\bibitem[CW10]{CaldararuWillerton2010}
Andrei C\u{a}ld\u{a}raru and Simon Willerton, \emph{The mukai pairing. i. a categorical approach}, New York Journal of Mathematics \textbf{16} (2010), 61--98.

\bibitem[EF08]{Felder}
Markus Engeli and Giovanni Felder, \emph{A riemann-roch-hirzebruch formula for traces of differential operators}, Annales scientifiques de l'École Normale Supérieure, Série 4 \textbf{41} (2008), no.~4, 623--655.

\bibitem[EiMM21]{eriksson2018bcov}
Dennis Eriksson, Gerard~Freixas i~Montplet, and Christophe Mourougane, \emph{Bcov invariants of {Calabi}–{Yau} manifolds and degenerations of {Hodge} structures}, no.~3, 379--454, 76 pages, Preprint arXiv:1809.05581.

\bibitem[ES19]{ES}
H{\'e}l{\`e}ne Esnault and Claude Sabbah, \emph{Good lattices of algebraic connections (with an appendix by claude sabbah)}, Documenta Mathematica \textbf{24} (2019), 271--301.

\bibitem[Fed96]{Fedosov1996DQ}
Boris Fedosov, \emph{Deformation quantization and index theory}, Mathematical Topics, vol.~9, Akademie Verlag, Berlin, 1996.

\bibitem[FG12]{FG12}
John Francis and Dennis Gaitsgory, \emph{Chiral {Koszul} duality}, Selecta Mathematica \textbf{18} (2012), no.~1, 27--87.

\bibitem[FLY08]{FangLuYoshikawa2008}
H.~Fang, Z.~Lu, and K.-I. Yoshikawa, \emph{Analytic torsion for {C}alabi--{Y}au threefolds}, J. Differential Geom. \textbf{80} (2008), no.~2, 175--259.

\bibitem[GLL17]{BVIndex}
Ryan Grady, Qin Li, and Si~Li, \emph{Batalin--vilkovisky quantization and the algebraic index}, Advances in Mathematics \textbf{317} (2017), 575--639.

\bibitem[Gol67]{Goldschmidt1967}
Hubert Goldschmidt, \emph{Integrability criteria for systems of nonlinear partial differential equations}, Journal of Differential Geometry \textbf{1} (1967), no.~3--4, 269--307.

\bibitem[GS91]{GilletSoule1991}
Henri Gillet and Christophe Soul{\'e}, \emph{Analytic torsion and the arithmetic {T}odd genus}, Topology \textbf{30} (1991), no.~1, 21--54, With an appendix by D. Zagier.

\bibitem[GWW25]{GuiWangWilliams2025Jouanolou}
Zhengping Gui, Minghao Wang, and Brian~R. Williams, \emph{Higher-dimensional chiral algebras in the jouanolou model}, 68 pages; submitted 30 Oct 2025.

\bibitem[JS16]{JS16}
B.~Jubin and P.~Schapira, \emph{Sheaves and $\mathcal{D}$-modules on lorentzian manifolds}, Letters in Mathematical Physics \textbf{106} (2016), 607--648.

\bibitem[Kas95]{Kashiwara1970}
Masaki Kashiwara, \emph{Algebraic study of systems of partial differential equations}, Mémoires de la Société Mathématique de France, no.~63, Société Mathématique de France, 1995, Master's thesis, University of Tokyo, 1970. Translated by A. D'Agnolo and J.-P. Schneiders.

\bibitem[KLV86]{KLV}
Joseph Krasil'shchik, Valentin Lychagin, and Alexander~M. Vinogradov, \emph{Geometry of jet spaces and nonlinear partial differential equations}, Advanced Studies in Contemporary Mathematics, vol.~1, Gordon and Breach Science Publishers, New York, 1986.

\bibitem[KO68]{KatzOda1968}
Nicholas~M. Katz and Tadao Oda, \emph{On the differentiation of {De Rham} cohomology classes with respect to parameters}, Journal of Mathematics of Kyoto University \textbf{8} (1968), no.~2, 199--213.

\bibitem[Kon95]{Kontsevich1994HAMS}
Maxim Kontsevich, \emph{Homological algebra of mirror symmetry}, Proceedings of the International Congress of Mathematicians, Vol. 1, 2 (Zürich), Birkhäuser, Basel, 1995, pp.~120--139.

\bibitem[KR25]{KR}
J.~Kryczka and V.~Rubtsov, \emph{Spencer hypercohomologies in the geometric theory of partial differential equations}, Lobachevskii Journal of Mathematics \textbf{46} (2025), no.~11, 5795–--5813.

\bibitem[Kry26]{Kry2026}
Jacob Kryczka, \emph{Conservation laws and cohomological intersections}, Journal of Geometry and Physics (2026), no.~105763.

\bibitem[KS90]{KS90}
Masaki Kashiwara and Pierre Schapira, \emph{Sheaves on manifolds}, Grundlehren der Mathematischen Wissenschaften, vol. 292, Springer-Verlag, Berlin, 1990, With a chapter in French by Christian Houzel.

\bibitem[KS12]{KSDQMods}
\bysame, \emph{Deformation quantization modules}, Ast\'erisque, no. 345, Société Mathématique de France, 2012. \MR{3012169}

\bibitem[KS14]{KashiwaraSchapira2014}
\bysame, \emph{Microlocal euler classes and hochschild homology}, Journal of the Institute of Mathematics of Jussieu \textbf{13} (2014), no.~3, 487--516.

\bibitem[KS25]{KSh1}
Jacob Kryczka and Artan Sheshmani, \emph{The $\mathcal{D}$-geometric hilbert scheme -- part i: Involutivity and stability}, arXiv preprint (2025).

\bibitem[KSY23]{KSY23}
Jacob Kryczka, Artan Sheshmani, and Shing-Tung Yau, \emph{Derived moduli spaces of nonlinear pdes {I}: Singular propagations}, arXiv preprint \textbf{arXiv:2312.05226} (2023), Submitted and updated version available on author's webpage: \url{https://sites.google.com/view/jacobkryczka/research}.

\bibitem[KSY24]{KSY2}
\bysame, \emph{Derived moduli spaces of nonlinear pdes {II}: Variational tricomplex and {BV} formalism}, arXiv preprint \textbf{arXiv:2406.16825} (2024).

\bibitem[Li23]{Li2023VertexQME}
Si~Li, \emph{Vertex algebras and quantum master equation}, Journal of Differential Geometry \textbf{123} (2023), no.~3, 461--521.

\bibitem[LR89]{LR}
Valentin~V. Lychagin and Vladimir~N. Rubtsov, \emph{Topological indices of spencer complexes associated with geometric structures}, Mathematical Notes of the Academy of Sciences of the USSR \textbf{45} (1989), no.~4, 305--312.

\bibitem[Lyc85]{L}
Valentin~V. Lychagin, \emph{Differential operators on fiber bundles}, Russian Mathematical Surveys \textbf{40} (1985), no.~2, 187--188, Translated from Uspekhi Matematicheskikh Nauk.

\bibitem[NT95]{NestTsygan1995AlgebraicIndex}
Ryszard Nest and Boris Tsygan, \emph{Algebraic index theorem}, Communications in Mathematical Physics \textbf{172} (1995), no.~2, 223--262.

\bibitem[NT96]{FormalAnalytic}
\bysame, \emph{Formal versus analytic index theorems}, International Mathematics Research Notices (1996), no.~11, 557--564.

\bibitem[Pau14]{Pau14}
Fr{\'e}d{\'e}ric Paugam, \emph{Towards the mathematics of quantum field theory}, vol.~59, Springer Science \& Business Media, 2014.

\bibitem[Qui64]{Q}
Daniel~G. Quillen, \emph{Formal properties of overdetermined systems of linear partial differential equations}, Phd thesis, Harvard University, 1964.

\bibitem[Qui85a]{Quillen2}
Daniel Quillen, \emph{Determinants of {C}auchy--{R}iemann operators over a {R}iemann surface}, Functional Analysis and Its Applications \textbf{19} (1985), no.~1, 31--34.

\bibitem[Qui85b]{Quillen1}
\bysame, \emph{Superconnections and the {C}hern character}, Topology \textbf{24} (1985), no.~1, 89--95.

\bibitem[Ram08]{Ramadoss2008}
Ajay~C. Ramadoss, \emph{The relative riemann-roch theorem from hochschild homology}, New York Journal of Mathematics \textbf{14} (2008), 643--717.

\bibitem[RS73]{RaySinger1971}
Daniel~B. Ray and Isadore~M. Singer, \emph{{$R$-torsion and the Laplacian on Riemannian manifolds}}, Annals of Mathematics \textbf{98} (1973), no.~1, 154--177.

\bibitem[Sab11]{Sa}
Claude Sabbah, \emph{Introduction to the theory of $\mathcal{D}$-modules}, Nankai University, 2011, Lecture Notes.

\bibitem[Sat73]{Sato}
Mikio Sato, \emph{Hyperfunctions and pseudodifferential equations}, Hyperfunctions and Pseudo-differential Equations, Lecture Notes in Mathematics, vol. 287, Springer, 1973, pp.~265--529.

\bibitem[Sch73]{Schmid1973}
Wilfried Schmid, \emph{Variation of hodge structure: the singularities of the period mapping}, Inventiones Mathematicae \textbf{22} (1973), 211--319.

\bibitem[Sch94]{Sch}
Jean-Pierre Schneiders, \emph{An introduction to $\mathcal{D}$-modules}, Bulletin de la Soci{\'e}t{\'e} Royale des Sciences de Li{\`e}ge \textbf{63} (1994), 223--295.

\bibitem[Spe64]{Sp}
Donald~C. Spencer, \emph{De rham theorems and neumann decompositions associated with linear partial differential equations}, Annales de l'Institut Fourier \textbf{14} (1964), no.~1, 223--241.

\bibitem[Spe69]{Sp2}
\bysame, \emph{Overdetermined systems of linear partial differential equations}, Bulletin of the American Mathematical Society \textbf{75} (1969), no.~2, 179--239.

\bibitem[SS94a]{SS94a}
Pierre Schapira and Jean-Pierre Schneiders, \emph{Elliptic pairs i. relative finiteness and duality}, Ast{\'e}risque \textbf{224} (1994).

\bibitem[SS94b]{SS94b}
\bysame, \emph{Elliptic pairs ii. euler class and relative index theorem}, Ast{\'e}risque \textbf{224} (1994).

\bibitem[Ste77]{Steenbrink1977}
Joseph H.~M. Steenbrink, \emph{Mixed hodge structure on the vanishing cohomology}, Real and Complex Singularities (Proc. Ninth Nordic Summer School/NAVF Sympos. Math., Oslo, 1976), Sijthoff and Noordhoff, Alphen aan den Rijn, 1977, pp.~525--563.

\bibitem[Ste76]{Steenbrink1975}
Joseph Steenbrink, \emph{Limits of hodge structures}, Inventiones Mathematicae \textbf{31} (1975/76), no.~3, 229--257.

\bibitem[TV20]{TVGRR}
Bertrand To{\"e}n and Gabriele Vezzosi, \emph{Algebraic foliations and derived geometry {II}: the {Grothendieck–Riemann–Roch} theorem}, arXiv preprint.

\bibitem[Vin01]{Vin01}
A.~Vinogradov, \emph{Cohomological analysis of partial differential equations and secondary calculus}, American Mathematical Society, 2001.

\bibitem[Wit82]{Witten1982Constraints}
Edward Witten, \emph{Constraints on supersymmetry breaking}, Nuclear Physics B \textbf{202} (1982), no.~2, 253--316.

\bibitem[Wit88]{WittenLoop}
\bysame, \emph{The index of the {D}irac operator in loop space}, Elliptic Curves and Modular Forms in Algebraic Topology (Alan~L. Carey, Richard~M. Hain, and Michael~K. Murray, eds.), Lecture Notes in Mathematics, vol. 1326, Springer, Berlin, 1988, pp.~161--181.

\bibitem[Yos04]{Yoshikawa2004}
Ken-Ichi Yoshikawa, \emph{K3 surfaces with involution, equivariant analytic torsion, and automorphic forms on the moduli space}, Inventiones Mathematicae \textbf{156} (2004), no.~1, 53--117.

\bibitem[Yos07]{Yoshikawa2007}
\bysame, \emph{On the singularity of quillen metrics}, Mathematische Annalen \textbf{337} (2007), no.~1, 61--89.

\bibitem[Yos15]{Yoshikawa2015}
\bysame, \emph{Degenerations of calabi-yau threefolds and bcov invariants}, International Journal of Mathematics \textbf{26} (2015), no.~4, 1540010, 33.

\bibitem[ZL94]{LZ}
Leonid Zil'bergleit and Valentin~V. Lychagin, \emph{Spencer cohomology of differential equations}, Global Analysis, World Scientific, 1994, pp.~215--252.

\bibitem[Zuc79]{Zucker1979}
Steven Zucker, \emph{Hodge theory with degenerating coefficients. l² cohomology in the poincaré metric}, Annals of Mathematics, Second Series \textbf{109} (1979), no.~3, 415--476.

\end{thebibliography}
\end{document}